\pdfoutput=1

\documentclass[12pt]{amsart}
\usepackage[usenames,dvipsnames]{xcolor}
\usepackage{mathtools,amssymb,amsthm,tikz,tikz-cd,multirow}
\usepackage{fullpage}
\usepackage{etoolbox}
\usepackage{extarrows}
\usepackage{enumitem}
\usepackage{makecell}
\usepackage{tensor}

\usepackage[font=footnotesize,labelfont=bf, width=.9\textwidth]{caption}

\usepackage[protrusion=true,expansion=true]{microtype}
\newcommand\GL{\operatorname{\mathbf{GL}}}
\newcommand\Gn{\widetilde{G}}
\newcommand\Bn{\widetilde{B}}
\newcommand\Tn{\widetilde{T}}

\newcommand\Hom{\operatorname{Hom}}
\newcommand\z{\mathbf{z}}
\newcommand\x{\mathbf{x}}

\newcommand\A{\mathcal{A}}
\newcommand\W{\mathcal{W}}
\newcommand\C{\mathbb{C}}

\newcommand\iso{\cong}

\DeclarePairedDelimiter{\ceiln}{\lceil}{\rceil_{\!n}}
\DeclarePairedDelimiter{\floorn}{\lfloor}{\rfloor_{\!n}}
\DeclarePairedDelimiter{\ceil}{\lceil}{\rceil}
\DeclarePairedDelimiter{\floor}{\lfloor}{\rfloor}

\usetikzlibrary{decorations,decorations.pathreplacing,positioning,calc,tikzmark}

\usepackage[protrusion=true,expansion=true]{microtype}
\usepackage{hyperref}

\hypersetup{
    pdftitle={ },
    pdfauthor={Ben Brubaker, Valentin Buciumas, Daniel Bump, Henrik P. A. Gustafsson},
    linktocpage=true,           % color toc pages, not text
    colorlinks=true,            % false: boxed links; true: colored links
    linkcolor=blue!75!black,    % color of internal links
    citecolor=green!50!black,   % color of links to bibliography
    filecolor=magenta,          % color of file links
    urlcolor=blue!75!black      % color of external links
}

\let\oldtocsubsection=\tocsubsection
\let\oldtocsubsubsection=\tocsubsubsection
\renewcommand{\tocsubsection}[2]{\hspace{1.8em}\oldtocsubsection{#1}{#2}}
\renewcommand{\tocsubsubsection}[2]{\hspace{3.6em}\oldtocsubsubsection{#1}{#2}}

\begin{document}

\title{Metaplectic Iwahori Whittaker functions and supersymmetric lattice models}
\author{Ben Brubaker}
\address{School of Mathematics, University of Minnesota, Minneapolis, MN 55455}
\email{brubaker@math.umn.edu}
\author{Valentin Buciumas}
\address{Department of Mathematics, 
	Pohang University of Science and Technology, Pohang, Republic of Korea 37673}
\email{valentin.buciumas@gmail.com}
\author{Daniel Bump}
\address{Department of Mathematics, Stanford University, Stanford, CA 94305-2125}
\email{bump@math.stanford.edu}
\author{Henrik P. A. Gustafsson}
\address{Department of Mathematics and Mathematical Statistics, Umeå University, SE-901 87 Umeå, Sweden}
\email{henrik.gustafsson@umu.se}

\pgfdeclarelayer{foreground}
\pgfsetlayers{main,foreground}

\tikzset{
  circles/.style={
    spin, draw=none,
    append after command={
      \pgfextra
      \pgfinterruptpath
      \begin{pgfonlayer}{foreground}
        \begin{scope}[reset cm]
        \foreach \options [count=\i] in {#1}{
          \expandafter\draw\expandafter[\options] let \p1=($(\tikzlastnode.east)-(\tikzlastnode.west)$), \n1 = {veclen(\p1)} in (\tikzlastnode) circle ({\n1/2-(\i-1)*1.45*\the\pgflinewidth}); 
        }
      \end{scope}
      \end{pgfonlayer}
      \endpgfinterruptpath
      \endpgfextra
    }
  }
}

\newcommand\eps{0.075}
\tikzstyle{scaled}=[scale=.8, every node/.append style={scale=0.8}]
\tikzstyle{Rscaled}=[scale=0.9/sqrt(2), every node/.append style={scale=0.8}]
\tikzstyle{spin}=[circle, draw, fill=white, minimum size=18pt, inner sep=1pt, outer sep=0pt, text=black]
\tikzstyle{spin-path}=[every node/.append style={spin}]
\tikzstyle{colored}=[very thick, solid]
\tikzstyle{colored-path}=[very thick, every node/.append style={colored, spin}]
\tikzstyle{scolored}=[very thick, densely dotted]
\tikzstyle{scolored-path}=[densely dotted, very thick, every node/.append style={scolored, spin}]
\tikzstyle{both}=[circles={{very thick, solid}, {very thick, densely dotted}}]
\tikzstyle{both-path}=[draw=none, preaction={draw, white, line width=1.28mm}, postaction={transform canvas={shift={(\eps/2,0)}}, draw, colored-path, preaction={transform canvas={shift={(-\eps/2,0)}}, draw, black, scolored-path}}, every node/.append style={both}]

\tikzstyle{halo}=[circle, fill=white, inner sep=0pt]
\tikzstyle{halo-rect}=[rectangle, fill=white, inner sep=0pt]

\newcommand{\state}[3][auto]{%
  \tikz[baseline=(text.base), every node/.style = {}, inner sep=0pt, minimum size=0pt, rectangle, color=black]{%, external/export=false]{ 
\begin{scope}[every node/.style={font=\ifstrequal{#1}{auto}{\ifblank{#2}{\normalsize}{\ifblank{#3}{\normalsize}{\scriptsize}}}{#1}}]%
  \node(text) {$#3$};%
  \node[left=0cm of text] {$#2$};%
  \draw[densely dotted, thick, transform canvas={shift={(0,-.05)}}] (text.south west) -- (text.south east);%
\end{scope}%
}%
}

\definecolor{oceanmist}{RGB}{50,120,120}
\definecolor{green}{RGB}{0,180,0}
\definecolor{brown}{RGB}{120,0,120}

\colorlet{sgreen}{green!50!black}
\colorlet{sblue}{blue!50!black}
\colorlet{sred}{red!50!black}

\newtheorem{theorem}{Theorem}[section]
\newtheorem{lemma}[theorem]{Lemma}
\newtheorem{proposition}[theorem]{Proposition}
\newtheorem{assumption}[theorem]{Assumption}
\newtheorem{corollary}[theorem]{Corollary}

\newtheorem{maintheorem}{Theorem}
\renewcommand{\themaintheorem}{\Alph{maintheorem}}

\theoremstyle{definition}
\newtheorem{remark}[theorem]{Remark}
\newtheorem{definition}[theorem]{Definition}

\newtheorem{example}{Example}            % for separate 1,2,3,... numbering

\numberwithin{equation}{section}

\definecolor{green}{RGB}{0,180,0}

\newcommand{\icegrid}[2]{
    % #1: N columns (+1)
    % #2: rows

    % Vertical lines
    \foreach \i in {0,...,#1}{
        \draw (2*\i+1,0) -- (2*\i+1,2*#2) node[spin] {$+$};
    }

    % Horizontal lines
    % \foreach \i in {1,...,#2}{
    %     \pgfmathtruncatemacro\row{int(#2-\i+1)}
    %     \draw (0,2*\i-1) node[spin] {$+$} -- (2*#1+2, 2*\i-1);
    % }

    % Underlying grid of plus spins
    \foreach \i in {0,...,#1}{
        \foreach \j in {1,...,#2}{
            \node[spin] at (2*\i+1, 2*\j-2) {$+$};
            %\node[spin] at (2*\i+2, 2*\j-1) {$+$};
        }
    }
}

\tikzstyle{state}=[circle, draw, fill=white, minimum size=12pt, inner sep=0pt, outer sep=0pt]
\tikzstyle{selection}=[thick, color=gray, fill=lightgray!30, rounded corners, text=black]
\tikzstyle{selection-line}=[selection, rounded corners=0pt, fill=none, shorten <=2pt, shorten >=2pt]

\subjclass[2020]{Primary: 22E50; Secondary: 82B23, 16T25, 05E05, 17B37, 11F70}

\begin{abstract}
In this paper we compute new values of Iwahori Whittaker functions on $n$-fold metaplectic covers $\widetilde{G}$ of $\mathbf{G}(F)$ with $\mathbf{G}$ a split reductive group over a non-archimedean local field $F$. 
For \emph{every} Iwahori Whittaker function $\phi$, and for \emph{every} $g\in\widetilde{G}$, we evaluate $\phi(g)$ by recurrence relations over the Weyl group using novel ``vector Demazure-Whittaker operators.'' The general formula and strategy of proof are inspired by ideas appearing in the theory of integrable systems.

Specializing to the case of $\mathbf{G} = \mathbf{GL}_r$, we construct a solvable lattice model of a new type associated with the quantum affine super group $U_q(\widehat{\mathfrak{gl}}(r|n))$ and prove that its partition function equals $\phi(g)$.
To prove this equality we match the recurrence relations on the lattice model side (obtained from the Yang-Baxter equation) to the recurrence relations for $\phi(g)$ derived by using the representation theory of $\widetilde{G}$.
Remarkably, there is a bijection between the boundary data specifying the partition function and the data determining all values of the Whittaker functions.
\end{abstract}

\maketitle

\tableofcontents

\section{Introduction}
This work is the culmination of a series of papers~\cite{mice,BBB,BBBGVertex,BBBF,BBBGIwahori} concerning a remarkable parallel between the theory of non-archimedean Whittaker functions and certain solvable lattice models.
Our results are naturally divided into three parts represented by the subsequent sections of the paper: 

\begin{itemize}[label=$\boldsymbol{\cdot}$, leftmargin=*]
  \item The computation of \emph{all} values of metaplectic Iwahori Whittaker functions for split reductive groups via a new representation theoretic construct: vector Demazure-Whittaker operators.
  \item The construction of novel \emph{supersymmetric} solvable lattice models connected to the quantum group $U_q\bigl(\hat{\mathfrak{gl}}(r|n)\bigr)$ and associated recursive relations for computing its partition functions.
  \item Connecting the above results in the special case of a metaplectic $n$-cover of $\GL_r$ through a natural bijection between the data that specify the values for the metaplectic Iwahori Whittaker function and the boundary data for the lattice model.
\end{itemize}
Together, these results show that quantum superalgebras and their associated lattice 
models are able to capture deep aspects of metaplectic group representations at the 
Iwahori level. Moreover, the lattice models naturally suggest the Cartan type independent 
explicit formulas we use to represent all values of metaplectic Iwahori Whittaker 
functions for any split reductive group, leading to the most general known version of a Casselman-Shalika formula.

The Casselman-Shalika formula~\cite{CasselmanShalika} is a remarkable lodestar in local 
representation theory. It states that Whittaker functions for {\it infinite dimensional} 
unramified principal series representations $V$ of an algebraic group $G$ over a non-archimedean local 
field $F$ are expressible in terms of characters of {\it finite dimensional} 
representations on the complex Langlands dual group $G^\vee(\mathbb{C})$. In the local 
representation theory, $V$ is indexed by continuous parameters 
$\mathbf{z} = (z_2, \ldots, z_r)$ in $(\mathbb{C}^\times)^r$ and the Whittaker function is 
evaluated at torus elements corresponding to dominant integral coweights $\lambda$. On 
the dual group side, these roles are reversed.  Thus $\lambda$ is a weight for 
$G^\vee(\mathbb{C})$ indexing the corresponding highest-weight representation, whose 
character is evaluated at the continuous parameters $\mathbf{z}$.

Originally conjectured by Langlands, and shown for general linear groups by Shintani, the Casselman-Shalika formula (as it is commonly known) was proved in \cite{CasselmanShalika} for any unramified reductive group in the case of {\it spherical} Whittaker functions. Recall that ``spherical'' here refers to the unique-up-to-scalar vectors in the principal series representation fixed by a maximal compact subgroup. The formula is both a tool in the local theory of automorphic forms~\cite{Shalika} and a signpost for the Langlands program in expressing local representation theory in terms of invariants of the complex dual group~\cite{FGKV}. 

There are two natural modes of generalization the Casselman-Shalika formula. First, 
we may 
shrink the compact subgroup whose fixed vectors are used in the Whittaker model. Second, 
we may enlarge the family of groups we consider. Both directions have consequences for 
the formula as a tool {\it and} as a signpost. In what follows, we briefly survey prior 
results to motivate these two directions and to contextualize the contributions of the 
present work. 

In the first direction, we may consider Whittaker functions fixed by the Iwahori 
subgroup. 
From the representation theoretic point of view it is better to work at the Iwahori 
level, because the category of smooth representations generated by their Iwahori-fixed 
vectors forms a block in the category of smooth representations. 
In particular, it is closed under passage to subquotients.
This fails for representations generated by their $K$-fixed vectors.
The Casselman-Shalika formula for the spherical case was proved in \cite{CasselmanSpherical, CasselmanShalika} making essential use of Iwahori-fixed vectors.

A study of Iwahori Whittaker functions was initiated by Reeder~\cite{ReederCompositio} and results in~\cite{BBL} highlighted the role of Demazure-like divided difference operators. 
Definitive results completely characterizing the functions were given by the authors in \cite{BBBGIwahori}. 
In particular, these latter results were critically inspired by lattice models where such divided difference operators arise naturally. 
So the Iwahori Whittaker functions serve as a tool in obtaining spherical formulas for use in automorphic forms, but also as a signpost to a deeper theory in revealing Demazure-like operators that connect them to specializations of non-symmetric Macdonald polynomials. 

In the second direction of generalization, we may consider certain central extensions of our algebraic groups --- the so-called metaplectic groups. 
The resulting principal series no longer have unique Whittaker models, as shown in Kazhdan and Patterson's ground breaking work \cite{KazhdanPatterson}. 
Motivated by the theory of Weyl group multiple Dirichlet series, explicit formulas for one special choice of spherical Whittaker function were given in various guises in~\cite{ChintaGunnellsPuskas, McNamara:Duke, ChintaOffen, wmd5book, McNamaraCS}. 
Each is analogous to a way of expressing the classical case of highest weight characters as (i) generating functions on combinatorial data, (ii) via the Weyl character formula, and (iii) via the Demazure character formula. These involve the arithmetic of the cover, introducing Gauss sums into the formulas. 

In~\cite{BBB}, Whittaker models for a particular cover of the general linear group are expressed in terms of solvable lattice models, and this was extended to a more general class of covers in~\cite{FrechetteGeneral}.
This can be used alongside standard tools in the theory of lattice models to recover the results cited in the last paragraph, as well as prove new ones (see for example~\cite{StatementB}). 
As we will detail in subsequent sections, solvable lattice models produce generating functions that naturally give rise to Demazure-like recursions, so we see two of the three aforementioned guises, namely (i) and (iii), for metaplectic Whittaker functions incorporated in the lattice. 

As a tool, formulas for special choices of spherical metaplectic Whittaker functions have been crucial in a host of number theoretic applications in the context of multiple Dirichlet series~\cite{bbf:Eisenstein:crystals,wmd5book,McNamara:Duke}, and one may expect additional such applications from the added flexibility of varying the choice of Whittaker model and cover. 

As a signpost, we may seek an algebraic structure in this context that allows one to express the spherical metaplectic Whittaker function. Versions of ``dual groups'' that capture the representation theory of metaplectic groups appear in~\cite{SavinCrelle, McNamaraEdinburgh, Weissman}, where root datum for a dual group is proposed in terms of the degree of the cover and the bilinear form defining the metaplectic two-cocycle, as well as in \cite{gaitsgory, GaitsgoryLysenko,Finkelberg-Lysenko} from a geometric perspective. 
In particular, it is a conjecture of Gaitsgory and Lurie~\cite{gaitsgory}, proved in different settings in \cite{campbell:dhillon:raskin, BuciumasPatnaik}, that Lusztig's quantum group at a root of unity acts as a dual group controlling the Whittaker functions on the metaplectic group. 
The results in~\cite{BBB} suggest, via the connection to solvable lattice models, that for Cartan type A, a different quantum group 
(a twist of $U_q(\hat{\mathfrak{gl}}(1|n))$ whose coalgebra structure is appropriately modified to introduce Gauss sums) 
captures the algebraic difficulties of writing down expressions for spherical, metaplectic Whittaker functions. 
This observation has been made more precise for
$U_q(\hat{\mathfrak{gl}}(n))$ in work by Buciumas (unpublished) and 
Gao-Gurevich-Karasiewicz~\cite{GGK2}, where it is shown how the quantum affine 
group acts on the 
Whittaker models (and implicitly on the category of representations of Lusztig's quantum 
group at a root of unity). We note that the representation theoretic meaning of the full 
quantum group $U_q(\hat{\mathfrak{gl}}(1|n))$ is not yet understood in this setting.

The present paper treats {\it both} generalizations simultaneously, the Iwahori and metaplectic, to provide a new Casselman-Shalika formula in the broadest possible context. Accomplishing this requires much more than just aggregating the previously studied pieces. Our starting point is a solvable lattice model which will be described in \S\ref{sub:intro:latticemodels}. We formulate a solvable lattice model to express {\it all} values of {\it all} Iwahori fixed vectors under {\it all} metaplectic Whittaker functionals on a cover of the general linear group. Demazure-like operators naturally arise from this construction, but are now {\it vector-valued} to account for each possible Whittaker model. Those same formulas may be used to give type-free recursions for metaplectic, Iwahori Whittaker functions independent of (but critically inspired by) lattice models, which we then prove for all Cartan types using solely the representation theory of the metaplectic group (Theorem~\ref{thm:dem_recurse}). Those vector-valued Demazure operators give rise to a representation of the metaplectic Iwahori Hecke algebra on the space of polynomial functions on the dual torus (see Theorem~\ref{thm:hecke}). 
This representation is isomorphic to the one appearing in~\cite{SSV,SSV2}; the difference is that our representation arises naturally from the non-archimedean representation theory, while theirs follows solely from considering the affine Hecke algebra and the double affine Hecke algebra. This representation is also an important ingredient in $p$-adic versions~\cite{BuciumasPatnaik} of the fundamental local equivalence of Gaitsgory-Lurie~\cite{gaitsgory}. 

So what have we gained in doing this? 
This project is the culmination of a long quest to determine whether lattice models are flexible enough to encode the most general versions of Casselman-Shalika-type formulas despite the numerous constraints imposed by solvability. We present a positive answer which pushes beyond the limits of what has appeared previously in the lattice model literature, while using every bit of available data from the lattice models to express all values of Whittaker functions (see Figure~\ref{fig:translation}).
Solvability allows us to produce formulas for metaplectic Iwahori Whittaker functions as generating functions, as well as to use old and new techniques in the theory of integrable systems to study them.
For example, classic techniques can now be used to prove Cauchy identities and branching rules. Novel techniques developed in~\cite{AggarwalBorodinWheelerColored} can be used to study structure coefficients; more on this in \S\ref{subsec:connections}. 

Moreover, these results highlight the role of Demazure-like operators in the theory of non-archimedean Whittaker functions for all types, which has inspired further interesting work on algebraic structures that produce them, including recent work by Sahi-Stokman-Venkateswaran~\cite{SSV, SSV2, VenkateswaranDuality}, Gao-Gurevich-Karasiewicz~\cite{ggk} and Buciumas-Patnaik~\cite{BuciumasPatnaik}. 
As a signpost, the lattice model suggests that the quantum group $U_{q}(\widehat{\mathfrak{gl}}(r|n))$ is the right algebraic object for expressing metaplectic Whittaker functions at the Iwahori level. 
It also reveals a surprising duality between metaplectic, spherical and nonmetaplectic, Iwahori Whittaker functions that allows one to transfer facts and structure from one context to the other \cite{BBBGDuality}.
This duality is very unexpected from the point of view of representation theory, but very natural in both the lattice model framework and its associated quantum groups.
More specifically, the duality in~\cite{BBBGDuality} manifests as a supersymmetry transformation relating quantum deformations of  $\widehat{\mathfrak{gl}}(1|n)$ and $\widehat{\mathfrak{gl}}(n|1)$.
The results of the present paper suggest that a much richer duality going far beyond \cite{BBBGDuality} is possible, making full use of the quantum group $U_{q}(\widehat{\mathfrak{gl}}(r|n))$.

Let us now go into more detail regarding the results and techniques appearing in this paper.

\subsection{Supersymmetric lattice models}\label{sub:intro:latticemodels}

The lattice models developed in this paper consist of admissible states, each of which is described by sets of colored paths traveling through the edges of a square lattice in a fixed finite rectangle.
A novelty of our approach is that those paths may travel in \emph{two different directions} --- one set of paths traveling down and to the right through the lattice and another set traveling down and to the left. We distinguish the two directions by referring to rightward moving paths as \emph{colored} paths and leftward moving paths as \emph{supercolored} paths (or sometimes ``scolored'' paths for brevity).

Earlier works (e.g.~\cite{BorodinWheelerColored, BBBGIwahori}) featured colored paths moving in just one direction (and whether that is rightward or leftward is of little difference, up to symmetry considerations). In particular, the work of Borodin and Wheeler~\cite{BorodinWheelerColored} and their subsequent collaborations has been influential on our present program in illustrating the utility of colored models to refine uncolored ones. For example, the original six-vertex model is known to have admissible states described by paths of a single color moving in one direction \cite[Ch.~8]{Baxter}, and the colored paths may be understood as refinements of these models. 
If we reverse the roles on \emph{horizontal} edges in the six-vertex model, exchanging the presence of a ``path'' with ``no path,'' then we form new paths --- a dual set of paths --- through the model that travel in the opposite direction. 
Supercolored paths refine these dual paths. In this way, the models of this paper go beyond the prior literature to give a doubly-refined version of the six-vertex model through the use of color and supercolor. Figure~\ref{fig:refinement} illustrates this refinement.

\begin{figure}[htpb]
  \centering
\begin{tikzpicture}[baseline={(0,1.5)}, scaled, scale=0.6, every node/.append style={font=\normalsize}]
  \begin{scope}[scale = 2]
  % horizontal lines with labels
  \foreach \i in {1,2,3}{
    \node at (-0.5,3-\i)  [label={[label distance=1em]left:$z_\i$}] {};
  }
  
  % vertical lines with labels
  \foreach \j[evaluate=\j as \col using {int(6-\j)}, evaluate=\j as \w using {int(2-mod((\j-1),3))}] in {1,...,6}{
    \draw (\j-1, -0.5) -- (\j-1,2.5) node [label={[align=center, label distance=1em]above:{$\col$}}] {};
   }

  % vertex dots
  \foreach \i in {0,...,3}{
    \foreach \j in {1,...,6}{
      \node[spin, font=\normalsize] at (\j-1,\i-0.5) {$+$};
    }
  }

  \end{scope}

  \draw[scolored-path, black] (10-\eps,5) -- (10-\eps,4) -- 
      (9, 4) node {\state{+}{}} --
      (7, 4) node {\state{+}{}} --
      (5, 4) node {\state{+}{}} -- (4-\eps,4) -- (4-\eps,2) --
      (3, 2) node {\state{+}{}} --
      (1, 2) node {\state{+}{}} --
      (-1, 2) node {\state{+}{}};

  \draw[scolored-path, black] (6-\eps,5) -- (6-\eps,0) --
      (5,0) node {\state{+}{}} --
      (3,0) node {\state{+}{}} --
      (1,0) node {\state{+}{}} --
      (-1,0) node {\state{+}{}};

  \draw[scolored-path, black] (2-\eps,5) -- (2-\eps,4) --
      (1,4) node {\state{+}{}} --
      (-1,4) node {\state{+}{}};

  \draw[colored-path, black] (2,5) node[circles={{colored, black}}] {\state{-}{}} -- (2,4) -- 
      (3,4) node {\state{-}{}} -- (4,4) --
      (4,3) node[circles={{colored, black}}] {\state{-}{}} -- (4,2) --
      (5,2) node {\state{-}{}} --
      (7,2) node {\state{-}{}} --
      (9,2) node {\state{-}{}} --
      (11,2) node {\state{-}{}};

  \draw[colored-path, black] (6, 5) node[circles={{colored, black}}] {\state{-}{}} --
      (6,3) node[circles={{colored, black}}] {\state{-}{}} --
      (6,1) node[circles={{colored, black}}] {\state{-}{}} --
      (6,0) --
      (7,0) node {\state{-}{}} --
      (9,0) node {\state{-}{}} --
      (11,0) node {\state{-}{}};

  \draw[colored-path, black] (10,5) node[circles={{colored, black}}] {\state{-}{}} -- (10,4) --
      (11,4) node {\state{-}{}};
  
\end{tikzpicture}  
  \quad
\begin{tikzpicture}[baseline={(0,1.5)}, scaled, scale=0.6, every node/.append style={font=\normalsize}]
  \begin{scope}[scale = 2]
  % horizontal lines with labels
  \foreach \i in {1,2,3}{
    \node at (-0.5,3-\i)  [label={[label distance=1em]left:$z_\i$}] {};
  }
  
  % vertical lines with labels
  \foreach \j[evaluate=\j as \col using {int(6-\j)}, evaluate=\j as \w using {int(2-mod((\j-1),3))}] in {1,...,6}{
    \draw (\j-1, -0.5) -- (\j-1,2.5) node [label={[align=center, label distance=1em]above:{$\col$}}] {};
   }

  % vertex dots
  \foreach \i in {0,...,3}{
    \foreach \j in {1,...,6}{
      \node[spin, font=\normalsize] at (\j-1,\i-0.5) {$+$};
    }
  }

  \end{scope}

  \draw[scolored-path, sred] (10-\eps,5) -- (10-\eps,4) -- 
      (9, 4) node {\state{}{0}} --
      (7, 4) node {\state{}{0}} --
      (5, 4) node {\state{}{0}} -- (4-\eps,4) -- (4-\eps,2) --
      (3, 2) node {\state{}{0}} --
      (1, 2) node {\state{}{0}} --
      (-1, 2) node {\state{}{0}};

  \draw[scolored-path, sblue] (6-\eps,5) -- (6-\eps,0) --
      (5,0) node {\state{}{2}} --
      (3,0) node {\state{}{2}} --
      (1,0) node {\state{}{2}} --
      (-1,0) node {\state{}{2}};

  \draw[scolored-path, sgreen] (2-\eps,5) -- (2-\eps,4) --
      (1,4) node {\state{}{1}} --
      (-1,4) node {\state{}{1}};

  \draw[colored-path, red] (2,5) node[circles={{colored, red}, {scolored, sgreen}}] {\state{3}{1}} -- (2,4) -- 
      (3,4) node {\state{3}{}} -- (4,4) --
      (4,3) node[circles={{colored, red}, {scolored, sred}}] {\state{3}{0}} -- (4,2) --
      (5,2) node {\state{3}{}} --
      (7,2) node {\state{3}{}} --
      (9,2) node {\state{3}{}} --
      (11,2) node {\state{3}{}};

  \draw[colored-path, blue] (6, 5) node[circles={{colored, blue}, {scolored, sblue}}] {\state{2}{2}} --
      (6,3) node[circles={{colored, blue}, {scolored, sblue}}] {\state{2}{2}} --
      (6,1) node[circles={{colored, blue}, {scolored, sblue}}] {\state{2}{2}} --
      (6,0) --
      (7,0) node {\state{2}{}} --
      (9,0) node {\state{2}{}} --
      (11,0) node {\state{2}{}};

  \draw[colored-path, green] (10,5) node[circles={{colored, green}, {scolored, sred}}] {\state{1}{0}} -- (10,4) --
      (11,4) node {\state{1}{}};
  
\end{tikzpicture}
\caption{Left: An admissible state from a six-vertex model, depicting both the path interpretation and the edge labels from the set $\{+,-\}$. Right: A doubly-refined version of the admissible state with $n=r=3$ where the set of $r$ colors are shown with solid lines and labeled from $\{1,2,3\}$ and the set of $n$ supercolors has dotted lines labeled from $\{ 0,1,2\}$. Here $+$ denotes the absence of color and scolor on the vertical edges where it appears.}
  \label{fig:refinement}
\end{figure}

The key feature we prove about these doubly refined lattice models is that they are \emph{solvable}. This means that there are so-called Boltzmann weights attached to each vertex, determined by the colors and scolors on adjacent edges, satisfying \emph{Yang-Baxter equations}. 
As the adjective suggests, the existence of Yang-Baxter equations typically allows one to ``solve'' for (i.e., give explicit expressions for) the weighted sum of admissible states made from taking the product of weights at each vertex. The resulting generating function is known as the \emph{partition function}. 

Solutions to Yang-Baxter equations are rare but a common algebraic source are the R-matrices of modules of quantum groups. As we mentioned above, we will ultimately show that the R-matrices for our models come from quantum affine superalgebras $U_q(\widehat{\mathfrak{gl}}(r|n))$ when representing Whittaker functions on $n$-fold covers of $\GL_r$. To connect this algebraic point of view to the lattice, one associates a basis vector in the quantum group module to every allowable set of colors on a lattice edge. Horizontal edges in our model carry a color or supercolor, and we show that these edges correspond to basis vectors in the $(n+r)$-dimensional evaluation representations of the superalgebra leading to the R-matrix in question. In the literature, these are known as ``Perk-Schultz'' R-matrices and appear in~{\cite{PerkSchultz,BazhanovShadrikov,ZhangRBSupergroup,ZhangFundamental,Kojima}}. Unfortunately, such a connection does not immediately lead to the necessary Yang-Baxter equations because we do not have a quantum superalgebra module interpretation of the vertical edges in our model. The resulting Boltzmann weights appear in Figure~\ref{fig:boltzmann_weights} and are completely new. In order to prove solvability, we present a combinatorial analogue of the fusion process~\cite{KulishReshetikhinSklyanin} that replaces each vertical edge with a family of vertical edges for each possible color and allows for a reduction argument to succeed.

What is yet more remarkable is that this doubly refined model of (s)colored paths from quantum superalgebra modules gives precisely the flexibility necessary to represent all metaplectic Iwahori Whittaker functions. Stated more precisely, the resulting partition function of the model is designed to give all values of these Whittaker functions. We highlight several details of this connection in the next subsection.

\subsection{Connecting integrable systems to metaplectic groups}
The lattice models of this paper are of Cartan type~A, and as such the associated partition functions may be used to represent metaplectic Iwahori-Whittaker functions on covers of $\mathbf{GL}_r(F)$. 
To determine a partition function for a lattice model, we must specify the boundary data; that is, the colors and supercolors associated to each boundary edge.
Figure~\ref{fig:translation} gives an overview of how each component of the metaplectic Iwahori-Whittaker function value determines a different aspect of the boundary data which will be described in detail below.
\begin{figure}[htpb]
  \centering
  \vspace{1em}
  \begin{equation*}
    \mathbf{z}^\rho\phi_{\tikzmarknode{theta1}{\theta}, \tikzmarknode{w1}{w}}(\mathbf{z}; \tikzmarknode{g1}{g}) = Z \left(\hspace{.5em}
\begin{tikzpicture}[remember picture, baseline=0.5cm, scale=0.5, every node/.append style={scale=0.8}]
  
  \draw[selection] (-.5,-.5) rectangle (.5,3.2) coordinate[pos=0] (theta2) {};
  \draw[selection] (6,-.5) rectangle (7,3.2) coordinate[pos=0] (w2) {};
  \draw[selection] (.75,4.5) rectangle (5.75,2.75) coordinate[pos=0] (g2) {};

  \draw[very thick] (0,0) node[very thick, densely dotted, sred, state] {} -- (0.75,0) node[label=right:$z_3$] {};
  \draw[very thick] (0,1) node[very thick, densely dotted, sblue, state] {} -- (0.75,1) node[label=right:$z_2$] {};
  \draw[very thick] (0,2) node[very thick, densely dotted, sred, state, label={above:$\theta$}] {} -- (0.75,2) node[label=right:$z_1$] {};
  \foreach \x in {0,...,4}{
    \draw[very thick] (\x+1.25, -1.25) node[state] {$+$} -- (\x+1.25,-.5);
    \draw[very thick] (\x+1.25, 3.25) node[state] {} -- (\x+1.25, 2.5);
  }
  \node at (3.25, 4) {$(w', \lambda)$};
  \draw[very thick] (6.5, 0) node[very thick, green, state] {} -- (5.75, 0);
  \draw[very thick] (6.5, 1) node[very thick, blue, state] {} -- (5.75, 1);
  \draw[very thick] (6.5, 2) node[very thick, red, state, label={above:$w$}] {} -- (5.75, 2);
  \draw[very thick] (0.75,-.5) rectangle (5.75, 2.5);
\end{tikzpicture}
\hspace{.5em}\right)
\begin{tikzpicture}[overlay, remember picture]
  \draw[selection-line] ($(g2) + (1.25,0.1)$) -| ++(0,0.2) -| node[pos=0.25,halo-rect, inner sep=1pt] {\tiny $g \in \GL_r(F)$} (g1);
  \draw[selection-line] ($(w1.south) - (0,0.2)$) -| ++(0,-1.5) -| node[pos=0.375, halo-rect, inner sep=1pt] {\tiny Iwahori vector} ($(w2) + (0.25,0)$);
  \draw[selection-line] ($(theta1.south) - (0,0.2)$) -| ++(0,-1.8) -| node[pos=0.25, halo-rect, inner sep=1pt] {\tiny Whittaker functional} ($(theta2) + (0.25,0)$);
\end{tikzpicture}
\end{equation*}
\vspace{1em}
\caption{An illustration of the way the boundary data encodes each of the various defining objects for an evaluation of a metaplectic Iwahori-Whittaker function.}
  \label{fig:translation}
\end{figure}

The value of the metaplectic Iwahori Whittaker function $\phi$ appearing on the left-hand side of the equality in the figure is shown to be equal to the partition function $Z$ up to the prefactor $\mathbf{z}^\rho$ where $\rho = (r-1, r-2, \ldots, 0)$ in Theorem~\ref{thm:main}. That equality makes use of the following dictionary. 
The Langlands-Satake parameters $\mathbf{z} = (z_1, \ldots, z_r) \in (\mathbb{C}^\times)^r$ determine the unramified principal series for $\phi$ and the weights for vertices at row $i$ in the lattice model.
The $\theta \in (\mathbb{Z}/n\mathbb{Z})^r$ labels the metaplectic Whittaker functional for $\phi$ as well as the scolors on the left boundary of $Z$.
The Weyl group element $w$ labels the Iwahori fixed basis vector in the unramified principal series for $\phi$ and the order of the colors on the right boundary of $Z$.
Lastly, $\phi$ is determined by its values on particular group elements $g$ which are specified by certain pairs of weights $\lambda$ and Weyl group elements $w'$.
On the lattice model side $\lambda$ and $w'$ determine the top boundary conditions.

Theorem~\ref{thm:main} is proved using the solvability of the lattice model. 
More precisely, we show that the same recursive relations govern both the partition functions of our lattice models and the metaplectic Iwahori-Whittaker functions. 
Figure~\ref{fig:intertwiners-R-matrices-simplified} illustrates how this connection arises. For the lattice models, those recursive relations come from partition function identities using the Yang-Baxter equation. 
More precisely, one attaches a new, rotated ``R-vertex'' (whose Boltzmann weights are taken from the corresponding R-matrix) to boundary edges in a pair of adjacent rows 
and uses the familiar ``train argument'' by repeatedly applying Yang-Baxter equations to obtain a recursive relation on partition functions. As mentioned before, the R-matrix that solves the system corresponds to the quantum group $U_q(\hat{\mathfrak{gl}}(r|n))$. In the $p$-adic representation theory, there is a parallel story where the recursion on Whittaker functions arises from two different ways of computing the effect of standard intertwining operators for principal series. 
If one expresses the composition of a given Whittaker functional with the intertwining operator in the basis of metaplectic Whittaker functionals, the resulting coefficients are known as the ``Kazhdan-Patterson scattering matrix'' which are seen to {\it exactly} match the R-matrix for the lattice model involving pairs of supercolors. 
Alternatively, one may act by intertwining operators on the Iwahori-fixed vectors in the representation and expand in the basis of Iwahori-fixed vectors. 
The resulting ``intertwiner structure constants'' {\it exactly} match the R-matrix entries, now formed with pairs of colors.  
The interaction of colors and supercolors is only witnessed in the middle of the train argument as seen in the center of Figure~\ref{fig:intertwiners-R-matrices-simplified}.
Here the mixed R-vertex appears in the midst of a pair of rows, but is never seen on the $p$-adic side, which only sees the subalgebras $U_{q^{-1}}(\hat{\mathfrak{gl}}(n))$ and $U_q(\hat{\mathfrak{gl}}(r))$ of the larger quantum group $U_q(\hat{\mathfrak{gl}}(r|n))$. 
Full details on this connection are provided in Section~\ref{sec:conclusions}.

\begin{figure}[htpb]
  \centering

\begin{equation*}
  \begin{gathered}
    \tikzmarknode[selection, draw, inner sep=2pt]{A1}{\Omega_\theta^{s_i\mathbf{z}^{}} \circ \mathcal{A}_{s_i}^{\mathbf{z}^{}}}\hspace{1pt} \bigl( \pi(g) \Phi_w^{\mathbf{z}^{}} \bigr) = \Omega_\theta^{s_i\mathbf{z}^{}}\bigl( \pi(g) \tikzmarknode[selection, draw, inner sep=2pt]{A2}{\mathcal{A}_{s_i}^{\mathbf{z}^{}}\Phi_w^{\mathbf{z}^{}}}\hspace{1pt} \bigr) \\[1em]
\begin{tikzpicture}[remember picture, baseline=0.5cm, scale=0.5, every node/.append style={scale=0.8}]
  \begin{scope}

    \draw[selection] (.5,2.5) rectangle (-1.5,.5) coordinate[pos=0] (R1) {}; 
    
    \draw[very thick] (0,0) node[very thick, densely dotted, sred, state] {} -- (0.75,0);
    \draw[very thick] (0,1) node[very thick, densely dotted, sblue, state] (A) {} -- (0.75,1);
    \draw[very thick] (0,2) node[very thick, densely dotted, sred, state] (B) {} -- (0.75,2);
    \foreach \x in {0,...,3}{
      \draw[very thick] (\x+1.25, -1.25) node[state] {$+$} -- (\x+1.25,-.5);
      \draw[very thick] (\x+1.25, 3.25) node[state] {} -- (\x+1.25, 2.5);
    }
    \draw[very thick] (5.5, 0) node[very thick, green, state] {} -- (4.75, 0);
    \draw[very thick] (5.5, 1) node[very thick, blue, state] {} -- (4.75, 1);
    \draw[very thick] (5.5, 2) node[very thick, red, state] {} -- (4.75, 2);
    \draw[very thick] (0.75,-.5) rectangle (4.75, 2.5);

    \draw[very thick] (-1,2) node[very thick, densely dotted, sred, state] {} -- (A);
    \draw[very thick] (-1,1) node[very thick, densely dotted, sblue, state] {} -- (B);
    \node[scale=1/0.8] at (7.25+0.5,1) {$=\ldots=$};
  \end{scope}

  \begin{scope}[xshift=10cm]

    \draw[selection] (3,2.5) rectangle (5,.5) coordinate[pos=0] (R2) {};

    \draw[selection-line] (4,2.5) -- (4,4.5) node[halo-rect, inner sep=1pt, scale=1/0.8] {\footnotesize supersymmetric R-matrix};
    
    \draw[very thick] (0,0) node[very thick, densely dotted, sred, state] {} -- (0.75,0);
    \draw[very thick] (0,1) node[very thick, densely dotted, sblue, state] {} -- (0.75,1);
    \draw[very thick] (0,2) node[very thick, densely dotted, sred, state] {} -- (0.75,2);
    \foreach \x in {0,...,1}{
      \draw[very thick] (\x+1.25, -1.25) node[state] {$+$} -- (\x+1.25,-.5);
      \draw[very thick] (\x+1.25, 3.25) node[state] {} -- (\x+1.25, 2.5);
    }
    \draw[very thick] (5.25, 0) -- (2.75, 0);
    \draw[very thick] (3.5, 1) node[very thick, red, state] (G) {} -- (2.75, 1);
    \draw[very thick] (3.5, 2) node[very thick, densely dotted, sgreen, state] (H) {} -- (2.75, 2);
    \draw[very thick] (0.75,-.5) rectangle (2.75, 2.5);

    \draw[very thick] (5.25,2) -- (4.5,2) node[very thick, red, state] {} -- (G);
    \draw[very thick] (5.25,1) -- (4.5,1) node[very thick, densely dotted, sgreen, state] {} -- (H);
    
    \draw[very thick] (5.25,-.5) rectangle (7.25, 2.5);
    \foreach \x in {0,...,1}{
      \draw[very thick] (\x+1.25+4.5, -1.25) node[state] {$+$} -- (\x+1.25+4.5,-.5);
      \draw[very thick] (\x+1.25+4.5, 3.25) node[state] {} -- (\x+1.25+4.5, 2.5);
    }

    \draw[very thick] (8, 0) node[very thick, green, state] {} -- (7.25, 0);
    \draw[very thick] (8, 1) node[very thick, blue, state] {} -- (7.25, 1);
    \draw[very thick] (8, 2) node[very thick, red, state] {} -- (7.25, 2);
    \node[scale=1/0.8] at (9.5+0.5,1) {$=\ldots=$};
  \end{scope}
  \begin{scope}[xshift=22cm]

    \draw[selection] (5,2.5) rectangle (7,.5) coordinate[pos=0] (R2) {};
    
    \draw[very thick] (0,0) node[very thick, densely dotted, sred, state] {} -- (0.75,0);
    \draw[very thick] (0,1) node[very thick, densely dotted, sblue, state] {} -- (0.75,1);
    \draw[very thick] (0,2) node[very thick, densely dotted, sred, state] {} -- (0.75,2);
    \foreach \x in {0,...,3}{
      \draw[very thick] (\x+1.25, -1.25) node[state] {$+$} -- (\x+1.25,-.5);
      \draw[very thick] (\x+1.25, 3.25) node[state] {} -- (\x+1.25, 2.5);
    }
    \draw[very thick] (5.5, 0) node[very thick, green, state] {} -- (4.75, 0);
    \draw[very thick] (5.5, 1) node[very thick, red, state] (G) {} -- (4.75, 1);
    \draw[very thick] (5.5, 2) node[very thick, blue, state] (H) {} -- (4.75, 2);
    \draw[very thick] (0.75,-.5) rectangle (4.75, 2.5);

    \draw[very thick] (6.5,2) node[very thick, red, state] {} -- (G);
    \draw[very thick] (6.5,1) node[very thick, blue, state] {} -- (H);
  \end{scope}
\end{tikzpicture}
\end{gathered}
\begin{tikzpicture}[overlay, remember picture]
  \draw[selection-line] ($(R1) + (-0.5,0)$) .. controls ($(R1) + (0,1.5)$) and ($(A1) + (-1,-1)$) .. (A1) node[midway, halo-rect, inner sep=1pt] {\footnotesize Kazhdan-Patterson scattering matrix};
  \draw[selection-line] ($(R2) + (0.5,0)$) .. controls ($(R2) + (0,1.5)$) and ($(A2) + (1,-1)$) .. (A2) node[midway, halo-rect, inner sep=1pt] {\footnotesize intertwiner structure constants};
\end{tikzpicture}
\end{equation*}
\caption{The action of the intertwining operators $\mathcal{A}_{s_i}$ on the Whittaker functionals $\Omega_\theta$ and on the Iwahori basis vectors $\Phi_w$ correspond to the Kazhdan–Patterson scattering matrix and interwiner structure constants, respectively.
These are equal to the supercolor and color restrictions of the quantum group R-matrix which act on the partition function giving functional equations for the latter.}
  \label{fig:intertwiners-R-matrices-simplified}
\end{figure}

Figure~\ref{fig:intertwiners-R-matrices-simplified} shows one example in pictures of how the $R$-vertex may attach to the original rectangular lattice of interest. For any pair of (s)colors in adjacent rows, there are actually two such $R$-vertices with non-zero Boltzmann weights corresponding to preserving or swapping the (s)colors involved. From the two possibilities on the left- and right-hand boundaries we obtain relations involving four different partition functions. Prior lattice models with just one colored boundary and one uncolored boundary resulted in three-term relations expressible as Demazure-like divided difference operators. We show that our new four term relations are expressible as vector-valued Demazure-like operators (called ``vector metaplectic Demazure-Whittaker operators'' in Section~\ref{sec:DW-operators}).

We have made this connection between lattice models and $p$-adic representation theory in type A. By nature, lattice model geometries must strongly depend on the Cartan type; compare for example the more elaborate geometries in~\cite{Kuperberg:roof, Ivanov,gray:thesis, buciumas:scrimshaw:BC, zhong}. However, we emphasize that the connection between lattice models and representation theory in type A motivated us to consider the {\it type-independent} treatment of metaplectic Iwahori Whittaker functions by vector-valued divided difference operators in Section~\ref{sec:whittaker}.
The proofs are based solely on representation theory.

Having established the equality in Theorem~\ref{thm:main}, one immediately obtains a combinatorial formula for metaplectic Whittaker functions on covers of general linear groups as a weighted sum over admissible states of our model. These states can in turn be shown to be in bijection to colored Gelfand-Tsetlin patterns, leading to a refinement of McNamara's result~\cite[Theorem 8.6]{McNamara:Duke}. Naprienko~\cite{Naprienko} shows how this expression arises naturally from certain double coset decompositions in $p$-adic groups.

\subsection{Connections to other works}\label{subsec:connections}
At around the same time the present work initially appeared, Aggarwal, Borodin and Wheeler~\cite{AggarwalBorodinWheelerColored} independently introduced supersymmetric models that make use of modules for the quantum group $U_q(\widehat{\mathfrak{sl}}(m|n))$. 
The Boltzmann weights in the respective papers were produced by very different means resulting in very different fusion procedures for constructing solvable lattice models.
We were recently informed by Aggarwal that both the lattice models in~\cite{BBB} and~\cite{BBBGIwahori}, which are special cases of the present paper, can be obtained as a specialization of their most general model (their functions $\tilde{F}_\lambda$ appearing in \cite[Definition 11.2.2]{AggarwalBorodinWheelerColored} should correspond after some renormalization and introduction of Gauss sums to the metaplectic spherical Whittaker functions of~\cite{BBB}). 

The papers~\cite{AggarwalBorodinPetrovWheelerColored,AggarwalBorodinWheelerColored} introduce many innovations to the study of special functions, one example being the study of certain transition and structure coefficients for the functions $\tilde{F}_{\lambda}$ mentioned above. It has been observed by Aggarwal and Buciumas that combining results of~\cite{AggarwalBorodinWheelerColored, BBB} one can reinterpret and reprove certain results in~\cite{BuciumasPatnaik} combinatorially. In particular, one may use the results of~\cite{AggarwalBorodinWheelerColored,BBB} to show that the structure coefficients appearing in the action of the metaplectic spherical Hecke algebra on its Gelfand-Graev representation are quantum Littlewood-Richardson coefficients (at least in Cartan type A). 
This is a result at the spherical level, motivated by a conjecture of Lurie and Gaitsgory~\cite{gaitsgory} and proved using different methods in~\cite[Corollary 1.0.4]{BuciumasPatnaik}. 
In the future, we would like to understand how to use the techniques in~\cite{AggarwalBorodinWheelerColored} and the ones developed in this paper to study the Gelfand-Graev module combinatorially at the Iwahori level in view of Gaitsgory's work~\cite{gaitsgory:iwahori}.

\medskip
\noindent
\textbf{Contents.}
The paper is outlined as follows.
In Section~\ref{sec:lattice} we introduce our new class of supersymmetric lattice models which are constructed through a fusion process similar to that of~\cite{KulishReshetikhinSklyanin}.
The model that is used as an input in this fusion process is called \emph{monochrome} and is detailed in Section~\ref{sec:monochrome}, while the fusion process itself is explained in Section~\ref{sec:fusion}.
In preparation for the proof of Theorem~\ref{thm:main} we evaluate the partition function of a certain ground state system in Section~\ref{sec:ground-state}.
In Section~\ref{sec:R-matrix} we prove that our lattice model is solvable; that is, we prove an associated Yang-Baxter equation in Corollary~\ref{cor:YBE}. To show this for every $n$ and $r$, we first prove an auxiliary Yang-Baxter equation for the monochrome model in Theorem~\ref{thm:ybemc}
by reducing the problem to the verification of a finite number of equations. 
In the same section we also show that the R-matrix is a Drinfeld twist of the $U_q(\widehat{\mathfrak{gl}}(r|n))$ R-matrix in Proposition~\ref{prop:Drinfeldtwisting}.
From the Yang-Baxter equation we obtain a recursion relation for the partition function in Proposition~\ref{prop:lattice-recursion} of Section~\ref{sec:lattice-recursion} that we later compare with the recursion relations for the metaplectic Iwahori Whittaker functions. 

In Section~\ref{sec:whittaker} we turn our attention to metaplectic Iwahori Whittaker functions for general reductive groups, with some introductory theory and notation introduced in Section~\ref{sec:metaplectic-prel}.
Here, we also compute a base case for the Whittaker functions in Proposition~\ref{prop:rep-base-case}.
The recursion relations for the Whittaker functions are obtained in Proposition~\ref{prop:rep-recursion} in Section~\ref{sec:DW-operators} by studying intertwining operators.
These Whittaker functions can be conveniently described using the (vector) metaplectic Demazure-Whittaker operators as shown in Theorem~\ref{thm:dem_recurse} in Section~\ref{sec:DW-operators}.
These operators are shown to produce a representation of the affine metaplectic Hecke algebra in Theorem~\ref{thm:hecke}.
An expression for all values of the metaplectic Iwahori Whittaker functions in terms of these operators is deduced in Corollary~\ref{cor:evaluate}.
In Section~\ref{sec:scalar-Demazure} we match our vector metaplectic Demazure-Whittaker operators with the averaged metaplectic Demazure-Whittaker operators of~\cite{ChintaGunnellsPuskas,PatnaikPuskasIwahori} in Theorem~\ref{thm:T-average}, which also recovers the Chinta-Gunnells Weyl group action from non-archimedean representation theory.

We conclude the paper with Section~\ref{sec:conclusions} where we state and prove the main theorem, Theorem~\ref{thm:main} in Section~\ref{sec:main-thm-and-proof}, after detailing the computation of the metaplectic Iwahori Whittaker functions for $\GL_r$ and a particular convenient cocycle in Section~\ref{sec:GLr}.
By comparing the partition function for the ground state system with the above base case for the Whittaker function and then using the matching recurrence relations, we can conclude the equality for all cases.
Then in Section~\ref{sec:intertwiners-R-matrices}, we give an interpretation of the intertwining operators in the non-archimedean representation theory in terms of different R-matrices on the lattice model side as illustrated by equation~\eqref{eq:commutativediagrams} and Figure~\ref{fig:intertwiners-R-matrices}.

\medskip
\noindent\textbf{Acknowledgements.}
The authors thank Slava Naprienko, Manish Patnaik, Anna Pusk\'as, Siddhartha
Sahi, Jason Saied and Chenyang Zhong for useful discussions and comments.
The authors would like to thank Amol Aggarwal, Alexei Borodin and Michael Wheeler for illuminating discussion regarding their work~\cite{AggarwalBorodinWheelerColored}. 

This work was supported by NSF grants DMS-1801527 (Brubaker) and DMS-1601026 (Bump).
Buciumas was supported by the Australian Research Council DP180103150 and DP17010264, NSERC Discovery RGPIN-2019-06112 and the endowment of the M.V. Subbarao Professorship in Number Theory.
Gustafsson was supported by the Swedish Research Council (Vetenskapsr\aa det) grant 2018-06774.

\section{Supersymmetric lattice models}
\label{sec:lattice}

\subsection{The monochrome model}
\label{sec:monochrome}

Our model consists of {\it systems}~$\mathfrak{S}$, certain
finite collections of states and associated \textit{Boltzmann weights} for each state in the system, to be described
more precisely in a moment. The sum $Z(\mathfrak{S})$ of these Boltzmann weights over all states in the system $\mathfrak{S}$ is 
the \textit{partition function} of the system.
The solvability of the model amounts to the existence of \textit{Yang-Baxter equations}
that allow us to study these partition functions through the commutation
properties of the \textit{row transfer matrices}, following the method
by which Baxter~\cite{Baxter} studied such models.

This paper treats three equivalent versions of the model, which
we will call \textit{monochrome}, \textit{color-fused} and
\textit{fully-fused}. 
Collectively, we call these models \emph{metaplectic Iwahori ice} in reference to six-vertex models often being called ice models.
In this section we describe the
monochrome model. Its admissible states are described by labelings of
a rectangular grid with $Nnr$
columns and $r$ rows, where $n$ and $r$ are fixed integers (later, the cover degree and the rank of
the associated metaplectic general linear group, respectively) and $N$ is a sufficiently large integer. 
An example of such a grid with $n=r=3$ and $N=2$ is given in Figure~\ref{fig:thegrid}. In particular,
there is a vertex at each crossing of a row and column in the grid, and each vertex has 4 adjacent edges.
Edges that are adjacent to just one vertex are called \textit{boundary edges} while edges adjacent to two vertices
are called \textit{internal edges}. A fixed labeling of all boundary edges is referred to as a \textit{boundary condition}
and, given a boundary condition, an \textit{(admissible) state} is a labeling of internal edges satisfying certain given restrictions. 
We will dictate the restrictions for our monochrome models momentarily. Finally, a \textit{system} $\mathfrak{S}$ is defined 
as the set of admissible states for a fixed boundary condition,
together with a set of \textit{Boltzmann weights} $\beta(\mathfrak{s})$ for every admissible state $\mathfrak{s} \in \mathfrak{S}$.  

Before we describe monochrome systems formally, we give an intuitive
description of its admissible states. An (admissible) state of the system assigns labels (or ``attributes'')
called \textit{color} and \textit{scolor} to the edges.
Borrowing the naming convention from physics, the prefix \emph{s} is used to imply that scolor is the supersymmetry partner of color.
A \textit{color} $c$ is an attribute indexed by an
integer modulo~$r$. The set of possible colors will be denoted
$\{c_1,\cdots,c_r\}$ and ordered so that $c_i<c_j$ if and only if
$i<j$. A \textit{scolor} $w$ is an attribute indexed by
an integer modulo $n$. The possible scolors are denoted
$\{w_0,\cdots,w_{n-1}\}$ and are ordered so that
$w_i<w_j$ if and only if $i<j$. 

In any state, the edges
that are labeled with a color $c$ or scolor $w$ will also be said to
\textit{carry} the color or scolor. According to the restrictions we will impose, the edges carrying any
color $c$ will form a connected path through our rectangular grid, traveling from the top boundary
and exiting at the right boundary. Similarly the edges
that carry the scolor $w$ form a path beginning at the top boundary of the grid
and end up at the left boundary. To separate the two, we will draw color-paths with thick solid lines and scolor-paths with dotted lines; see Figure~\ref{fig:monochrome}
for an example. More than one scolor-path may have the same scolor $w$,
but for the current application we do not need systems in which
different color-paths have the same color.

Edges may be classified as \textit{horizontal} (occurring in a row of the grid) or \textit{vertical} (occuring in a column).
In any (admissible) state $\mathfrak{s}$, every horizontal edge carries a unique
attribute, either a color or scolor. Every vertical edge carries
either no attribute, or two: a color and a scolor. In Figure~\ref{fig:monochrome},
the vertical edges that carry the color $c_i$ and the scolor $w_j$ are
labeled $c_iw_j$, or simply $ij$ when no confusion may arise. So roughly,
admissible states are paths of a given color or scolor traveling through the grid as described
above and the boundary conditions serve to dictate the beginning and ending location of each path.

We now give a more formal description of the systems. First, 
we label each column of the model with a pair $(c, w)$ consisting
of a color $c$ and scolor $w$ according to the following cyclic ordering.
Let $(c,w)=(c_i,w_j)$ be such a pair
with $1\leqslant i\leqslant r$ and $0\leqslant j\leqslant n-1$.
Then the \textit{successor} $(c',w')$ to $(c,w)$ in this cyclic
ordering is defined by
\begin{equation}
\label{ccsuccessor}
  (c', w') = \left\{\begin{array}{ll}
     (c_{i+1}, w_j) & \text{if $i<r$,}\\
     (c_1, w_{j-1}) & \text{if $i = r, j > 0$}\\
     (c_1,w_{n-1}) & \text{if $i = r, j = 0$} .
   \end{array}\right.
\end{equation}
Note that the colors $c_i$ first increase to $r$, then the accompanying scolor $w_j$ decreases by one and color is reset to $c_1$,
and then cycle repeats. The leftmost column is labeled $(c_1, w_{n-1})$ and each column to the right of any $(c,w)$ is
labeled by its successor $(c',w')$ in the cyclic ordering. According to the size of our grid, each pair appears $N$ times in
the labeling. Next we label each row with a formal parameter $z_i$ where
the index $i$ increases from $1$ to $r$ from top row to bottom row; when we specialize these $r$ parameters 
$(z_1, \ldots, z_r)$ to non-zero complex numbers, we 
call them \textit{spectral parameters}. See Figure~\ref{fig:thegrid} for an example of the labels
on columns and rows.
In the figure we have grouped each sequence of $r$ columns of the same scolor into blocks enumerated from right to left.
These blocks will become the columns of the color-fused system described in Section~\ref{sec:fusion}.
Finally, we label each vertex $T$ in the grid with the labels from its row and column, writing $T_z^{cw}$ for
the vertex with column label $(c,w)$ and row label $z$.  

\begin{figure}[htbp]
\centering
\begin{tikzpicture}[scale=0.75, every node/.style={font=\scriptsize}]
  % horizontal lines with labels
  \foreach \i in {1,2,3}{
    \draw (-0.5,3-\i) node [label=left:$z_\i$] {} -- (17.5,3-\i);
  }
  
  % vertical lines with labels
  \foreach \j[evaluate=\j as \c using {int(mod(\j-1,3)+1)}, evaluate=\j as \w using {int(2-mod(int((\j-1)/3),3))}] in {1,...,18}{
    \draw (\j-1, -0.5) -- (\j-1,2.5) node [label={[align=center]above:$c_{\c}$ \\ $w_{\w}$}] {};
  }

  % vertex dots
  \foreach \i in {1,2,3}{
    \foreach \j in {1,...,18}{
      \draw[fill=black] (\j-1,\i-1) circle (.05);
    }
  }

  % column color blocks
  \foreach \j[evaluate=\j as \col using {int(5-\j)}] in {0,...,5}{
    \draw[line width=1pt, decoration={brace},decorate] (3*\j-0.3,4) -- node[above=6pt] {block $\col$} (3*\j+2+0.3,4);
  }

  \node[label={[align=right, xshift=-1em]above:$c = \phantom{w_2}$ \\ $w = \phantom{w_2}$}] at (0,2.5) {};
  \node[label={left:$z = \phantom{z_1}$}] at (-0.5,2) {};

\end{tikzpicture}
\caption{The monochrome grid. Every vertex (represented here by a dot) in each row is assigned the same spectral parameter
$z_i$, and every vertex in each column is assigned the same scolor and color;
the colors $c$ and scolors $w$ rotate cyclicly through all the possibilities.
The columns with the same scolors are combined into blocks numbered from right to left, and these blocks will later become the columns of the color-fused model.
In this example, $n=r=3$ and $N=2$.}
\label{fig:thegrid}
\end{figure}

\begin{figure}[htbp]
{\scriptsize
\[\begin{array}{|c|c|c|c|c|c|}
\hline
\text{\normalsize $\texttt{a}_1$} & \text{\normalsize $\texttt{a}_2$} & \text{\normalsize $\texttt{b}_1$} & \text{\normalsize $\texttt{b}_2$} & \text{\normalsize $\texttt{c}_1$} & \text{\normalsize $\texttt{c}_2$} \\
\hline
%a1
\begin{tikzpicture}[scaled, font=\normalsize]
  \draw[spin-path] (0,1) node {$+$} -- (0,-1) node {$+$};
  \draw[scolored-path] (-1,0) node {$y$} -- (1,0) node {$y$};
  \node[halo] at (0,0) {$T_z^{cw}$};
\end{tikzpicture}
&
%a2
\begin{tikzpicture}[scaled, font=\normalsize]
  \draw[both-path] (0,1) node {$\scriptstyle cw$} -- (0,-1) node {$\scriptstyle cw$};
  \draw[colored-path] (-1,0) node {$a$} -- (1,0) node {$a$};
  \node[halo] at (0,0) {$T_z^{cw}$};
\end{tikzpicture}
&
%b1
\begin{tikzpicture}[scaled, font=\normalsize]
  \draw[both-path] (0,1) node {$\scriptstyle cw$} -- (0,-1) node {$\scriptstyle cw$};
  \draw[scolored-path] (-1,0) node {$y$} -- (1,0) node {$y$};
  \node[halo] at (0,0) {$T_z^{cw}$};
\end{tikzpicture}
&
%b2
\begin{tikzpicture}[scaled, font=\normalsize]
  \draw[spin-path] (0,1) node {$+$} -- (0,-1) node {$+$};
  \draw[colored-path] (-1,0) node {$a$} -- (1,0) node {$a$};
  \node[halo] at (0,0) {$T_z^{cw}$};
\end{tikzpicture}
&
%c1
\begin{tikzpicture}[scaled, font=\normalsize]
  \draw[spin-path] (0,1) node {$+$} -- (0,0);
  \draw[colored-path] (-1,0) node {$c$} -- (0,0);
  \draw[both-path] (0,0) -- (0,-1) node {$\scriptstyle cw$};
  \draw[scolored-path] (0,0) -- (1,0) node {$w$};
  \node[halo] at (0,0) {$T_z^{cw}$};
\end{tikzpicture}
&
%c2
\begin{tikzpicture}[scaled, font=\normalsize]
  \draw[spin-path] (0,-1) node {$+$} -- (0,0);
  \draw[colored-path] (1,0) node {$c$} -- (0,0);
  \draw[both-path] (0,0) -- (0,1) node {$\scriptstyle cw$};
  \draw[scolored-path] (0,0) -- (-1,0) node {$w$};
  \node[halo] at (0,0) {$T_z^{cw}$};
\end{tikzpicture}
\\ \hline
%a1
1
&
%a2
\begin{array}{ll} v & \text{if } c>a\\ z & \text{if } c=a\\ 1 & \text{if } c<a\end{array} 
&
%b1
g(y-w) 
&
%b2
\begin{array}{ll} z&\text{if $a=c$,}\\ 1 &\text{otherwise.}\end{array} 
&
%c1
(1-v)z 
&
%c2
1
\\ \hline
\end{array}
\]}
\caption{The Boltzmann weight at a vertex of color $c$
and scolor $w$ labeled by $cw$. Here $a$ is a color (which may be
the same or different from $c$) and $y$ is a scolor (which may be the
same or different from $w$). We label each vertex configuration according to the first row of the table.
If the vertex configuration of edges does not
appear in this table, the Boltzmann weight is zero.
For the Boltzmann weight $g(y-w)$ we identify each scolor $w_j$ with the integer $j$.}
\label{fig:boltzmann_weights}
\end{figure}

In decorating the edges of our model with attributes color and scolor, we only allow for the
configurations in Figure~\ref{fig:boltzmann_weights} on the edges adjacent to any vertex.
In particular, in the column labeled $(c,w)$, the vertical edges are only
allowed to carry \textit{both} attributes $c$ and $w$ or
\textit{neither}. 
This explains the name \emph{monochrome} for the system, since each column can only contain a specific color and scolor.
A horizontal edge must carry a single color~$c_i$ or a single scolor~$w_j$ (independent of column labels) but not both. It is
sometimes convenient to refer to the collection of all possible sets of attributes on a given edge. We call this collection the
\textit{spinset} for an edge. So the spinset for every horizontal edge is
\begin{equation}
\label{horizontalspinset}
\{c_1,\cdots,c_r,w_0,\cdots w_{n-1} \}
\end{equation}
while the spinset for a vertical edge adjacent to vertex $T_z^{cw}$ is the two-element set $\{ +, (c,w) \}$, where $+$ is a placeholder for no attribute. 

We will now describe the boundary conditions which assign to each boundary edge an element of its spinset.
These conditions are parametrized by $\mu\in\mathbb{Z}^r$ with non-negative parts, an element $\theta$ in $(\mathbb{Z}/n\mathbb{Z})^r$ and a Weyl group element $w$ in $S_r$ as summarized in Figure~\ref{fig:conventions}.
At row $i$, the left boundary edge is assigned the scolor $\theta_i$ and the right boundary edge the color $c_{r+1-w^{-1}(i)}$, that is, the color $(w P)_i$ where $P = (c_r, c_{r-1}, \ldots, c_1)$ is the palette of colors in decreasing order.
Recall that vertical edges have either no attribute, which we denote by $+$ and will call \emph{empty}, or the color and scolor of the corresponding column, which we will called \emph{filled in}.
Thus we only need to specify which columns are filled in and which are empty.
All edges at the bottom boundary are empty, that is, assigned $+$.
For each $j$ we fill in the column at the top boundary which is in block~$\mu_j$ with color~$c_{r+1-j}$, that is, color~$(P)_j$.

\begin{figure}[htpb]
  \centering
  \begin{tikzpicture}
    \draw (0,0) -| (5,3)
      node[pos=0.75, right, align=center] {$\quad\begin{aligned}\text{row} & \quad i \\ \text{color} & \quad (w P)_i \end{aligned}$}
      -| (0,0)
      node[pos=0.25, above, align=center] {$\begin{aligned}\text{block number} & \quad \mu_j \\ \text{color} & \quad (P)_j\end{aligned}$}
      node[pos=0.75, right, align=center] {$\quad\begin{gathered} z_1 \\ z_2 \\ \vdots \\ z_r \end{gathered}$}
      node[pos=0.75, left, align=center] {$\begin{aligned}\text{row} & \quad i \\ \text{scolor} & \quad \theta_i \end{aligned}\quad$};
  \end{tikzpicture}
  \caption{Boundary conditions as determined by $\mu\in\mathbb{Z}^r$ with non-negative parts, $w \in S_r$ and $\theta \in (\mathbb{Z}/n\mathbb{Z})^r$.
    The block number refers to the blocks shown in Figure~\ref{fig:thegrid} for the monochrome system, or equivalently to the column numbers of the color-fused system discussed in Section~\ref{sec:fusion}.
  Here $P = (c_r, c_{r-1}, \ldots, c_1)$ is a tuple of the colors ordered from largest to smallest.}%
  \label{fig:conventions}
\end{figure}

We recall the following definition from~\cite{BBBGIwahori} for a root datum with weight lattice $\Lambda$, root system $\Delta$ with positive and negative roots $\Delta^\pm$ and Weyl group $W$ as usual.

\begin{definition}
  \label{def:almost_dominant}
  Let $w'\in W$, let $\alpha_i$ be a simple root and $\alpha_i^\vee$ the
  corresponding coroot. A weight $\lambda\in\Lambda$ is \textit{$w'$-almost dominant} if
  \begin{equation}
    \label{eq:almost-dominant}
    \langle \alpha^\vee_i, \lambda \rangle \geqslant
    \begin{cases}
      0 & \text{if } (w')^{-1}\alpha_i \in\Delta^+\,, \\
      -1 & \text{if } (w')^{-1} \alpha_i \in\Delta^-\,.
    \end{cases}
    \qquad \text{for all simple roots $\alpha_i$.}
  \end{equation}
\end{definition}

In this section we use the $\GL_r$ weight lattice $\Lambda$ which we identify with~$\mathbb{Z}^r$.

\begin{remark}
  \label{rem:almost-dominant}
  We showed in Proposition~7.2 of~\cite{BBBGIwahori} that for each weight
  $\mu$ there exists a unique pair $(w',\lambda)$ with $\lambda \in \Lambda$ being
  $w'$-almost dominant and $w'\mu = \lambda + \rho$ where $\rho =
  (r-1,r-2, \ldots, 0)$.  Concretely, this means that $w'\mu$ is a dominant
  weight, but if $\mu$ contains repeated entries, there are several
  permutations $w'$ that would make $w'\mu$ dominant. The weight $\lambda = w'\mu-\rho$
  may not be dominant, but the Weyl group element $w'$ that makes $\lambda$
  $w'$-almost dominant is unique in the sense that the order of the colors at
  the top boundary read from left to right is $w'P$.  Indeed, for each~$j$ we
  fill in the column in block $(w'\mu)_j$ with color $(w'P)_j$, and since
  $(w'\mu)_j \geqslant (w'\mu)_{j+1}$ we only need to check the case when
  $(w'\mu)_j = (w'\mu)_{j+1}$ which means that
  $\langle \alpha_i, \lambda \rangle = -1$ and $w'^{-1} \alpha_i < 0$.  In
  that case $w'^{-1}(j) > w'^{-1}(j+1)$ which is equivalent to $(w'P)_j =
  c_{r+1-(w')^{-1}(j)} < c_{r+1-(w')^{-1}(j+1)} = (w'P)_{j+1}$.
\end{remark}

With the boundary spins fixed, a \textit{state} $\mathfrak{s}$
of the system assigns a \textit{spin} --- that is, an element of the spinset --- to every edge of the grid. See Figure~\ref{fig:monochrome} for an example.
We call the set of states with given boundary spins a system, denoted by $\mathfrak{S}_{\mu, \theta, w}$.
Given a state $\mathfrak{s} \in \mathfrak{S}_{\mu, \theta, w}$, each vertex $T^{cw}_z$ is assigned
a \textit{Boltzmann weight}, according to the values in Figure~\ref{fig:boltzmann_weights} where we identify each scolor $w_j$ with the integer $j$.
These depend on $c,w$ and $z$ for the vertex, as well as a parameter $v$ and a function $g$ of the integers modulo $n$ that satisfies
the following condition.
\begin{assumption}\label{assumptionga}
  The function $g$ satisfies $g(0)=-v$, and if $a$ is
  not congruent to $0$ modulo $n$, then $g(a)g(-a)=v$.
\end{assumption}

Note that this assumption is satisfied when $g$ is a certain non-archimedean Gauss sum and the parameter $v$ is associated
to the cardinality of the residue field. This will be important when connecting these lattice models to metaplectic
Whittaker functions in a later section. 

The \textit{Boltzmann weight} $\beta(\mathfrak{s})$ of the state $\mathfrak{s}$ is the product
of the Boltzmann weights at all vertices of $\mathfrak{s}$. The \textit{partition function}
$Z(\mathfrak{S})$ of a system $\mathfrak{S}$ is the sum of the Boltzmann weight $\beta(\mathfrak{s})$ over all
states $\mathfrak{s}$ in the system $\mathfrak{S}$ and will be considered as a function of the row parameters $\mathbf{z} = (z_1, \ldots, z_r)$.

\subsection{Equivalent model via fusion}
\label{sec:fusion}
In~\cite{BBBGIwahori} we employed a procedure we termed
\textit{fusion} to introduce additional equivalent versions of a lattice model. The terminology
was chosen owing to the resemblance with lattice model incarnations of the fusion procedure for tensor products
of quantum group modules and their subquotients as described by Kulish, Reshetikhin, and Sklyanin \cite{KulishReshetikhinSklyanin}, though
the precise algebraic explanation of the fusion process in~\cite{BBBGIwahori} remains a mystery. Nevertheless the foundations presented in \cite{BBBGIwahori}
allow us to conclude Yang-Baxter equations for ``fused'' models from similar equations for ``unfused'' models. Here, we will employ a similar procedure to
 the monochrome systems of the previous section.

\begin{figure}[ht]
\[\begin{array}{|c|c|}\hline
  \text{Block of unfused (monochrome) vertices} & \text{Equivalent fused vertex} \\
  \hline
\begin{tikzpicture}
\draw (0,1) to (4,1);
\draw (1,0) to (1,2);
\draw (3,0) to (3,2);
\node at (5,2) {$\cdots$};
\node at (5,1) {$\cdots$};
\node at (5,0) {$\cdots$};
\draw (6,1) to (8,1);
\draw (7,0) to (7,2);
\draw[fill=black] (1,1) circle (.1);
\draw[fill=black] (3,1) circle (.1);
\draw[fill=black] (7,1) circle (.1);
\node at (-.5,1) {$A$};
\node[scale=.9] at (1.5,0.6) {$T_z^{c_1w}$};
\node[scale=.9] at (3.5,0.6) {$T_z^{c_2w}$};
\node[scale=.9] at (7.5,0.6) {$T_z^{c_rw}$};
\node at (1,2.25) {$B_1$};
\node at (3,2.25) {$B_2$};
\node at (7,2.25) {$B_r$};
\node at (1,-.25) {$D_1$};
\node at (3,-.25) {$D_2$};
\node at (7,-.25) {$D_r$};
\node at (8.25,1) {$C$};
\node at (2,1.25) {$E_1$};
\node at (3.5,1.25) {$E_2$};
\node at (6.25,1.25) {$E_{r-1}$};
\end{tikzpicture}&
\begin{tikzpicture}
\draw (0,1) to (2,1);
\draw (1,0) to (1,2);
\draw[fill=black] (1,1) circle (.1);
\node at (-.5,1) {$A$};
\node[scale=.9] at (1.45,0.6) {$T_z^w$};
\node at (2.5,1) {$C$};
\node at (1,2.25) {$\mathbf{B}$};
\node at (1,-.25) {$\mathbf{D}$};
\end{tikzpicture}\\\hline\end{array}\]
\caption{Fusion. This procedure replaces a sequence of vertices by a single vertex. Left:
vertices in the unfused system. Right: a single replacement vertex in the fused system,
with replacement edges $\mathbf{B}$ and $\mathbf{D}$. The spinset for the fused vertical
edge $\mathbf{B}$ is the Cartesian product of the spinsets of the edges $B_i$, and similarly
for $\mathbf{D}$.}
\label{fig:fusion}
\end{figure}
 
The fusion procedure replaces a sequence of adjacent vertices in a model by a single vertex. We may preserve the weight of the state in the process by 
taking the Boltzmann weight of the fused vertex to be the partition function of the adjacent vertices that were fused together. In our case, we 
fuse a block of adjacent vertices $T_z^{c_1w},\ldots,T_z^{c_rw}$ in a given row of the grid over a complete set of $r$ colors $\{ c_i \}$ into a single
fused vertex which is labeled $T_z^w$. 
The fusion process is illustrated in Figure~\ref{fig:fusion}. In the notation of the figure, the spinset for the fused vertical
edge $\mathbf{B}$ is the Cartesian product of the spinsets of the edges $B_i$, and similarly
for $\mathbf{D}$. Thus specifying spins for the edges $A$, $\mathbf{B}$, $C$ and
$\mathbf{D}$ determines the spins for $B_i$ and $D_i$ for the (unfused) monochrome
configuration. In particular, the spin for $B_i$ is either of the form $c_iw$ or $+$ so the spinset for $\mathbf{B}$ is either $\mathbf{c} w$, where $\mathbf{c}$ is some subset of colors, or $+$. If there are spins for the edges $E_i$ that make the Boltzmann weights nonzero using the monochrome vertex configurations of Figure~\ref{fig:boltzmann_weights}, it is
easy to see that they are uniquely determined. So the partition function for these adjacent vertices is just the product of the Boltzmann weights at the unfused
vertices. Hence the Boltzmann weight of the fused vertex $T_z^w$ is just defined to be the product of these monochrome weights at $T_z^{c_1w},\ldots, T_z^{c_rw}$.

In Figure~\ref{fig:monochrome} we have a state of a monochrome system, and in Figure~\ref{fig:colorfused} we
have the corresponding state of the corresponding \textit{color-fused system}. Since the (unfused) monochrome grid has
$Nnr$ columns, the corresponding color-fused grid has $Nn$ columns.  Again the rows are numbered
from $1$ to $r$ in increasing order from top to bottom, and the columns are
labeled in descending order from $Nn-1$ to $0$ moving left to right.
Another example of fusion is shown in Figure~\ref{fig:colorfused-block}. 

\begin{figure}[htbp]
  \centering
\begin{tikzpicture}[scaled, xscale=0.5, every node/.append style={font=\normalsize}]
  \renewcommand\eps{0.075/0.5}
  \begin{scope}[scale = 2]
  % horizontal lines with labels
  \foreach \i in {1,2,3}{
    \node at (-0.5,3-\i)  [label={[label distance=1em]left:$z_\i$}] {};
  }
  
  % vertical lines with labels
  \foreach \j[evaluate=\j as \c using {int(mod(\j-1,3)+1)}, evaluate=\j as \w using {int(2-mod(int((\j-1)/3),3))}] in {1,...,18}{
    \draw (\j-1, -0.5) -- (\j-1,2.5) node [label={[align=center, label distance=1em]above:$c_{\c}$ \\ $w_{\w}$}] {};
  }

  % column color blocks
  \foreach \j[evaluate=\j as \col using {int(5-\j)}] in {0,...,5}{
    \draw[line width=1pt, decoration={brace},decorate] (3*\j-0.3,3.5) -- node[above=6pt] {block $\col$} (3*\j+2+0.3,3.5);
  }

  % vertex dots
  \foreach \i in {0,...,3}{
    \foreach \j in {1,...,18}{
      \node[spin, font=\normalsize] at (\j-1,\i-0.5) {$+$};
    }
  }

  \end{scope}

  \draw[scolored-path, sred] (30-\eps,5) -- (30-\eps,4) -- 
      (29, 4) node {\state{}{0}} --
      (27, 4) node {\state{}{0}} --
      (25, 4) node {\state{}{0}} --
      (23, 4) node {\state{}{0}} --
      (21, 4) node {\state{}{0}} --
      (19, 4) node {\state{}{0}} --
      (17, 4) node {\state{}{0}} -- (16-\eps,4) -- (16-\eps,2) --
      (15, 2) node {\state{}{0}} --
      (13, 2) node {\state{}{0}} --
      (11, 2) node {\state{}{0}} --
      (9, 2) node {\state{}{0}} --
      (7, 2) node {\state{}{0}} --
      (5, 2) node {\state{}{0}} --
      (3, 2) node {\state{}{0}} --
      (1, 2) node {\state{}{0}} --
      (-1, 2) node {\state{}{0}};

  \draw[scolored-path, sblue] (20-\eps,5) -- (20-\eps,0) --
      (19,0) node {\state{}{2}} --
      (17,0) node {\state{}{2}} --
      (15,0) node {\state{}{2}} --
      (13,0) node {\state{}{2}} --
      (11,0) node {\state{}{2}} --
      (9,0) node {\state{}{2}} --
      (7,0) node {\state{}{2}} --
      (5,0) node {\state{}{2}} --
      (3,0) node {\state{}{2}} --
      (1,0) node {\state{}{2}} --
      (-1,0) node {\state{}{2}};

  \draw[scolored-path, sgreen] (10-\eps,5) -- (10-\eps,4) --
      (9,4) node {\state{}{1}} --
      (7,4) node {\state{}{1}} --
      (5,4) node {\state{}{1}} --
      (3,4) node {\state{}{1}} --
      (1,4) node {\state{}{1}} --
      (-1,4) node {\state{}{1}};

  \draw[colored-path, red] (10,5) node[circles={{colored, red}, {scolored, sgreen}}] {\state{3}{1}} -- (10,4) -- 
      (11,4) node {\state{3}{}} --
      (13,4) node {\state{3}{}} --
      (15,4) node {\state{3}{}} -- (16,4) --
      (16,3) node[circles={{colored, red}, {scolored, sred}}] {\state{3}{0}} -- (16,2) --
      (17,2) node {\state{3}{}} --
      (17,2) node {\state{3}{}} --
      (19,2) node {\state{3}{}} --
      (21,2) node {\state{3}{}} --
      (23,2) node {\state{3}{}} --
      (25,2) node {\state{3}{}} --
      (27,2) node {\state{3}{}} --
      (29,2) node {\state{3}{}} --
      (31,2) node {\state{3}{}} --
      (33,2) node {\state{3}{}} --
      (35,2) node {\state{3}{}};

  \draw[colored-path, blue] (20, 5) node[circles={{colored, blue}, {scolored, sblue}}] {\state{2}{2}} --
      (20,3) node[circles={{colored, blue}, {scolored, sblue}}] {\state{2}{2}} --
      (20,1) node[circles={{colored, blue}, {scolored, sblue}}] {\state{2}{2}} --
      (20,0) --
      (21,0) node {\state{2}{}} --
      (23,0) node {\state{2}{}} --
      (25,0) node {\state{2}{}} --
      (27,0) node {\state{2}{}} --
      (29,0) node {\state{2}{}} --
      (31,0) node {\state{2}{}} --
      (33,0) node {\state{2}{}} --
      (35,0) node {\state{2}{}};

  \draw[colored-path, green] (30,5) node[circles={{colored, green}, {scolored, sred}}] {\state{1}{0}} -- (30,4) --
      (31,4) node {\state{1}{}} --
      (33,4) node {\state{1}{}} --
      (35,4) node {\state{1}{}};
  
\end{tikzpicture}

\caption{A state of a system $\mathfrak{S}_{\mu,w,\theta}$ using the expanded grid and monochrome weights with $\mu = (4,2,0)$, $\theta = (1,0,2)$, $w = s_1 s_2$, and $n = r = 3$.
  For readability we have labeled a vertical edge with the attribute $c_kw_l$ by $\protect\state[\protect\normalsize]{k}{l}$ where the scolor label is underlined by a dotted line in recognition of how the scolor-paths are styled. 
  The vertex in the column labeled $c_iw_j$ and row labeled $z_k$ is of type $T^{c_iw_j}_{z_k}$,
  with Boltzmann weights as in Figure~\ref{fig:boltzmann_weights}.
  Note how the color and scolor paths are mirrored.}
  \label{fig:monochrome}
\end{figure}

\begin{figure}[htbp]
\begin{tikzpicture}[scaled, every node/.append style={font=\normalsize}]
  \begin{scope}[scale = 2]
  % horizontal lines with labels
  \foreach \i in {1,2,3}{
    \node at (-0.5,3-\i)  [label={[label distance=1em]left:$z_\i$}] {};
  }
  
  % vertical lines with labels
  \foreach \j[evaluate=\j as \col using {int(6-\j)}, evaluate=\j as \w using {int(2-mod((\j-1),3))}] in {1,...,6}{
    \draw (\j-1, -0.5) -- (\j-1,2.5) node [label={[align=center, label distance=1em]above:{$\col$ \\[1em] $w_{\w}$}}] {};
  }

  % vertex dots
  \foreach \i in {0,...,3}{
    \foreach \j in {1,...,6}{
      \node[spin, font=\normalsize] at (\j-1,\i-0.5) {$+$};
    }
  }

  \end{scope}

  \draw[scolored-path, sred] (10-\eps,5) -- (10-\eps,4) -- 
      (9, 4) node {\state{}{0}} --
      (7, 4) node {\state{}{0}} --
      (5, 4) node {\state{}{0}} -- (4-\eps,4) -- (4-\eps,2) --
      (3, 2) node {\state{}{0}} --
      (1, 2) node {\state{}{0}} --
      (-1, 2) node {\state{}{0}};

  \draw[scolored-path, sblue] (6-\eps,5) -- (6-\eps,0) --
      (5,0) node {\state{}{2}} --
      (3,0) node {\state{}{2}} --
      (1,0) node {\state{}{2}} --
      (-1,0) node {\state{}{2}};

  \draw[scolored-path, sgreen] (2-\eps,5) -- (2-\eps,4) --
      (1,4) node {\state{}{1}} --
      (-1,4) node {\state{}{1}};

  \draw[colored-path, red] (2,5) node[circles={{colored, red}, {scolored, sgreen}}] {\state{3}{1}} -- (2,4) -- 
      (3,4) node {\state{3}{}} -- (4,4) --
      (4,3) node[circles={{colored, red}, {scolored, sred}}] {\state{3}{0}} -- (4,2) --
      (5,2) node {\state{3}{}} --
      (7,2) node {\state{3}{}} --
      (9,2) node {\state{3}{}} --
      (11,2) node {\state{3}{}};

  \draw[colored-path, blue] (6, 5) node[circles={{colored, blue}, {scolored, sblue}}] {\state{2}{2}} --
      (6,3) node[circles={{colored, blue}, {scolored, sblue}}] {\state{2}{2}} --
      (6,1) node[circles={{colored, blue}, {scolored, sblue}}] {\state{2}{2}} --
      (6,0) --
      (7,0) node {\state{2}{}} --
      (9,0) node {\state{2}{}} --
      (11,0) node {\state{2}{}};

  \draw[colored-path, green] (10,5) node[circles={{colored, green}, {scolored, sred}}] {\state{1}{0}} -- (10,4) --
      (11,4) node {\state{1}{}};
  
\end{tikzpicture}
   \caption{The color-fused system equivalent to
   Figure~\ref{fig:monochrome}. For reference we have added row and
   column numbers. Note that column numbers decrease from
   left to right, and row numbers increase from top to bottom. This
   system is $\mathfrak{S}_{\mu,\theta,w}$ where $\mu=(4,2,0)$,
   $\theta =(1,0,2)$, and $w=s_1s_2$.}
   \label{fig:colorfused}
\end{figure}

\begin{figure}[htpb]
  \centering
\begin{tikzpicture}[baseline={(0,0.7)}, scaled, xscale=0.6, every node/.append style={font=\normalsize}]
  \renewcommand\eps{0.075/0.6}
  \begin{scope}[scale = 2]
  % horizontal lines with labels
  \foreach \i in {1,2}{
    \node at (-0.5,2-\i)  [label={[label distance=1em]left:$z_\i$}] {};
  }
  
  % vertical lines with labels
  \foreach \j[evaluate=\j as \c using {int(mod(\j-1,3)+1)}, evaluate=\j as \w using {int(2-mod(int((\j-1)/3),3))}] in {1,...,3}{
    \draw (\j-1, -0.5) -- (\j-1,1.5) node [label={[align=center, label distance=1em]above:$c_{\c}$ \\ $w_{1}$}] {};
  }

  % column color blocks
  \foreach \j[evaluate=\j as \col using {int(5-\j)}] in {0}{
    \draw[line width=1pt, decoration={brace},decorate] (3*\j-0.3,2.5) -- node[above=6pt] {block $4$} (3*\j+2+0.3,2.5);
  }

  % vertex dots
  \foreach \i in {0,...,2}{
    \foreach \j in {1,...,3}{
      \node[spin, font=\normalsize] at (\j-1,\i-0.5) {$+$};
    }
  }

  \end{scope}

  \draw[scolored-path, sgreen] (4-\eps,3) -- (4-\eps,0) --
      (3,0) node {\state{}{1}} -- (2-\eps,0) -- (2-\eps,0) --
      (1,0) node {\state{}{1}} --
      (-1,0) node {\state{}{1}};

  \draw[scolored-path, sgreen] (5,2) node {\state{}{1}} -- 
      (3,2) node {\state{}{1}} -- (2-\eps,2) -- (2-\eps,-1);

  \draw[colored-path, red] (4,3) node[circles={{colored, red}, {scolored, sgreen}}] {\state{3}{1}} --
      (4,1) node[circles={{colored, red}, {scolored, sgreen}}] {\state{3}{1}} -- (4,0) -- 
      (5,0) node {\state{3}{}};

  \draw[colored-path, blue] (-1, 2) node {\state{2}{}} --
      (1,2) node{\state{2}{}} -- (2,2) --
      (2,1) node[circles={{colored, blue}, {scolored, sgreen}}] {\state{2}{1}} --
      (2,-1) node[circles={{colored, blue}, {scolored, sgreen}}] {\state{2}{1}};
  
\end{tikzpicture} \hspace{2em} $=$ \hspace{1.5em}
\begin{tikzpicture}[baseline={(0,0.7)}, scaled, every node/.append style={font=\normalsize}]
  \begin{scope}[scale = 2]

  % horizontal lines with labels
  \foreach \i in {1,2}{
    \node at (-0.5,2-\i)  [label={[label distance=1em]left:$z_\i$}] {};
  }
  
  % vertical lines with labels
  \foreach \j[evaluate=\j as \c using {int(mod(\j-1,3)+1)}, evaluate=\j as \w using {int(2-mod(int((\j-1)/3),3))}] in {1}{
    \node at (\j-1,1.5) [label={[align=center, label distance=1em]above:$w_{1}$}] {};
  }

  % column color blocks
  \foreach \j[evaluate=\j as \col using {int(5-\j)}] in {0}{
    \draw[line width=1pt, decoration={brace},decorate] (0,2.5) -- node[above=6pt] {column $4$} (0,2.5);
  }

  \end{scope}

  \draw[scolored-path, sgreen] (0,3) -- (0,1) -- (-2*\eps,1) -- (-2*\eps,0) -- (-1,0) node {\state{}{1}};

  \draw[scolored-path, sgreen] (1,2) node {\state{}{1}} -- (2*\eps,2) -- (2*\eps,1) -- (-\eps,1) -- (-\eps,-1);

  \draw[colored-path, red] (0,3) node[circles={{colored, red}, {scolored, sgreen}}] {\state{3}{1}} (\eps,3) -- (\eps,0) -- 
      (1,0) node {\state{3}{}};

  \draw[colored-path, blue]
      (-1,2) node{\state{2}{}} -- (-\eps,2) -- (-\eps,1) --
      (0,1) node[circles={{colored, red}, {colored, blue}, {scolored, sgreen}, {scolored, sgreen}}, label={[label distance=0.25em, font=\scriptsize]right:$(2\otimes\protect\state[\scriptsize]{}{1})\wedge (3\otimes\protect\state[\scriptsize]{}{1})$}] {} --
      (0,-1) node[circles={{colored, blue}, {scolored, sgreen}}] {\state{2}{1}};
  
\end{tikzpicture}
  \caption{An example of color-fused block where a resulting vertical state contains more than one color.
  We denote by $(2\otimes\protect\state[\footnotesize]{}{1})\wedge (3\otimes\protect\state[\footnotesize]{}{1})$ the fused edge carrying the colors $c_2$ and
  $c_3$ together with the scolor $w_1$. There are $2^{rn}$ possible states
  for the vertical fused edges, and these are in bijection with the exterior
  algebra over the free vector space space spanned by pairs consisting of one
  color and one supercolor, hence the notation. Note that the number of color
  paths and the number of scolor paths going through a vertical edge have to be
  equal.}\label{fig:colorfused-block}
\end{figure}

In Figure~\ref{fig:fully_color_fused_boltzmann_weights} we write the Boltzmann weights of the color-fused model where a vertical edge is described by a set $\Sigma$ of colors.
At a color-fused vertex $T^w_z$ the set determines which columns $(c,w)$ with $c \in \Sigma$ that are occupied in the corresponding monochrome block.
If the set $\Sigma$ is empty, we denote the vertical edge by a plus sign.
Note that these weights precisely match the weights in Figure~12 of~\cite{BBBGIwahori}, when one allows the existence of only one scolor $w$. 
The appearance of the term $g(w-y)^{|\Sigma|}$ in the bottom right entry of the table is surprising, 
but the exponent on the Gauss sum indeed checks 
out in the matching with Figure~12 of \cite{BBBGIwahori} because
$g(w-y) = g(0) = -v$ in that degenerate setting.

\tikzstyle{fused-path}=[line width=5pt, shorten >=-3pt, shorten <=-3pt, every node/.append style={label distance=-5pt}]
\begin{figure}[ht]
{\scriptsize
\[\begin{array}{|c|c|c|c|c|c|}
\hline
\begin{tikzpicture}[scaled, font=\normalsize]
  \draw[fused-path] (0,1) node[label=above:{$\Sigma$}] {} -- (0,-1) node[label=below:{$ \Sigma$}]{};
  \draw[colored-path] (-1,0) node {$c_i$} -- (1,0) node {$c_i$};
  \node[halo-rect, inner sep=2pt] at (0,0) {$T_z^{w}$};
\end{tikzpicture}
 &
\begin{tikzpicture}[scaled, font=\normalsize]
  \draw[fused-path] (0,1) node[label=above:{$\tensor*{\Sigma}{*^-_i^+_j}$}] {} -- (0,-1) node[label=below:{$\Sigma$}] {};
  \draw[colored-path] (-1,0) node {$c_i$} -- (1,0) node {$c_j$};
  \node[halo-rect, inner sep=2pt] at (0,0) {$T_z^{w}$};
\end{tikzpicture}
  &
\begin{tikzpicture}[scaled, font=\normalsize]
  \draw[fused-path] (0,1) node[label=above:{$ \tensor*{\Sigma}{*^-_j^+_i} $}] {} -- (0,-1) node[label=below:{$ \Sigma $}] {};
  \draw[colored-path] (-1,0) node {$c_j$} -- (1,0) node {$c_i$};
  \node[halo-rect, inner sep=2pt] at (0,0) {$T_z^{w}$};
\end{tikzpicture}
\\
\hline
z v^{|\Sigma_{[i+1,r]}|}
& 
(1-v)z (-v)^{|\Sigma_{[i+1,j-1]}|} v^{|\Sigma_{[j+1,r]}|}
& 
0 \\
\hline\hline
\begin{tikzpicture}[scaled, font=\normalsize]
  \draw[fused-path] (0,-1) node[label=below:{$ \Sigma $}] {} -- (0,1) node[label=above:{$ \Sigma_{i}^+ $}] {};
  \draw[colored-path] (1,0) node {$c_i$} -- (0,0);
  \draw[scolored-path] (0,0) -- (-1,0) node {$w$};
  \node[halo-rect, inner sep=2pt] at (0,0) {$T_z^{w}$};
\end{tikzpicture}
&
\begin{tikzpicture}[scaled, font=\normalsize]
  \draw[fused-path] (0,1) node[label=above:{$\Sigma_i^- $}] {} -- (0,-1) node[label=below:{$ \Sigma $}] {};
  \draw[colored-path] (-1,0) node {$c_i$} -- (0,0);
  \draw[scolored-path] (0,0) -- (1,0) node {$w$};
  \node[halo-rect, inner sep=2pt] at (0,0) {$T_z^{w}$};
\end{tikzpicture}
 &
 \begin{tikzpicture}[scaled, font=\normalsize]
   \draw[fused-path] (0,1) node[label=above:{$ \Sigma $}] {} -- (0,-1) node[label=below:{$ \Sigma $}] {};
  \draw[scolored-path] (-1,0) node {$y$} -- (1,0) node {$y$};
  \node[halo-rect, inner sep=2pt] at (0,0) {$T_z^{w}$};
\end{tikzpicture}
\\\hline
(-v)^{|\Sigma_{[1,i-1]}|} v^{|\Sigma_{[i+1,r]}|} &
(1-v)z(-v)^{|\Sigma_{[i+1,r]}|} & 
g(y-w)^{|\Sigma_{[1,r]}|} 
\\
\hline
\end{array}
\]}
\caption{The Boltzmann weights for the color-fused model with scolor $w$, with vertex labeled $T_z^w$. 
Here $\Sigma$ is an arbitrary set of colors (which may be empty), and $| \Sigma |$ denotes the cardinality of $\Sigma$. Moreover, 
$\Sigma_{[i,j]}$ denotes the colors in $\Sigma$ with indices in the interval $[i,j]$.
We introduce the operation $\Sigma_i^+ := \Sigma \cup \{ c_i \}$ to be used only if $c_i \not\in \Sigma$ and similarly for $\Sigma_i^{-} \cup \{ c_i \} = \Sigma$ and $\tensor*{\Sigma}{*^-_i^+_j} \cup \{ c_i \} = \Sigma^+_j$.
We denote by $c_i, c_j$ colors with $i < j$ and $y$ is a scolor (which may be the
same or different from $w$).} 
\label{fig:fully_color_fused_boltzmann_weights}
\end{figure}

We will now translate the boundary conditions from the monochrome system to the color-fused system where a block of columns in the former becomes a single column in the latter.
Since they are equivalent, the system will still be denoted $\mathfrak{S}_{\mu,\theta,w}$ where $\mu\in\mathbb{Z}^r$ with non-negative parts, $\theta$ is an element in $(\mathbb{Z} / n \mathbb{Z})^r$ and $w$ a Weyl group element in $S_r$.
Here $r$ is the number of rows in the grid with $Nn$ color-fused columns.
As before, along the bottom boundary of the color-fused model, all column vertices are assigned $+$, meaning the empty attribute.
Along the right-hand border, the horizontal boundary edge in row~$i$ receives color~$c_{r+1-w^{-1}(i)}$, and along the left-hand border, the horizontal boundary edge in row~$i$ receives scolor~$\theta_i$.

This leaves the top boundary.
As before, the vertical boundary edges receive attributes according to the parts of $\mu$ so that each color appears exactly once along the top boundary edges.
However, more than one color may appear in a single edge by gathering the colors in a block of columns in the monochrome system into a single column in the color-fused system.
Specifically, column $i$ gets attributes $\mathbf{c} w_j$ where $i \equiv j$ (mod $n$) and $\mathbf{c}$ is the subset of colors $\{ c_{r+1-j_1}, \ldots, c_{r+1-j_\ell} \}$ such that the parts $\mu_{j_1} = \cdots = \mu_{j_\ell} = i$, and if this subset is empty we denote the edge by $+$.

Recall that according to Remark~\ref{rem:almost-dominant}, we may either parametrize the top
boundary by $\mu\in\mathbb{Z}^r$ or by a pair of elements $(w', \lambda)$ where
$w'$ is in $S_r$ and $\lambda$ is a $w'$-almost-dominant weight
(Definition~\ref{def:almost_dominant}) such that
$w'(\mu) = \lambda + \rho$ with $\rho = (r-1, \ldots, 1, 0)$.

We have seen that color-fused systems are equivalent to monochrome systems by using a fusion process to collapse $r$ consecutive columns over a full set of colors. It is natural to make a further collapse of the color-fused systems by means of a fusion process over a full set of $n$ scolors, resulting in a grid with $N$ columns and $r$ rows. We refer to this third, equivalent collection of systems as the \emph{fully-fused} model. See Figure~\ref{fig:fullyfused} for the result of performing this collapse on our running example. Again, none of the partition functions is altered by the fusion process, as the Boltzmann weights are taken to be the product of Boltzmann weights in the un-fused model. The fully-fused model is best for making connections to associated quantum group modules.

\begin{remark}
\label{colordet}
Once a state of the system is fixed, colored and scolored lines emerge.
The colored lines begin at the top edge and end at the right edge;
the scolored lines begin at the top edge and end at the left edge. 
In the monochrome and color-fused systems, if the colored paths are specified,
then the scolored paths are determined, so (if an allowable configuration results) the state is uniquely determined. 
This follows from
the facts (apparent from the Boltzmann weights) that every
vertical edge carries a color and a scolor, and every
horizontal edge carries a color or a scolor (but not both);
that the number of edges at a vertex carrying a
given scolor (or color) is even; and that (in the monochrome
and color-fused systems) the scolor carried by any vertical
edge is predetermined. From these properties of the Boltzmann
weights, it is obvious that each scolored line beginning at the top edge (and
specified by boundary conditions) can be traced uniquely if the colored lines
are known. It is similarly true in the monochrome systems
that the scolored lines determine the state.
\end{remark}

\begin{figure}[ht]
  \begin{tikzpicture}[scaled, large/.style={minimum size=20pt}]
  \draw (3,-5) -- (3,-6) node[spin, large] {$+$};
  \draw[colored, green] (3+\eps, 0) -- (3+\eps,-1) -- (4,-1) node [spin] {$\state{1}{}$};
  \draw[scolored, sblue] (3-\eps,0) -- (3-\eps,-5) -- (2,-5) node [spin] {$\state{}{2}$} -- (0, -5) node [spin] {$\state{}{2}$};
  \draw[scolored, sred] (3-2*\eps,0) -- (3-2*\eps,-1) -- (2,-1) node [spin] {$\state{}{0}$} -- (1+\eps,-1) -- (1+\eps,-2) -- (1-\eps,-2) -- (1-\eps,-3) -- (0,-3) node [spin] {$\state{}{0}$};
  \draw[colored, blue] (3, 0) node [spin, large, label={[black]right:$\scriptstyle \state{2}{2} \,\wedge\, \state{1}{0}$}, circles={{colored, blue}, {colored, green}, {scolored, sblue}, {scolored, sred}}] {} --
    (3,-2) node [spin, large, circles={{colored, blue}, {scolored, sblue}}] {$\state{2}{2}$} --
    (3,-4) node [spin, large, circles={{colored, blue}, {scolored, sblue}}] {$\state{2}{2}$} --
    (3,-5) -- (4,-5) node[spin] {$\state{2}{}$};

  \draw (1,-3) -- (1,-4) node[spin, large] {$+$} -- (1,-6) node [spin, large] {$+$};
  \draw[scolored, sgreen] (1-\eps,0) -- (1-\eps,-1) -- (0,-1) node [spin] {$\state{}{1}$};
  \draw[colored, red] (1,0) node[spin, large, circles={{colored, red}, {scolored, sgreen}}] {$\state{3}{1}$} -- 
    (1,-2) node[spin, large, circles={{colored, red}, {scolored, sred}}] {$\state{3}{0}$} --
    (1,-3) -- (2,-3) node[spin] {$\state{3}{}$} -- (4,-3) node[spin] {$\state{3}{}$};
\end{tikzpicture}
   \caption{The state of the fully-fused system that corresponds to the
   states in Figures~\ref{fig:monochrome} and~\ref{fig:colorfused}.
 \label{fig:fullyfused}}
\end{figure}

\subsection{The partition function for the ground state system}
\label{sec:ground-state}
\begin{lemma}
  \label{lem:lattice-ground-state}
  Let $\theta \in (\mathbb{Z}/n\mathbb{Z})^r$ and let $\mu\in\mathbb{Z}_{\geqslant0}^r$.
  Then there is a unique $w' \in W$ such that the colors for the top boundary edges for the system $\mathfrak{S}_{\mu, \theta, w'}$ in the monochrome picture is $w'P$ from left to right where $P = (c_r, c_{r-1}, \ldots, c_1)$ which is equivalent to $\lambda :=
w' \mu - \rho$ being $w'$-almost dominant (Definition~\ref{def:almost_dominant}).
  Furthermore,
  \begin{equation}
    Z\bigl(\mathfrak{S}_{\mu, \theta, w'}\bigr)(\mathbf{z}) = 
    \begin{cases}
      v^{\ell(w')} \mathbf{z}^{\lambda + \rho} & \text{if } \theta \equiv \lambda + \rho \bmod{n} \\
      0 & \text{otherwise.}
    \end{cases} 
  \end{equation}
\end{lemma}

\begin{proof}
  We will use the monochrome version of the system and prove that the system $\mathfrak{S}_{\mu,\theta, w'}$ contains only at most a single state, which we call a \emph{ground state}.

  Since the colors in $P$ are unique and the colors of the top boundary are determined by $\mu$ there is a unique $w'$ such that the order of colors is $w'P$.
  That this is equivalent to $\lambda$ being $w'$-almost dominant follows from Remark~\ref{rem:almost-dominant} which is based on Proposition~7.2 of~\cite{BBBGIwahori}.
  Note that $\lambda + \rho = w' \mu$ is a dominant weight which means that we fill in top boundary columns at the blocks $(\lambda + \rho)_i$ with color $(w'P)_i$ in the order from left to right increasing with $i$. 

  From Remark~\ref{colordet} we have that a state is determined by its colored paths (or, alternatively its scolored paths).
  By the assumptions, the colors on the top boundary from left to right is the same as the colors on the right boundary from top to bottom.
  We will now follow the possible colored paths from the top boundary to the right boundary and argue that there is at most one admissible state for $\mathfrak{S}_{\mu, \theta, w'}$.
  From the admissible vertex configurations in Figure~\ref{fig:boltzmann_weights} we see that the colored paths can only go rightwards and downwards.
  This means that the path starting at the left-most color on the top boundary must immediately go straight to the right boundary.
  Now the path starting at the second left-most color needs to go down and right to the right boundary on the second row.
  However, all the horizontal edges on the first row to the right of the second left-most color on the top boundary are occupied by the previous colored path.
  Thus it must first go straight down to the second row and then straight right to the right boundary.
  Continuing in this fashion we see that all colored paths are L-shaped in this way, and that there is at most one admissible state since an admissible state is determined by its colored paths.
  
  We will now argue that we get an admissible state if and only if $\theta_i \equiv (\lambda + \rho)_i \bmod{n}$ at all rows $i$.
  An example of such a state, which we call a ground state, is shown in Figure~\ref{fig:ground-state}.
  As seen in the vertex configurations of Figure~\ref{fig:conventions}, the scolor paths, starting from the top boundary and ending at the left boundary, can only go leftwards and downwards, and the latter only when together with a color path.
  If we consider the left-most scolor path starting from the top boundary we see the it has to follow the first color path to the first row and immediately go straight to the left boundary.
  The second scolor path needs to follow the second color path until the second row and then go straight to the left boundary.
  Continuing in this fashion we see that the scolored paths all have a vertically mirrored L-shape and that we clearly get an admissible state of the system provided it satisfies the boundary conditions.
  The left boundary scolor at row $i$ matches the scolor of the $i$-th left-most scolor on the top boundary which is $(\lambda + \rho)_i \bmod n$ by construction and thus the boundary conditions are satisfied if and only if $\theta \equiv \lambda + \rho \bmod n$.

  It remains to compute the weight of this admissible state.
  The involved vertices are of types $\texttt{a}_1$, $\texttt{c}_2$, $\texttt{a}_2$ and $\texttt{b}_2$ shown in Figure~\ref{fig:boltzmann_weights} where only the latter two make any non-trivial contributions.
  For both type $\texttt{a}_2$ and $\texttt{b}_2$ we obtain, at row $i$, a factor of $z_i$ whenever the vertex color equals the horizontal edge color.
  This happens once for every color block of $r$ columns in the fully expanded system, which means that we get a contribution of $z_i^{(\lambda+\rho)_i}$.
  The only remaining non-trivial contributions come from the $c>a$ case in type $\texttt{a}_2$ each with a factor of~$v$.
  These occur whenever we have a crossing of two colored paths where the color $a$ of the horizontal edge is less than the color $c$ of the vertical edge.
  The ordering of the colors on the top and right boundaries is given by $w'P$, and thus the number of such crossings is $\ell(w')$.
  Hence, the weight of the unique admissible state, called the ground state, is $v^{\ell(w')}\mathbf{z}^{\lambda + \rho}$ which concludes the proof.
\end{proof}

\begin{figure}[htbp]
  \centering  
\begin{tikzpicture}[scaled, xscale=0.5, every node/.append style={font=\normalsize}]
  \renewcommand\eps{0.075/0.5}
  \begin{scope}[scale = 2]
  % horizontal lines with labels
  \foreach \i in {1,2,3}{
    \node at (-0.5,3-\i)  [label={[label distance=1em]left:$z_\i$}] {};
  }
  
  % vertical lines with labels
  \foreach \j[evaluate=\j as \c using {int(mod(\j-1,3)+1)}, evaluate=\j as \w using {int(1-mod(int((\j-1)/3),2))}] in {1,...,12}{
    \draw (\j-1, -0.5) -- (\j-1,2.5) node [label={[align=center, label distance=1em]above:$c_{\c}$ \\ $w_{\w}$}] {};
  }

  % column color blocks
  \foreach \j[evaluate=\j as \col using {int(3-\j)}] in {0,...,3}{
    \draw[line width=1pt, decoration={brace},decorate] (3*\j-0.3,3.5) -- node[above=6pt] {block $\col$} (3*\j+2+0.3,3.5);
  }

  % vertex dots
  \foreach \i in {0,...,3}{
    \foreach \j in {1,...,12}{
      \node[spin, font=\normalsize] at (\j-1,\i-0.5) {$+$};
    }
  }

  \end{scope}

  \draw[scolored-path, sred] (10-\eps,5) -- (10-\eps,2) -- 
      (9, 2) node {\state{}{0}} --
      (7, 2) node {\state{}{0}} --
      (5, 2) node {\state{}{0}} --
      (3, 2) node {\state{}{0}} --
      (1, 2) node {\state{}{0}} --
      (-1, 2) node {\state{}{0}};

  \draw[scolored-path, sred] (18-\eps,5) -- (18-\eps,0) --
      (17,0) node {\state{}{0}} --
      (15,0) node {\state{}{0}} --
      (13,0) node {\state{}{0}} --
      (11,0) node {\state{}{0}} --
      (9,0) node {\state{}{0}} --
      (7,0) node {\state{}{0}} --
      (5,0) node {\state{}{0}} --
      (3,0) node {\state{}{0}} --
      (1,0) node {\state{}{0}} --
      (-1,0) node {\state{}{0}};

  \draw[scolored-path, sgreen] (2-\eps,5) -- (2-\eps,4) --
      (1,4) node {\state{}{1}} --
      (-1,4) node {\state{}{1}};

  \draw[colored-path, blue] (2,5) node[circles={{colored, blue}, {scolored, sgreen}}] {\state{2}{1}} -- (2,4) -- 
      (3,4) node {\state{2}{}} --
      (5,4) node {\state{2}{}} --
      (7,4) node {\state{2}{}} --
      (9,4) node {\state{2}{}} --
      (11,4) node {\state{2}{}} --
      (13,4) node {\state{2}{}} --
      (15,4) node {\state{2}{}} --
      (17,4) node {\state{2}{}} --
      (19,4) node {\state{2}{}} --
      (21,4) node {\state{2}{}} --
      (23,4) node {\state{2}{}};

  \draw[colored-path, red] (10, 5) node[circles={{colored, red}, {scolored, sred}}] {\state{3}{0}} --
      (10,3) node[circles={{colored, red}, {scolored, sred}}] {\state{3}{0}} --
      (10,2) --
      (11,2) node {\state{3}{}} --
      (13,2) node {\state{3}{}} --
      (15,2) node {\state{3}{}} --
      (17,2) node {\state{3}{}} --
      (19,2) node {\state{3}{}} --
      (21,2) node {\state{3}{}} --
      (23,2) node {\state{3}{}};

  \draw[colored-path, green] (18,5) node[circles={{colored, green}, {scolored, sred}}] {\state{1}{0}} --
      (18,3) node[circles={{colored, green}, {scolored, sred}}] {\state{1}{0}} -- 
      (18,1) node[circles={{colored, green}, {scolored, sred}}] {\state{1}{0}} --
      (18,0) --
      (19,0) node {\state{1}{}} --
      (21,0) node {\state{1}{}} --
      (23,0) node {\state{1}{}};
  
\end{tikzpicture}  
\caption{The unique state, the so called ground state, for the system $\mathfrak{S}_{\mu, \theta, w'}$ with $\mu = (2,3,0)$, $w' = s_1$, $\theta = (1,0,0)$ with $n=2$ and $r = 3$.
Note that $w'$ is such that the colors of the top boundary is of the order $w'P$ from left to right.}
\label{fig:ground-state}
\end{figure}

\subsection{The R-matrix and the Yang-Baxter equation}
\label{sec:R-matrix}

The Yang-Baxter equation may be expressed as an identity
of partition functions consisting of three vertices, two of which are of the kind discussed in the previous section denoted $T$, and one of a different kind denoted $R$ which involves only horizontal edges as illustrated in Figure~\ref{fig:ybelr}.
A solution to the (RTT) Yang-Baxter equation given sets of weights for the $T$-vertices, is a third set of weights for the $R$-vertices so that the partition functions agree. The Boltzmann weights at
each vertex encode an endomorphism of a tensor product of vector spaces, so we
may speak of a solution to the Yang-Baxter equation as a set of Boltzmann weights or
equivalently as encoded in an ``R-matrix'' describing the associated endomorphism, and we use these terms interchangeably.
Our ultimate goal in this section is a solution to the Yang-Baxter equation for the fully-fused models of
the previous section using the weights at vertices $T_{z_i}$ and $T_{z_j}$.

We accomplish this by describing solutions to auxiliary Yang-Baxter-like equations for the monochrome vertices labeled 
$T_{z_i}^{cw}$ and $T_{z_j}^{cw}$, described in the following theorem. By a similar argument as in Lemma~5.4 of \cite{BBBGIwahori}, such
auxiliary solutions guarantee a solution of the Yang-Baxter equation for the color-fused and fully-fused models.

\begin{figure}[htbp]
\[\begin{tikzpicture}[scaled, baseline=(current bounding box.center),scale=1.1,every node/.style={scale=.9}]
  \draw (0,1) to [out = 0, in = 180] (2,3) to (4,3);
  \draw (0,3) to [out = 0, in = 180] (2,1) to (4,1);
  \draw (3,0) to (3,4);
  \draw[fill=white] (0,1) circle (.3);
  \draw[fill=white] (0,3) circle (.3);
  \draw[fill=white] (3,4) circle (.3);
  \draw[fill=white] (4,3) circle (.3);
  \draw[fill=white] (4,1) circle (.3);
  \draw[fill=white] (3,0) circle (.3);
  \node at (0,1) {$\sigma$};
  \node at (0,3) {$\tau$};
  \node at (3,4) {$\beta$};
  \node at (4,3) {$\theta$};
  \node at (4,1) {$\rho$};
  \node at (3,0) {$\alpha$};
  \draw[fill=white] (2,1) circle (.3);
  \node at (2,1) {$\mu$};
  \draw[fill=white] (2,3) circle (.3);
  \node at (2,3) {$\nu$};
  \draw[fill=white] (3,2) circle (.3);
  \node at (3,2) {$\gamma$};
  \path[fill=white] (1,2) circle (.35);
  \node at (1,2) {$R_{z_i,z_j}^{cw}$};
  \path[fill=white] (3,3) circle (.35);
  \node at (3,3) {$T_{z_i}^{cw}$};
  \path[fill=white] (3,1) circle (.35);
  \node at (3,1) {$T^{cw}_{z_j}$};
\end{tikzpicture}
\qquad\qquad
\begin{tikzpicture}[scaled, baseline=(current bounding box.center),scale=1.1,every node/.style={scale=.9}]
  \draw (0,1) to (2,1) to [out = 0, in = 180] (4,3);
  \draw (0,3) to (2,3) to [out = 0, in = 180] (4,1);
  \draw (1,0) to (1,4);
  \draw[fill=white] (0,1) circle (.3);
  \draw[fill=white] (0,3) circle (.3);
  \draw[fill=white] (1,4) circle (.3);
  \draw[fill=white] (4,3) circle (.3);
  \draw[fill=white] (4,1) circle (.3);
  \draw[fill=white] (1,0) circle (.3);
  \node at (0,1) {$\sigma$};
  \node at (0,3) {$\tau$};
  \node at (1,4) {$\beta$};
  \node at (4,3) {$\theta$};
  \node at (4,1) {$\rho$};
  \node at (1,0) {$\alpha$};
  \path[fill=white] (3,2) circle (.35);
  \node at (3,2) {$R_{z_i,z_j}^{c'w'}$};
  \path[fill=white] (1,1) circle (.35);
  \node at (1,1) {$T^{cw}_{z_i}$};
  \path[fill=white] (1,3) circle (.35);
  \node at (1,3) {$T^{cw}_{z_j}$};
  \draw[fill=white] (1,2) circle (.3);
  \draw[fill=white] (2,3) circle (.3);
  \draw[fill=white] (2,1) circle (.3);
  \node at (1,2) {$\delta$};
  \node at (2,3) {$\psi$};
  \node at (2,1) {$\phi$};
  \end{tikzpicture}\]
  \caption{The auxiliary Yang-Baxter equations. These are miniature systems with
  fixed boundary spins $\alpha,\sigma,\tau,\beta,\theta,\rho$. The
  partition function is the sum over the interior spins $\mu,\nu,\gamma,\delta,\phi,\psi$ which are
  allowed to vary. Here $(c',w')$ is the successor to $(c,w)$ defined
  in (\ref{ccsuccessor}). Theorem~\ref{thm:ybemc} asserts that the
  partition functions of these two systems are equal.}
  \label{fig:ybelr}
\end{figure}

\begin{theorem}
    \label{thm:ybemc}
    Let $c$ and $w$ be a color and scolor, respectively. Let $\sigma,\tau,\theta,\rho$
    be elements of the horizontal spinset \eqref{horizontalspinset}, and let
    $\alpha,\beta$ be spins of vertical type $cw$, that is, elements of the
    two-element set $\{+,cw\}$. Let $(c',w')$ be the successor to $(c,w)$ according to
    \eqref{ccsuccessor}. Then using the Boltzmann weights for the vertices $T^{cw}_{z_i}$ from Figure~\ref{fig:boltzmann_weights} and $R_{z_i,z_j}^{cw}$ from Figure~\ref{fig:genrmatrix}, the partition functions of the two
    systems in Figure~\ref{fig:ybelr} are equal.
\end{theorem}

The proof is postponed to later in this section.

\begin{figure}[htbp]
{\scriptsize
\[\begin{array}{|c|c|c|c|}
\hline
\begin{tikzpicture}[Rscaled, font=\normalsize]
  \draw[colored-path] (-1,-1) node {$a$} -- (0,0) -- (1,-1) node {$a$};
  \draw[colored-path] (-1,+1) node {$a$} -- (0,0) -- (1,+1) node {$a$};
  \node[halo] at (0,0) {$R_{z_i,z_j}^{cw}$}; 
\end{tikzpicture}
&
\begin{tikzpicture}[Rscaled, font=\normalsize]
  \draw[colored-path] (-1,-1) node {$a$} -- (0,0) -- (1,+1) node {$a$};
  \draw[colored-path] (-1,+1) node {$b$} -- (0,0) -- (1,-1) node {$b$};
  \node[halo] at (0,0) {$R_{z_i,z_j}^{cw}$}; 
\end{tikzpicture}
&
\begin{tikzpicture}[Rscaled, font=\normalsize]
  \draw[colored-path] (-1,-1) node {$a$} -- (0,0) -- (1,-1) node {$a$};
  \draw[colored-path] (-1,+1) node {$b$} -- (0,0) -- (1,+1) node {$b$};
  \node[halo] at (0,0) {$R_{z_i,z_j}^{cw}$}; 
\end{tikzpicture}
&
\begin{tikzpicture}[Rscaled, font=\normalsize]
  \draw[scolored-path] (-1,-1) node {$x$} -- (0,0) -- (1,+1) node {$x$};
  \draw[scolored-path] (-1,+1) node {$x$} -- (0,0) -- (1,-1) node {$x$};
  \node[halo] at (0,0) {$R_{z_i,z_j}^{cw}$}; 
\end{tikzpicture}
\\\hline
z^n_i-vz^n_j &
\begin{array}{ll} v(z^n_i-z^n_j) & a<b\\z^n_i-z^n_j & a>b\end{array} &
\begin{array}{ll}
       (1 - v) z^n_i & a < b < c\\
       (1 - v) z^n_i & c \leqslant a < b\\
       (1 - v) z^{n-1}_i z_j & a < c \leqslant b\\
       (1 - v) z^{n-1}_j z_i & a \geqslant c > b\\
       (1 - v) z^n_j & a > b \geqslant c\\
       (1 - v) z^n_j & c > a > b
     \end{array}
&
z^n_j-v z^n_i 
\\\hline\hline
\begin{tikzpicture}[Rscaled, font=\normalsize]
  \draw[scolored-path] (-1,-1) node {$x$} -- (0,0) -- (1,+1) node {$x$};
  \draw[scolored-path] (-1,+1) node {$y$} -- (0,0) -- (1,-1) node {$y$};
  \node[halo] at (0,0) {$R_{z_i,z_j}^{cw}$}; 
\end{tikzpicture}
&
\begin{tikzpicture}[Rscaled, font=\normalsize]
  \draw[scolored-path] (-1,-1) node {$x$} -- (0,0) -- (1,-1) node {$x$};
  \draw[scolored-path] (-1,+1) node {$y$} -- (0,0) -- (1,+1) node {$y$};
  \node[halo] at (0,0) {$R_{z_i,z_j}^{cw}$}; 
\end{tikzpicture}
&
\begin{tikzpicture}[Rscaled, font=\normalsize]
  \draw[scolored-path] (-1,-1) node {$x$} -- (0,0) -- (1,+1) node {$x$};
  \draw[colored-path] (-1,+1) node {$a$} -- (0,0) -- (1,-1) node {$a$};
  \node[halo] at (0,0) {$R_{z_i,z_j}^{cw}$}; 
\end{tikzpicture}
&
\begin{tikzpicture}[Rscaled, font=\normalsize]
  \draw[scolored-path] (-1,+1) node {$x$} -- (0,0) -- (1,-1) node {$x$};
  \draw[colored-path] (-1,-1) node {$a$} -- (0,0) -- (1,+1) node {$a$};
  \node[halo] at (0,0) {$R_{z_i,z_j}^{cw}$}; 
\end{tikzpicture}
\\ \hline
g(y-x)(z^n_i-z^n_j) &
\begin{array}{ll}(1-v)z^n_i {(\tfrac{z_j}{z_i})^{y-x}} & y>x\\(1-v)z^n_j {(\tfrac{z_j}{z_i})^{y-x}} & y<x\end{array} &
v(z^n_i-z^n_j) & 
z^n_i-z^n_j
\\\hline\hline
\multicolumn{2}{|c|}{
\begin{tikzpicture}[Rscaled, font=\normalsize]
  \draw[scolored-path] (-1,+1) node {$x$} -- (0,0) -- (1,+1) node {$x$};
  \draw[colored-path] (-1,-1) node {$a$} -- (0,0) -- (1,-1) node {$a$};
  \node[halo] at (0,0) {$R_{z_i,z_j}^{cw}$}; 
\end{tikzpicture}
}
& 
\multicolumn{2}{c|}{
\begin{tikzpicture}[Rscaled, font=\normalsize]
  \draw[scolored-path] (-1,-1) node {$x$} -- (0,0) -- (1,-1) node {$x$};
  \draw[colored-path] (-1,+1) node {$a$} -- (0,0) -- (1,+1) node {$a$};
  \node[halo] at (0,0) {$R_{z_i,z_j}^{cw}$}; 
\end{tikzpicture}
}
\\\hline
\multicolumn{2}{|c|}{\begin{array}{ll}
    (1-v)z^n_i{(\tfrac{z_j}{z_i})^{x-w}} & a<c, w\leqslant x\\
    (1-v)z^n_j{(\tfrac{z_j}{z_i})^{x-w}}& a<c, w>x\\
    (1-v)z^n_i {(\tfrac{z_j}{z_i})^{x-w-1}} & a\geqslant c,w<x\\
    (1-v)z^n_j {(\tfrac{z_j}{z_i})^{x-w-1}} & a\geqslant c,w\geqslant x\\
\end{array}} &
\multicolumn{2}{c|}{\begin{array}{ll}
    (1-v)z^n_j{(\tfrac{z_i}{z_j})^{x-w}} & a<c, w\leqslant x\\
(1-v)z^n_i{(\tfrac{z_i}{z_j})^{x-w}} & a<c, w>x\\
(1-v)z^n_j {(\tfrac{z_i}{z_j})^{x-w-1}} & a\geqslant c,w<x\\
(1-v)z^n_i {(\tfrac{z_i}{z_j})^{x-w-1}} & a\geqslant c,w\geqslant x\\\end{array}}
\\\hline
\end{array}
\]}
\caption{The R-matrix depending on a color $c$ and charge $w$, as well
as a pair of spectral parameters $z_i$ and $z_j$.
Here $a$ and $b$ are distinct colors and $x$ and $y$ are distinct scolors. We
allow $a$ or $b$ to equal $c$, and $w$ to equal $x$ or $y$.}
\label{fig:genrmatrix}
\end{figure}

Note that the R-matrix is realized as a vertex that is attached to four edges
of horizontal type, with the corresponding admissible spinsets consisting of a single attribute (one color or one scolor). 

Ultimately the proof of Theorem~\ref{thm:ybemc} may be reduced to a computation that
can be checked using a computer. But it is not \textit{a priori} clear how
this can be checked for all $n$ and $r$ and for all choices of the function $g$ in a finite manner.  Thus our main
task is to show how the proof can be reduced to a finite number of equations that need to be checked.  We
have taken $g$ to be an arbitrary function of the integers modulo $n$, subject
to Assumption~\ref{assumptionga}.  Now let us define a particular choice
for~$g$:

\begin{equation}\label{eq:gquantumgroup}
g_0(a)=\left\{\begin{array}{ll} -v & \text{if $a\equiv 0$ mod $n$,}\\
-\sqrt{v} &\text{otherwise.}\end{array}\right.
\end{equation}
Here $\sqrt{v}$ is a fixed square root of $v$. 

\begin{lemma}
    Suppose Theorem~\ref{thm:ybemc} is true for $g=g_0$. Then it is true for an
    arbitrary $g$ satisfying Assumption~\ref{assumptionga}.
\end{lemma}

\begin{proof}
   Write $g(a)=\gamma(a)g_0(a)$ where $\gamma(a)=1$ if $a\equiv 0$ mod $n$
   and $\gamma(a)\gamma(-a)=1$. Consider the scolored lines in both figures of
   Figure~\ref{fig:ybelr} that begin in one of the three vertices $\beta$,
   $\theta$ and $\rho$, and terminate in one of $\tau$, $\sigma$, $\alpha$. We
   ask what possible contributions of $\gamma(a)$ with $a\not\equiv 0$ mod $n$
   can occur. From the Boltzmann weights in
   Figures~\ref{fig:boltzmann_weights} and~\ref{fig:genrmatrix}
   we see that such contributions occur when two
   scolored lines with distinct scolors cross. For both systems, such crossings
   can be predicted just from the order of the scolors of the boundary spins
   $\beta$, $\theta$, $\rho$ and $\tau$, $\sigma$, $\alpha$, with the
   following caveat: we cannot determine from this information whether two
   scolored lines cross twice or not at all.  But if they cross twice, they
   produce factors $\gamma(a)$ and $\gamma(-a)$ for some $a$, and hence these
   factors cancel. Since modulo this caveat, the crossings of the
   scolored lines (with distinct scolors) is determined by the order
   of the boundary spins, the product of the factors $\gamma$
   that can occur are the same for the left and right systems in
   Figure~\ref{fig:ybelr}.
\end{proof}

\begin{remark}
We made the choice $g=g_0$ in equation~\eqref{eq:gquantumgroup} in order to simplify computations.  
Ultimately, we will specialize $g$ to certain quantities that have number theoretic meaning called Gauss sums. The R-matrix under this specialization (or any other that satisfies Assumption~\ref{assumptionga}) will correspond to a \textit{Drinfeld twist} of the $U_{q}\bigl(\widehat{\mathfrak{gl}}(r|n)\bigr)$ R-matrix (see Proposition~\ref{prop:Drinfeldtwisting}). Drinfeld twisting is a procedure that changes the comultiplication, antipode and universal R-matrix of a given quantum group, and is particularly useful in constructing new solutions of the Yang-Baxter equation. It was introduced by Drinfeld~\cite{Drinfeldtwist} and further studied by Reshetikhin~\cite{ReshetikhinDrinfeldtwist}.
\end{remark}

\begin{proof}[Proof of Theorem~\ref{thm:ybemc}]
  To make the auxiliary Yang-Baxter equation of Figure~\ref{fig:ybelr} amenable to proof by checking a finite number of equations independent of $n$ and $r$ we make a change of basis for the vector spaces of horizontal edges used to the describe the various weights for $T^{cw}_{z}$ and $R^{cw}_{z_i,z_j}$ in Figures~\ref{fig:boltzmann_weights} and~\ref{fig:genrmatrix}.
  
  Starting from a basis $\{ v_A \}$ where $A$ is an element of the horizontal spinset~\eqref{horizontalspinset} this change of basis will perform a rescaling of each basis element that will depend not only on the spin~$A$ but also on the data $z_i$, $c$ and $w$ of a vertex to the left of which the horizontal edge may be attached.
  By choosing this rescaling to be $z_i^{w-x+1-n}$ if $A=x$ is a scolor and $1$ otherwise we obtain the Boltzmann weights of Figure~\ref{fig:proof-weights} and the R-matrix remains unchanged except for the weights shown in Figure~\ref{fig:proof-R-matrix}.
  Here we have replaced $z_i^n$ with $\zeta_i$ and similarly for $\zeta_j$, and we will be able to prove the equivalent corresponding Yang-Baxter equation for these weights while keeping $\zeta_i$ and $\zeta_j$ as independent parameters.
  Note that all horizontal edges for the R-matrix $R^{cw}_{z_i,z_j}$ all attach to the left of a vertex $(c,w)$ with color $c$ and scolor $w$, while for $T^{cw}_{z_i}$ the left and right horizontal edges attach to the left of vertices $(c,w)$ and $(c',w')$ respectively where $(c',w')$ is the successor of $(c,w)$ according to~\eqref{ccsuccessor}.
  
\begin{figure}[ht]
{\scriptsize
\[
  \begin{array}{|c|c|c|}
\hline
\text{\normalsize $\texttt{a}_1$} & \text{\normalsize $\texttt{a}_2$} & \text{\normalsize $\texttt{b}_1$} \\
\hline
%a1
\begin{tikzpicture}[scaled, font=\normalsize]
  \draw[spin-path] (0,1) node {$+$} -- (0,-1) node {$+$};
  \draw[scolored-path] (-1,0) node {$y$} -- (1,0) node {$y$};
  \node[halo] at (0,0) {$T_z^{cw}$};
\end{tikzpicture}
&
%a2
\begin{tikzpicture}[scaled, font=\normalsize]
  \draw[both-path] (0,1) node {$\scriptstyle cw$} -- (0,-1) node {$\scriptstyle cw$};
  \draw[colored-path] (-1,0) node {$a$} -- (1,0) node {$a$};
  \node[halo] at (0,0) {$T_z^{cw}$};
\end{tikzpicture}
&
%b1
\begin{tikzpicture}[scaled, font=\normalsize]
  \draw[both-path] (0,1) node {$\scriptstyle cw$} -- (0,-1) node {$\scriptstyle cw$};
  \draw[scolored-path] (-1,0) node {$y$} -- (1,0) node {$y$};
  \node[halo] at (0,0) {$T_z^{cw}$};
\end{tikzpicture}
\\ \hline
%a1
\begin{array}{ll} 
  1 & \text{if } c < c_r \\
  z & \text{if } c = c_r, w > w_0 \\
  z/\zeta & \text{if } c = c_r, w = w_0
\end{array}
&
%a2
\begin{array}{ll} v & c>a\\ z & c=a\\ 1 & c<a\end{array} 
&
%b1
g(y-w) 
\times
\begin{cases}
  1 & c < c_r \\
  z & c = c_r, w > w_0 \\
  z/\zeta & c = c_r, w = w_0
\end{cases}
\\ \hline \hline
\text{\normalsize $\texttt{b}_2$} & \text{\normalsize $\texttt{c}_1$} & \text{\normalsize $\texttt{c}_2$} 
\\ \hline
%b2
\begin{tikzpicture}[scaled, font=\normalsize]
  \draw[spin-path] (0,1) node {$+$} -- (0,-1) node {$+$};
  \draw[colored-path] (-1,0) node {$a$} -- (1,0) node {$a$};
  \node[halo] at (0,0) {$T_z^{cw}$};
\end{tikzpicture}
&
%c1
\begin{tikzpicture}[scaled, font=\normalsize]
  \draw[spin-path] (0,1) node {$+$} -- (0,0);
  \draw[colored-path] (-1,0) node {$c$} -- (0,0);
  \draw[both-path] (0,0) -- (0,-1) node {$\scriptstyle cw$};
  \draw[scolored-path] (0,0) -- (1,0) node {$w$};
  \node[halo] at (0,0) {$T_z^{cw}$};
\end{tikzpicture}
&
%c2
\begin{tikzpicture}[scaled, font=\normalsize]
  \draw[spin-path] (0,-1) node {$+$} -- (0,0);
  \draw[colored-path] (1,0) node {$c$} -- (0,0);
  \draw[both-path] (0,0) -- (0,1) node {$\scriptstyle cw$};
  \draw[scolored-path] (0,0) -- (-1,0) node {$w$};
  \node[halo] at (0,0) {$T_z^{cw}$};
\end{tikzpicture}
\\ \hline
%b2
\begin{array}{ll} z&\text{if $a=c$,}\\ 1 &\text{otherwise.}\end{array} 
&
%c1
(1-v)\zeta \times
\begin{cases}
  1 & c < c_r \\
  z & c = c_r, w > w_0 \\
  z/\zeta & c = c_r, w = w_0
\end{cases} 
&
%c2
z/\zeta
\\ \hline
\end{array}\]}
\caption{The Boltzmann weights after a change of basis from Figure~\ref{fig:boltzmann_weights}.
Here we have replaced $z^n$ with $\zeta$.
In the Yang-Baxter equation $(z, \zeta)$ will be replaced by either $(z_i, \zeta_i)$ or $(z_j, \zeta_j)$ where $\zeta_i$ and $\zeta_j$ will be regarded as a independent parameters in the proof separate from $z_i^n$ and $z_j^n$.} 
\label{fig:proof-weights}
\end{figure}

\begin{figure}[ht]
{\scriptsize
\[\begin{array}{|c|c|c|c|}
\hline
\begin{tikzpicture}[Rscaled, font=\normalsize]
  \draw[scolored-path] (-1,-1) node {$x$} -- (0,0) -- (1,-1) node {$x$};
  \draw[scolored-path] (-1,+1) node {$y$} -- (0,0) -- (1,+1) node {$y$};
  \node[halo] at (0,0) {$R_{z_i,z_j}^{cw}$}; 
\end{tikzpicture}
&
\begin{tikzpicture}[Rscaled, font=\normalsize]
  \draw[scolored-path] (-1,+1) node {$x$} -- (0,0) -- (1,+1) node {$x$};
  \draw[colored-path] (-1,-1) node {$a$} -- (0,0) -- (1,-1) node {$a$};
  \node[halo] at (0,0) {$R_{z_i,z_j}^{cw}$}; 
\end{tikzpicture}
& 
\begin{tikzpicture}[Rscaled, font=\normalsize]
  \draw[scolored-path] (-1,-1) node {$x$} -- (0,0) -- (1,-1) node {$x$};
  \draw[colored-path] (-1,+1) node {$a$} -- (0,0) -- (1,+1) node {$a$};
  \node[halo] at (0,0) {$R_{z_i,z_j}^{cw}$}; 
\end{tikzpicture}
\\\hline
\begin{array}{ll}(1-v)\zeta_i  & y>x\\(1-v)\zeta_j  & y<x\end{array} 
&
\begin{array}{ll}
  (1-v)\zeta_i^2 \zeta_j^{-1} z_i^{-1} z_j & a<c, w\leqslant x\\
  (1-v)\zeta_i z_i^{-1} z_j & a<c, w>x\\
  (1-v)\zeta_i^2 \zeta_j^{-1} & a\geqslant c,w<x\\
    (1-v)\zeta_i & a\geqslant c,w\geqslant x\\
\end{array} &
\begin{array}{ll}
  (1-v)\zeta_j^2 \zeta_i^{-1} z_j^{-1} z_i & a<c, w\leqslant x\\
  (1-v)\zeta_j z_j^{-1} z_i & a<c, w>x\\
  (1-v)\zeta_j^2 \zeta_i^{-1} & a\geqslant c,w<x\\
(1-v)\zeta_j & a\geqslant c,w\geqslant x\\\end{array}
\\\hline
\end{array}
\]}
\caption{The modified weights for the R-matrix after a change of basis.
  Remaining weights are unchanged from Figure~\ref{fig:genrmatrix} with the exception of replacing each occurrence of $z_i^n$ with $\zeta_i$ and similarly for $\zeta_j$ (such that the weights are all polynomial in $\zeta_i^{\pm1}$ and $\zeta_j^{\pm1}$).
  The parameters $z_i$ and $\zeta_i$, as well as $z_j$ and $\zeta_j$, will be treated as independent parameters in the proof.
\label{fig:proof-R-matrix}}
\end{figure}

Now there are two key observations.
The first observation is that at most three colors and at most three scolors can
appear on the boundary and interior edges of the systems in Figure~\ref{fig:ybelr}. Indeed, based on the conservation of colors and scolors in $T$ and $R$ we get that the colors of $(\beta, \sigma, \tau)$ are the same as the colors of $(\alpha, \rho, \theta)$ in some permutation (also counting non-colored edges). Similarly, we have that the scolors of $(\beta, \rho, \theta)$ are the same as those of $(\alpha, \sigma, \tau)$ in some permutation. Besides the colors and scolors of the boundary edges in Figure~\ref{fig:ybelr} we also have the colors and scolors of the vertices which are $(c, w)$ and $(c',w')$. Thus, we can have at most five different colors and at most five different scolors appearing in the Yang-Baxter equation.

The second observation is that (with $g=g_0$)
the Boltzmann weights for $R_{z_i,z_j}^{cw}$ and
$T^{cw}_{z_i}$ in Figures~\ref{fig:proof-weights} and~\ref{fig:proof-R-matrix} depend only on the \textit{inequalities}
between the colors and scolors and not their actual values.
This is the reason for the particular change of basis we made.

Therefore if we check the auxiliary Yang-Baxter equation displayed in Figure~\ref{fig:ybelr} for these weights with $r=5$ and $n=5$,
every possible relative ordering of color and scolor has occurred in these calculations. Note that
for this argument, it is crucial that $\zeta_i$ and $\zeta_j$ are not
linked to $z_i$ and $z_j$ by the equation $\zeta_i=z_i^n$, even
though this is what we later use to relate these weights to those of Figures~\ref{fig:boltzmann_weights} and~\ref{fig:genrmatrix}.

Thus if the auxiliary Yang-Baxter equation with the change of basis is valid for
$n=r=5$, then it is true for all $n$ and $r$, and so is the auxiliary Yang-Baxter equation in the statement of the theorem. The case $n = r = 5$ with the weights in Figures~\ref{fig:proof-weights} and~\ref{fig:proof-R-matrix} was
checked using a computer program (in Sage using symbolic manipulations), and so the theorem
is proved.
\end{proof}

After repeated use of these auxiliary Yang-Baxter equations, passing through a complete cycle of pairs of colors and scolors $(c,w)$,
we immediately arrive at the following:

\begin{corollary}
\label{cor:YBE}
The fully-fused lattice model satisfies the Yang-Baxter equation, written as an identity of endomorphisms on $V \otimes V \otimes W$:
\[(R_{z_i, z_j})_{12} (T_{z_i})_{13} (T_{z_j})_{23} = (T_{z_j})_{23} (T_{z_i})_{13} (R_{z_i, z_j})_{12}\]
where $V$ is a vector space with basis indexed by the set of colors and scolors and $W$ is a vector space with basis indexed by the power set of the set of pairs $(c,w)$ of colors and scolors. 
We use the R-matrix weights in Figure~\ref{fig:rmatrix} and the notation $(A)_{\ell k}$ means applying the endomorphism $A$ to the $\ell$- and $k$-th
factors in the tensor product.
\end{corollary}

\begin{remark}
The set of color-scolor pairs $(c,w)$ is in bijection with the set of odd negative roots of $U_{1/\sqrt{v}}(\mathfrak{gl}(r|n))$. Therefore the power set that is the basis of $W$ in Corollary~\ref{cor:YBE} is a basis of the exterior algebra on the space spanned by the odd negative roots. This observation may be the key to a quantum group interpretation of~$W$.
\end{remark}

\begin{figure}[ht]
{\scriptsize
\[\begin{array}{|c|c|c|c|}
\hline
\begin{tikzpicture}[Rscaled, font=\normalsize]
  \draw[colored-path] (-1,-1) node {$a$} -- (0,0) -- (1,-1) node {$a$};
  \draw[colored-path] (-1,+1) node {$a$} -- (0,0) -- (1,+1) node {$a$};
  \node[halo] at (0,0) {$R_{z_i,z_j}$}; 
\end{tikzpicture}
&
\begin{tikzpicture}[Rscaled, font=\normalsize]
  \draw[colored-path] (-1,-1) node {$a$} -- (0,0) -- (1,+1) node {$a$};
  \draw[colored-path] (-1,+1) node {$b$} -- (0,0) -- (1,-1) node {$b$};
  \node[halo] at (0,0) {$R_{z_i,z_j}$}; 
\end{tikzpicture}
&
\begin{tikzpicture}[Rscaled, font=\normalsize]
  \draw[colored-path] (-1,-1) node {$a$} -- (0,0) -- (1,-1) node {$a$};
  \draw[colored-path] (-1,+1) node {$b$} -- (0,0) -- (1,+1) node {$b$};
  \node[halo] at (0,0) {$R_{z_i,z_j}$}; 
\end{tikzpicture}
&
\begin{tikzpicture}[Rscaled, font=\normalsize]
  \draw[scolored-path] (-1,-1) node {$x$} -- (0,0) -- (1,+1) node {$x$};
  \draw[scolored-path] (-1,+1) node {$x$} -- (0,0) -- (1,-1) node {$x$};
  \node[halo] at (0,0) {$R_{z_i,z_j}$}; 
\end{tikzpicture}
\\\hline
z^n_i-vz^n_j &
\begin{array}{ll} v(z^n_i-z^n_j) & a<b\\z^n_i-z^n_j & a>b\end{array} &
\begin{array}{ll}
  (1 - v) z^n_i & a < b\\
  (1 - v) z^n_j & a > b\\
\end{array}
&
z^n_j-v z^n_i 
\\\hline\hline
\begin{tikzpicture}[Rscaled, font=\normalsize]
  \draw[scolored-path] (-1,-1) node {$x$} -- (0,0) -- (1,+1) node {$x$};
  \draw[scolored-path] (-1,+1) node {$y$} -- (0,0) -- (1,-1) node {$y$};
  \node[halo] at (0,0) {$R_{z_i,z_j}$}; 
\end{tikzpicture}
&
\begin{tikzpicture}[Rscaled, font=\normalsize]
  \draw[scolored-path] (-1,-1) node {$x$} -- (0,0) -- (1,-1) node {$x$};
  \draw[scolored-path] (-1,+1) node {$y$} -- (0,0) -- (1,+1) node {$y$};
  \node[halo] at (0,0) {$R_{z_i,z_j}$}; 
\end{tikzpicture}
&
\begin{tikzpicture}[Rscaled, font=\normalsize]
  \draw[scolored-path] (-1,-1) node {$x$} -- (0,0) -- (1,+1) node {$x$};
  \draw[colored-path] (-1,+1) node {$a$} -- (0,0) -- (1,-1) node {$a$};
  \node[halo] at (0,0) {$R_{z_i,z_j}$}; 
\end{tikzpicture}
&
\begin{tikzpicture}[Rscaled, font=\normalsize]
  \draw[scolored-path] (-1,+1) node {$x$} -- (0,0) -- (1,-1) node {$x$};
  \draw[colored-path] (-1,-1) node {$a$} -- (0,0) -- (1,+1) node {$a$};
  \node[halo] at (0,0) {$R_{z_i,z_j}$}; 
\end{tikzpicture}
\\ \hline
g(y-x)(z^n_i-z^n_j) &
\begin{array}{ll}(1-v)z^n_i {(\tfrac{z_j}{z_i})^{y-x}} & y>x\\(1-v)z^n_j {(\tfrac{z_j}{z_i})^{y-x}} & y<x\end{array} &
v(z^n_i-z^n_j) & 
z^n_i-z^n_j
\\\hline\hline
\multicolumn{2}{|c|}{
\begin{tikzpicture}[Rscaled, font=\normalsize]
  \draw[scolored-path] (-1,+1) node {$x$} -- (0,0) -- (1,+1) node {$x$};
  \draw[colored-path] (-1,-1) node {$a$} -- (0,0) -- (1,-1) node {$a$};
  \node[halo] at (0,0) {$R_{z_i,z_j}$}; 
\end{tikzpicture}
}
& 
\multicolumn{2}{c|}{
\begin{tikzpicture}[Rscaled, font=\normalsize]
  \draw[scolored-path] (-1,-1) node {$x$} -- (0,0) -- (1,-1) node {$x$};
  \draw[colored-path] (-1,+1) node {$a$} -- (0,0) -- (1,+1) node {$a$};
  \node[halo] at (0,0) {$R_{z_i,z_j}$}; 
\end{tikzpicture}
}
\\\hline
\multicolumn{2}{|c|}{
  (1-v)z^n_i {(\tfrac{z_j}{z_i})^{x}}
} &
\multicolumn{2}{|c|}{
(1-v)z^n_j {(\tfrac{z_i}{z_j})^{x}}
}
\\\hline
\end{array}\]}
\caption{The R-matrix for the fully-fused model. Here $a$ and $b$ are
\textit{distinct} colors and $x$ and $y$ are \textit{distinct}
scolors. The weights are obtained from the special case of Figure~\ref{fig:genrmatrix}
where $c=c_1$ and $w=w_{n-1}$.}
\label{fig:rmatrix}
\end{figure}

The fully-fused R-matrix in Figure~\ref{fig:rmatrix} is obtained from the unfused R-matrix in Figure~\ref{fig:genrmatrix} by taking the vertex color and scolor to be $c = c_1$ and $w = w_{n-1}$.
We may similarly obtain a fully-fused R-matrix from the one in Figure~\ref{fig:proof-R-matrix}, which is related to the one in Figure~\ref{fig:rmatrix} by the change of basis introduced in the proof of Theorem~\ref{thm:ybemc}.
The extra factor of $z_i^{1-n}$ in this change of basis was made to facilitate the comparison with known R-matrices from quantum groups resulting in the following proposition.

\begin{proposition}\label{prop:Drinfeldtwisting}
The fully-fused R-matrix in Figure~\ref{fig:rmatrix} for the metaplectic Iwahori ice model is a Drinfeld twist of the $U_{1/\sqrt{v}}\bigl(\widehat{\mathfrak{gl}}(r|n)\bigr)$ R-matrix.
\end{proposition}

\begin{proof}
  After the change of basis made in the proof of Theorem~\ref{thm:ybemc} we consider the fully-fused version of the R-matrix in Figure~\ref{fig:proof-R-matrix} and compare it to Kojima's $U_{1/\sqrt{v}}\bigl(\widehat{\mathfrak{gl}}(r|n)\bigr)$ R-matrix in~\cite{Kojima} acting on a tensor product of evaluation representations. In fact, Kojima considers the $U_{1/\sqrt{v}}\bigl(\widehat{\mathfrak{sl}}(r|n)\bigr)$ R-matrix, but the two quantum groups have the same R-matrix up to some normalization factor.  
  We start from the R-matrix in \cite{Kojima} parametrized by $q$ and $z$ which we equate to $1/\sqrt{v}$ and $\zeta_i/\zeta_j$ respectively, and then rescale each weight by the factor $\zeta_i - v \zeta_j$.
  Note that Kojima considers a graded Yang-Baxter equation, while we are not, which means that weights where all edge configurations are colors are multiplied by $-1$.
  The resulting R-matrix then differs from the one in Figure~\ref{fig:proof-R-matrix} by a Drinfeld twist parametrized by a function $\phi(a,b)$ where $a$ and $b$ are horizontal edge states rescaling weights of the form 
  
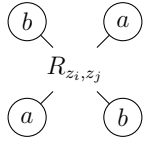
\[
    \begin{tikzpicture}[Rscaled]
      \draw[spin-path] (-1,-1) node {$a$} -- (1,1) node {$a$};
      \draw[spin-path] (-1,1) node {$b$} -- (1,-1) node {$b$};
      \node[halo] at (0,0) {$R_{z_i,z_j}$}; 
    \end{tikzpicture}
  \]
  The procedure of using such a function $\phi$ to perform a Drinfeld twist is
  explained in \cite[Section 4]{BBB} and \cite[Section~4]{BBBF}.
  The function satisfies $\phi(a,b) \phi(b,a) = 1$ and is here given by
  \begin{equation}
    \scriptsize
    \text{\normalsize $\phi(a,b) =$}
    \begin{cases}
      -1/\sqrt{v} & \text{if $a$ and $b$ are colors and $a > b$}, \\
      -\sqrt{v} & \text{if $a$ and $b$ are colors and $a < b$}, \\
      g(b-a)/\sqrt{v} & \text{if $a$ and $b$ are scolors and $a \neq b$}, \\
      \sqrt{v} & \text{if $a$ is a scolor and $b$ is a color}, \\
       1/\sqrt{v} & \text{if $a$ is a color and $b$ is a scolor}, \\
      1 & \text{otherwise.}
    \end{cases}
    \normalsize
  \end{equation}
\end{proof}

\subsection{Recurrence relations for the partition functions}
\label{sec:lattice-recursion}

The results of the previous section allow us to demonstrate that the partition functions satisfy the following recurrence.

For $x \in \mathbb{Z}/n\mathbb{Z}$, let $\floorn{x}$ denote the least nonnegative residue of $x \bmod{n}$, that is, the integer such that $0 \leqslant \floorn{x} < n$ and $x \equiv \floorn{x} \pmod{n}$.
Similarly, let $\ceiln{x}$ denote the integer such that $0 < \ceiln{x} \leqslant n$ and $x \equiv \ceiln{x} \pmod{n}$.
Then for $x, y \in \{0, 1, \ldots, n-1\}$, 
\begin{equation}
  \label{eq:ceiln}
  \ceiln{y-x} =
  \begin{cases*}
    y - x & if $y > x$ \\
    y - x + n & if $y \leqslant x$.
  \end{cases*}
\end{equation}

\begin{proposition} 
  \label{prop:lattice-recursion}
  Let $\mu\in\mathbb{Z}^r_{\geqslant0}$, $\theta \in (\mathbb{Z}/n\mathbb{Z})^r$, and $w \in S_r$ defining a system $\mathfrak{S}_{\mu, \theta, w}$.
  Then, for the simple reflection $s_i$, the following equality of partition functions hold 
  \begin{multline}
    \label{eq:coloredrecursion}
    (1-v) \mathbf{z}^{-\ceiln{\theta_i-\theta_{i+1}}\alpha_i} Z(\mathfrak{S}_{\mu, \theta, w})(\mathbf{z}) + g(\theta_i-\theta_{i+1}) (1 - \mathbf{z}^{-n\alpha_i}) Z(\mathfrak{S}_{\mu, s_i\theta, w})(\mathbf{z})  
    = \\
    = \begin{cases*}
      (1-v) Z(\mathfrak{S}_{\mu, \theta, w})(s_i\mathbf{z}) + (1 - \mathbf{z}^{-n\alpha_i}) Z(\mathfrak{S}_{\mu, \theta, s_iw})(s_i\mathbf{z}) & if $\ell(s_iw) > \ell(w)$ \\
      (1-v)\mathbf{z}^{-n\alpha_i} Z(\mathfrak{S}_{\mu, \theta, w})(s_i\mathbf{z}) + v (1 - \mathbf{z}^{-n\alpha_i}) Z(\mathfrak{S}_{\mu, \theta, s_i w})(s_i\mathbf{z}) & if $\ell(s_iw) < \ell(w)$ \\
    \end{cases*}
  \end{multline}
\end{proposition}

\begin{proof} 
  Let $y = \floorn{\theta_i}$ and $x = \floorn{\theta_{i+1}}$ be the left boundary scolors of $\mathfrak{S}_{\mu, \theta, w}$ on rows $i$ and $i+1$ and similarly let $d = (wP)_i$ and $c = (wP)_{i+1}$ be the right boundary colors.
  Applying the Yang-Baxter equation in the fully-fused model in rows $i$ and $i+1$, we have the following identity (as pictured in Figure~\ref{fig:modpf})
\begin{multline} 
\operatorname{wt}\left( 
\begin{tikzpicture}[baseline={(0,-0.1)}, Rscaled]
  \draw[scolored-path] (-1,-1) node {$x$} -- (0,0) -- (1,+1) node {$x$};
  \draw[scolored-path] (-1,+1) node {$y$} -- (0,0) -- (1,-1) node {$y$};
  \node[halo, inner sep=-3pt] at (0,0) {$R_{z_i,z_{i+1}}$}; 
\end{tikzpicture}  
\right) Z(\mathfrak{S}_{\mu, s_i\theta, w})(\mathbf{z}) + 
\operatorname{wt}
\left( 
\begin{tikzpicture}[baseline={(0,-0.1)}, Rscaled]
  \draw[scolored-path] (-1,-1) node {$x$} -- (0,0) -- (1,-1) node {$x$};
  \draw[scolored-path] (-1,+1) node {$y$} -- (0,0) -- (1,+1) node {$y$};
  \node[halo, inner sep=-3pt] at (0,0) {$R_{z_i,z_{i+1}}$}; 
\end{tikzpicture}
\right) Z(\mathfrak{S}_{\mu, \theta, w})(\mathbf{z}) = \\
\operatorname{wt} \left( 
\begin{tikzpicture}[baseline={(0,-0.1)}, Rscaled]
  \draw[colored-path] (-1,-1) node {$c$} -- (0,0) -- (1,+1) node {$d$};
  \draw[colored-path] (-1,+1) node {$d$} -- (0,0) -- (1,-1) node {$c$};
  \node[halo, inner sep=-3pt] at (0,0) {$R_{z_i,z_{i+1}}$}; 
\end{tikzpicture}
\right) Z(\mathfrak{S}_{\mu, \theta, w})(s_i\mathbf{z})
+
\operatorname{wt} \left( 
\begin{tikzpicture}[baseline={(0,-0.1)}, Rscaled]
  \draw[colored-path] (-1,-1) node {$d$} -- (0,0) -- (1,+1) node {$d$};
  \draw[colored-path] (-1,+1) node {$c$} -- (0,0) -- (1,-1) node {$c$};
  \node[halo, inner sep=-3pt] at (0,0) {$R_{z_i,z_{i+1}}$}; 
\end{tikzpicture}
\right) Z(\mathfrak{S}_{\mu, \theta, s_iw})(s_i\mathbf{z}).
\end{multline}
  Here $\operatorname{wt}(R)$ denotes the Boltzmann weight of the vertex $R$, and we assume $x \neq y$ and $c \neq d$.
  If $x = y$ the two terms on the left-hand side agree, but by the Yang-Baxter equation, only one of them should be included. 
  Note that the colors in our model are distinct, meaning that $c \neq d$. 
  
  Inserting the Boltzmann weights from Figure~\ref{fig:rmatrix}, being careful to distinguish between the cases when $x \ne y$ versus $x=y$, and using \eqref{eq:ceiln} proves the statement.
\end{proof}

\begin{figure}[htb]
\begin{tikzpicture}[scaled, every path/.append style={thick}]
\begin{scope}
  \draw (0,1) to [out = 0, in = 180] (2,3) to (4,3);
  \draw (0,3) to [out = 0, in = 180] (2,1) to (4,1);
  \draw (3,0.25) to (3,3.75);
  \draw (7,0.25) to (7,3.75);
  \draw (6,1) to (8,1);
  \draw (6,3) to (8,3);
  \node[scolored,spin] at (0,1) {$x$};
  \node[scolored,spin] at (0,3) {$y$};
  \node at (5,3) {$\cdots$};
  \node at (5,1) {$\cdots$};
  \draw[densely dashed] (3,3.75) to (3,4.25);
  \draw[densely dashed] (3,0.25) to (3,-0.25);
  \draw[densely dashed] (7,3.75) to (7,4.25);
  \draw[densely dashed] (7,0.25) to (7,-0.25);
  \node[colored,spin] at (8,1) {$c$};
  \node[colored,spin] at (8,3) {$d$};
  \node[halo] at (3,3) {$T_{z_{i}}$};
  \node[halo] at (3,1) {$T_{z_{i+1}}$};
  \node[halo] at (7,3) {$T_{z_{i}}$};
  \node[halo] at (7,1) {$T_{z_{i+1}}$};
  \node[halo] at (1,2) {$R_{z_{i},z_{i+1}}$};
\end{scope}
\begin{scope}[shift={(11,0)}]
  \node[scale=1/0.8] at (-1.5,2) {$=$};
  
  \draw (4,1) to (6,1) to [out = 0, in = 180] (8,3);
  \draw (4,3) to (6,3) to [out = 0, in = 180] (8,1);
  \draw (0,1) to (2,1);
  \draw (0,3) to (2,3);
  \draw (5,0.25) to (5,3.75);
  \draw (1,0.25) to (1,3.75);

  \node at (3,1) {$\cdots$};
  \node at (3,3) {$\cdots$};

  \draw[densely dashed] (1,3.75) to (1,4.25);
  \draw[densely dashed] (1,0.25) to (1,-0.25);
  \draw[densely dashed] (5,3.75) to (5,4.25);
  \draw[densely dashed] (5,0.25) to (5,-0.25);
  \node[halo] at (1,3) {$T_{z_{i+1}}$};
  \node[halo] at (1,1) {$T_{z_{i}}$};
  \node[halo] (a) at (5,3) {$T_{z_{i+1}}$};
  \node[halo] at (5,1) {$T_{z_{i}}$};
  \node[halo] at (7,2) {$R_{z_{i},z_{i+1}}$};
  \node[colored, spin] at (8,1) {$c$};
  \node[colored, spin] at (8,3) {$d$};
  \node[scolored, spin] at (0,1) {$x$};
  \node[scolored, spin] at (0,3) {$y$};
\end{scope}
\end{tikzpicture}
\caption{Left: the system $\mathfrak{S}_{\mu,\theta,w}$ with
  the R-matrix attached. Right: after using the Yang-Baxter equation.}
  \label{fig:modpf}
\end{figure}

Note that from the recursion relations from Proposition~\ref{prop:lattice-recursion} and the base case from Lemma~\ref{lem:lattice-ground-state} we can compute the partition function for all systems $\mathfrak{S}_{\mu, \theta, w}$.

\section{Metaplectic Iwahori Whittaker functions}
\label{sec:whittaker}

\subsection{Metaplectic preliminaries}
\label{sec:metaplectic-prel}
We begin with a brief review of the construction of metaplectic groups and
their unramified principal series representations. The discussion
follows~\cite{BBBF} whose treatment of metaplectic groups is based on the work
of Matsumoto~\cite{metaplecticMatsumoto}, Kazhdan and
Patterson~\cite{KazhdanPatterson}, Brylinski and
Deligne~\cite{BrylinksiDeligne}, and
McNamara~\cite{McNamaraEdinburgh,McNamaraCS}.  Because they may be stated
uniformly, the constructions in this section will be described for covers
associated to any split, reductive algebraic group $\mathbf{G}$. In Section~\ref{sec:conclusions}, 
where we perform explicit calculations for comparison with the lattice models of the prior section, 
we will restrict to a particular cover of the general linear group.

Let $F$ be a non-archimedean local field and let $\mathfrak{o}$ be its ring of integers with uniformizer $\varpi$ and maximal prime ideal $\mathfrak{p} = \varpi\mathfrak{o}$.
Denote the cardinality of the residue field $\mathfrak{o}/\mathfrak{p}$ by $v^{-1}$.
Fix a positive integer $n$. We shall work under the assumption that $v^{-1} \equiv 1 \bmod 2n$ which implies that $F$ contains the $2n$-th roots of unity. Let $\mu_n$ denote the group of $n$-th roots of unity in $F$ and fix an embedding $\iota : \mu_n \to \C^\times$. 

Let $\mathbf{G}$ be a connected reductive group defined and split
over $F$. Let $\Delta^\vee$ be the root system of $\mathbf{G}$, which we
assume to be irreducible.

\begin{remark}
The root system $\Delta^\vee$ is dual to the root system $\Delta$ of the
Langlands dual group~$\widehat{\mathbf{G}}$. The reason for this notational
choice is that in this paper it is the representation theory of the dual group
that will be more prevalent.
\end{remark}

Let $\mathbf{T}$ be a fixed split maximal torus of $\mathbf{G}$, and let $\widehat{\mathbf{T}}$
be the corresponding maximal torus of $\widehat{\mathbf{G}}$. Then
$\Delta^\vee$ is a subset of the character group $X^\ast(\mathbf{T})$.
The dual root system $\Delta$ is a subset of the cocharacter group
$X_\ast(\mathbf{T})$, which is identified with $X^\ast(\widehat{\mathbf{T}})$;
we will denote $\Lambda=X_\ast(\mathbf{T})=X^\ast(\widehat{\mathbf{T}})$.
If $\alpha^\vee\in\Delta^\vee$, the corresponding element of $\Delta$
is denoted $\alpha$.

We fix a pinning. This means that for every root
$\alpha^\vee\in\Delta^\vee$ there is a morphism
${i_{\alpha^\vee} : \mathbf{SL}_2 \rightarrow \mathbf{G}}$, whose image
we will denote $\mathbf{G_{\alpha^\vee}}$. Let
$x_{\alpha^\vee}(u)=i_{\alpha^\vee}\bigl(\begin{smallmatrix}1&u\\&1\end{smallmatrix}\bigr)$.
Then the image of~$x_{\alpha^\vee}$ is
the subgroup generated by one-parameter unipotent subgroups $\mathbf{N}_{\alpha^\vee}$ on which
$\mathbf{T}$ acts (through conjugation) by the character $\alpha^\vee\in X^\ast(\mathbf{T})$, and
\[\mathbf{G}_{\alpha^\vee}=\langle \mathbf{N}_{\alpha^\vee}, \mathbf{N}_{-{\alpha^\vee}} \rangle.\]
We require the maps $i_{\alpha^\vee}$ to satisfy commutation relations as in
Chevalley~\cite[Section III]{Chevalley}. 

The root system $\Delta^\vee$ may be partitioned into positive and negative
roots, $\Delta^\vee_+$ and $\Delta^\vee_-$. These choices correspond to positive
and negative Borel subgroups of $\mathbf{G}$, denoted $\mathbf{B}$ and $\mathbf{B}^-$ respectively. Then
$\mathbf{N}(F) \mathbf{T}(F)= \mathbf{B}(F)$ and
$\mathbf{N}^-(F) \mathbf{T}(F) = \mathbf{B}^-(F)$ where $\mathbf{N}
:= \prod_{{\alpha^\vee} \in \Delta^\vee_+} \mathbf{N}_{\alpha^\vee}$ and $\mathbf{N}^-
:= \prod_{{\alpha^\vee} \in \Delta^\vee_{-}} \mathbf{N}_{{\alpha^\vee}}$ 
is its opposite.

Let $K:=\mathbf{G}(\mathfrak{o})$ be the group generated by the
$x_{\alpha^\vee}(\mathfrak{o})$, which is a maximal compact subgroup of
$\mathbf{G}(F)$. Let
$W=N_{\mathbf{G}}(\mathbf{T})/\mathbf{T}$ denote the Weyl group of $\mathbf{G}$ with
simple reflections $s_i$ corresponding to each simple root $\alpha^\vee_i$. We may
choose Weyl group representatives in~$K$.

A \textit{metaplectic $n$-fold cover} of $\mathbf{G}(F)$ is a central extension of $\mathbf{G}(F)$ by $\mu_n$:
\begin{equation}\label{eq:metaplecticgroup} 
1 \longrightarrow \mu_n \longrightarrow \Gn \xlongrightarrow[]{p} \mathbf{G}(F) \longrightarrow 1.
\end{equation}
As a set $\Gn \equiv \mathbf{G}(F) \times \mu_n$, with the group
multiplication given by a cocycle $\sigma \in H^2(\mathbf{G}(F),\mu_n)$.
The group $K$ splits in $\Gn$, so we fix a splitting and
identify $K$ with its image in $\widetilde{G}$.
Let $\widetilde{B}=p^{-1}\bigl(\mathbf{B}(F)\bigr)$ and
$\widetilde{T}=p^{-1}\bigl(\mathbf{T}(F)\bigr)$.

As in \cite{BrylinksiDeligne,McNamaraEdinburgh,McNamaraCS}, the
cover \eqref{eq:metaplecticgroup} is associated with a $W$-invariant symmetric bilinear
form $B: \Lambda \times \Lambda \to \mathbb{Z}$ whose restriction to the
root lattice is required to be even. Thus
the quadratic form $Q(\mu):= B(\mu,\mu)/2$ is integer-valued. If
$x\in F^\times$ and $\lambda\in\Lambda\cong X_\ast(\mathbf{T})$
let $x^\lambda$ be a representative of $x$ under the
cocharacter of $\mathbf{T}$ corresponding to $\lambda$.
By abuse of notation, we will denote a representative
in $\widetilde{T}$ by the same notation $x^\lambda$.
The bilinear form $B$ is characterized by
\begin{equation}\label{bcommrel}
  [x^\lambda,y^\mu]=(x,y)_n^{B(\lambda,\mu)}, \qquad x,y \in F^\times, \quad \lambda,\mu \in \Lambda,
\end{equation}
where $[x^\lambda,y^\mu]$ is the group commutator and $(x,y)_n \in \mu_n$ is the $n$-th power Hilbert symbol. This
identity holds regardless of the representatives $x^\lambda$
and $y^\mu$. 
Since we are assuming that $F$ contains the $2n$-th roots of unity,
$(\varpi,\varpi)_n=1$. Thus, by \eqref{bcommrel}, $\varpi^\lambda$ and $\varpi^\mu$
commute and so we may choose the representatives
$\varpi^\lambda$ such that $\varpi^\lambda\varpi^\mu=\varpi^{\lambda+\mu}$.
See Section 5 in~\cite{BBBF} for a discussion on the bilinear form in relation to the metaplectic group.

Let $T(\mathfrak{o})=\Tn\cap K$.  Since the cover splits over $K$, the group
$T(\mathfrak{o})$ may (like $K$) be regarded as a subgroup of either
$\mathbf{G}(F)$ or~$\Gn$.  Denote the centralizer of $T(\mathfrak{o})$ in $\tilde T$ by $H:=
C_{\Tn}\big(T(\mathfrak{o})\big)$. It is a maximal abelian subgroup of
$\widetilde{T}$ by Lemma~1 of~\cite{McNamaraEdinburgh}.

The map $\lambda\mapsto\varpi^\lambda$ induces an isomorphism
$\Lambda \xlongrightarrow{\sim} \Tn / \mu_n T(\mathfrak{o})$. Let
\begin{equation}\label{eq:Lambdan} 
  \Lambda^{(n)} := \{ \mu \in \Lambda \mid B(y,\mu) \equiv 0 \bmod{n} \text{ for all } y\in \Lambda \}. 
\end{equation}
It follows from Lemma 1 in~\cite{McNamaraEdinburgh} (and its proof) that
the subgroup $H$ corresponds to $\Lambda^{(n)}$ under this map. In other words
the isomorphism $\Lambda \xlongrightarrow{\sim} \Tn / \mu_n T(\mathfrak{o})$
induces an isomorphism $\Lambda^{(n)} \iso H/ \mu_n
T(\mathfrak{o})$. Therefore
\begin{equation}\label{eq:LambdamodLambdan} 
\Lambda / \Lambda^{(n)} \iso \Tn / H.
\end{equation}
We will denote by $\Gamma$ a fixed choice of representatives of $\Lambda
/ \Lambda^{(n)}$.

We next turn to the definition of the \textit{metaplectic L-group}
given in McNamara~\cite{McNamaraEdinburgh}. In that paper, like
this one, it was assumed that $\mathbf{G}$ is split and
that the ground field $F$ contained the group $\mu_{2n}$ of $2n$-th roots of
unity. These assumptions simplify many technical issues.
See Weissman~\cite{Weissman} for a theory of the metaplectic L-group
without these limitations.

Identifying $\Lambda$ with $X^\ast(\widehat{\mathbf{T}})$, if $\z\in\widehat{\mathbf{T}}(\mathbb{C})$ and
$\lambda\in\Lambda$, denote by $\z^\lambda\in\mathbb{C}$ the image
of $\z$ under the rational character $\lambda$.
Let us denote by $\mathbb{C}[\Lambda]$ the Laurent polynomial
ring spanned over $\mathbb{C}$ by the polynomial functions $\z\mapsto\z^\lambda$ ($\lambda\in\Lambda$),
or by abuse of notation, by $\z^\lambda$. We will
denote by $\mathbb{C}(\Lambda)$ its field of fractions.
The rings $\mathbb{C}[\Lambda]$ and $\mathbb{C}(\Lambda)$
may be identified with the ring of regular and
rational functions, respectively, on $\widehat{\mathbf{T}}(\mathbb{C})$.

Similarly we define $\mathbb{C}[\Lambda^{(n)}]$ and
$\mathbb{C}(\Lambda^{(n)})$. Let 
$\widehat{\mathbf{T}}'$ be an algebraic torus
that is the quotient of $\widehat{\mathbf{T}}$ by
the finite subgroup of $\z$ such that $\z^\lambda=1$
for $\lambda\in\Lambda^{(n)}$. Thus 
$\mathbb{C}[\Lambda^{(n)}]$ and
$\mathbb{C}(\Lambda^{(n)})$ may be identified with
the rings of regular and rational functions on
$\widehat{\mathbf{T}}'(\mathbb{C})$. We will denote
by $\hat p:\widehat{\mathbf{T}}'(\mathbb{C})\to \widehat{\mathbf{T}}(\mathbb{C})$
the canonical surjection.

If $\alpha\in\Delta$, let 
\begin{equation}
  \label{eq:nidef}
  n_\alpha := \frac{n}{\operatorname{gcd}\bigl(n,Q(\alpha)\bigr)} \qquad n_i := n_{\alpha_i}
\end{equation}
as in \cite{McNamaraCS,McNamaraEdinburgh}. Then let
\begin{equation}
  \label{eq:Dn}
  \Delta_n=\{n_\alpha\alpha \mid \alpha\in\Delta\}.
\end{equation}
McNamara~\cite[Theorem~12]{McNamaraEdinburgh} proves
that $\Delta_n$ is a subset of $\Lambda^{(n)}=X^\ast(\widehat{\mathbf{T}}')$,
and that it is a root system.
Moreover let $\alpha^\vee\in
X_\ast(\widehat{\mathbf{T}})\iso\Hom(\Lambda,\mathbb{Z})$
be the coroot corresponding to $\alpha$. Then McNamara proves that the
image of $\alpha^\vee$ in $X_\ast(\widehat{\mathbf{T}}')\iso\Hom(\Lambda^{(n)},\mathbb{Z})$
is a multiple of $n_\alpha$. Let
\[\Delta_n^\vee=\{n_\alpha^{-1}\alpha^\vee|\alpha\in\Delta\}\subset X_\ast(\widehat{\mathbf{T}}').\]
Then $(X^\ast(\widehat{\mathbf{T}}'),\Delta_n,X_\ast(\widehat{\mathbf{T}}'),\Delta_n^\vee)$
form a root datum (Springer~\cite{SpringerReductive}) and by
\cite[Theorem~2.9]{SpringerReductive} they are associated with a reductive
group $\widehat{\mathbf{G}}'$. This is the \textit{metaplectic L-group}.

\begin{remark}
\label{rem:dichotomy}
If $\Delta$ is simply-laced, or if it is doubly-laced and $n$ is odd, or if
$\mathbf{G}=G_2$ and $n$ is not a multiple of $3$, then the Cartan types of
$\widehat{\mathbf{G}}$ and $\widehat{\mathbf{G}}'$ are the same. Otherwise, they are dual, since long
roots of $\Delta$ correspond to short roots of $\Delta_n$. This dichotomy
is reflected in the local Shimura correspondence~\cite{SavinHecke}.
\end{remark}

Let us define the \textit{unramified principal series} of $\Gn$.
A function $f$ on $H$, $\widetilde{T}$ or $\widetilde{G}$ is called \textit{genuine} if
$f(\varepsilon g)=\iota(\varepsilon)f(g)$ for $\varepsilon\in\mu_n$.
A quasicharacter of $H$ is called \textit{unramified} if it is trivial on
$T(\mathfrak{o})$. Unramified genuine quasicharacters of $H$ may be
parametrized by elements of $\widehat{\mathbf{T}}'(\C)$. If
$\widetilde{\z}\in\widehat{\mathbf{T}}'(\C)$, let us pick
$\z\in\widehat{\mathbf{T}}(\C)$ such that $\hat p(\z)=\widetilde{\z}$ under the
surjection
$\hat p:\widehat{\mathbf{T}}(\C)\longrightarrow\widehat{\mathbf{T}}'(\C)$.
There is a unique unramified
character $\chi=\chi_\z$ such that $\chi(\varpi^\lambda)=\z^\lambda$.
(The character $\chi_\z$ actually only depends on $\widetilde{\z}$.) Given an unramified genuine quasicharacter $\chi$ on $H$ let
\[i(\chi) = \operatorname{Ind}_H^{\Tn} \chi = \{ f: \Tn \to \C, f (h  t) = \chi(h)f(t) \text{ for all } t \in \Tn, h \in H \}. \] 
The group $\Tn$ acts on $i(\chi)$ by right translation. It
is a finite-dimensional $\Tn$-module that may be shown to be irreducible.

We inflate $i(\chi)$ from $\Tn$ to $\Bn$ and induce to $\Gn$ to obtain
$I(\chi) := \operatorname{Ind}_{\Bn}^{\Gn} i(\chi)$. Explicitly, we
have
\[ I(\chi) = \{ \text{smooth $f: \Gn \to i(\chi)$} \mid f(bg) = \bigl(\delta^{1/2} \chi\bigr)(b)f(g) \text{ for all } b \in \Bn, g \in \Gn \}, \]
where $\delta$ is the modular
quasicharacter.
The group $\Gn$ acts by right translation on $I(\chi)$ denoted by $\pi$.
For $\chi = \chi_\mathbf{z}$ defined above we denote $I(\chi_\mathbf{z})$ by
$I(\mathbf{z})$. Although it only depends on
$\widetilde{\z}\in\widetilde{\mathbf{T}}'$ we will not
put the tilde into the notation.

\begin{remark}
\label{rem:zassumptions}
To avoid complications we shall impose two conditions on the character
$\chi_\z$. We assume that $\z^{n_\alpha\alpha} \neq v^{\pm1}$ and
$\z^{n_\alpha\alpha}\neq 1$ for $\alpha\in\Delta$. The first condition assures that
the intertwining integrals (\ref{eq:intertwiner}) do not vanish, and that the
unramified principal series is irreducible. Although we impose this
assumption, the vanishing of the intertwining integral is precisely
the tool needed to understand the reducibility of the principal series
representations; see \cite[Theorem~I.2.9]{KazhdanPatterson} and its proof.
The second condition assures that the intertwining integrals do not
have poles. This condition can be removed later using
Proposition~\ref{prop:T-descent}, where it is shown that the Demazure-Whittaker
operators that define the recurrence formulas satisfied by Whittaker functions
do not have poles, due to cancellation of denominators.
\end{remark}

\begin{lemma}
\label{lemma:vo}
The module $i(\chi)=\operatorname{Ind}_H^{\Tn}(\chi)$ has a $T(\mathfrak{o})$-fixed vector
$v_0$ that is unique up to scalar multiple. If $t\in H$
then $t\cdot v_0=\chi(t)v_0$. The vectors $\varpi^{\lambda}\cdot v_0$
as $\lambda$ runs through a set of coset representatives for
$\Lambda/\Lambda^{(n)}$ are a basis of $i(\chi)$.
\end{lemma}

\begin{proof}
A complete set of double coset representatives
for $H\backslash\Tn/T(\mathfrak{o})$ consists of $\varpi^\lambda$
as $\lambda$ runs through a set of coset representatives
for $\Lambda/\Lambda^{(n)}$. By Mackey theory,
\[i(\chi)|_{T(\mathfrak{o})}=\bigoplus_{\lambda\in\Lambda/\Lambda^{(n)}}\xi_\lambda,\]
where $\xi_\lambda$ is the character of $T(\mathfrak{o})$ defined by
$\xi_\lambda(t)=\chi(\varpi^{\lambda}t\varpi^{-\lambda})$.
Since $\chi$ is genuine, $\xi_\lambda$ is the
trivial character only if $\varpi^\lambda$ lies
in the centralizer $H$ of $T(\mathfrak{o})$. Hence there
is, up to scalar multiple, only one $T(\mathfrak{o})$-fixed vector $v_0$.
Now another application of Mackey theory shows that $i(\chi)|_H$
contains a copy of $\chi$. Since $\chi|_{T(\mathfrak{o})}$ is trivial,
this must be the span of $v_0$, proving that $t\cdot v_0=\chi(t)v_0$
for $t\in H$.
\end{proof}

Define $\Phi_\circ^\z$ to be the unique $K$-fixed vector in $I(\z)$
normalized so that $\Phi_\circ^\z(k)=v_0$ when $k\in K$. It
may be defined by
\[\Phi_\circ^\z(bk)=\delta^{1/2}(b)b\cdot v_0, \]
according to the metaplectic version of the Iwasawa
decomposition $\Gn = \widetilde{B} K$.
Using $v_0$ here guarantees that this is well-defined.

Denote by $J^+$ and $J^-$ the positive and negative Iwahori subgroups of
$\mathbf{G}(\mathfrak{o})$. Thus if $k=\mathfrak{o}/\varpi\mathfrak{o}$ is the residue
field, these are the preimages of $\mathbf{B}(k)$ and $\mathbf{B}^-(k)$ respectively under
reduction modulo $\varpi$. We will work mostly with $J^-$ and so we will
denote it as~$J$.

Let $I(\z)^{J}$ be the space of Iwahori fixed vectors in $I(\z)$. The dimension of this space is $|W|$ according to~\cite[Proposition 5.1]{McNamaraCS}. Moreover, we have the decomposition $\Gn = \sqcup_{w\in W} \Bn w J$, where we have chosen a representative for $w$ in $K$, which (by abuse of notation)
we will also denote as $w$. This allows us to define a basis $\Phi^\z_w \in I(\z)^{J}$ for $w\in W$ as follows:
\begin{equation}\label{def:Iwahoribasis}
\Phi_w^\z (b w' l) := 
\begin{cases}
       \Phi^\z_\circ (b) &\text{ if } w'=w \\
      0 & \text{ otherwise}, \\
     \end{cases} 
     \qquad b \in \Bn, w' \in W, l \in J.
\end{equation}
From this definition it is easy to see that $\sum_{w\in W} \Phi_w^\z = \Phi_\circ^\z$. 

We next describe basics about Whittaker functionals on metaplectic principal series.
Recall that $\mathbf{N}^-(F)$ may be embedded in $\Gn$ using the trivial section. In
fact, later results in this section will require the following slightly stronger condition given as
Proposition~(b) of Appendix 1 of~\cite{MWbook}, which is attributed to Deligne. We refer the
reader to \cite{MWbook} for the proof, which is stated for global fields but works similarly for local
fields as noted in the final sentence of the appendix.

\begin{proposition}[Deligne] There exists a unique $\mathbf{T}(F)$-equivariant section $s_F$
of $\mathbf{N}^-(F)$ into $\widetilde{G}$ lifting the natural projection map $p$.
\label{prop:deligne} \end{proposition}

Let $\psi$ be a non-degenerate character of $\mathbf{N}^-(F)$ viewed as a
subgroup of $\Gn$. We assume that if $\alpha^\vee$ is a simple root, then the
character $\psi\circ x_{-\alpha^\vee}$ of $F$ is trivial on $\mathfrak{o}$ but
no larger fractional ideal.
The space of complex valued Whittaker functionals
\begin{equation}\label{eq:Wz}
\W_\z :=\{ f : I(\z) \to \C, f(nv) = \psi(n) f(v) \text{ for all } n\in \mathbf{N}^-(F), v \in I(\z)\}
\end{equation}
has a natural basis $\Omega^\mathbf{z}_{\mu}$ indexed by coset representatives of the form $\varpi^\mu$ for $\Tn/H$.
We will now define this basis explicitly. 

We begin by considering the map $\Lambda\to i(\chi)$
defined by
\begin{equation}
\label{ichibasis}
\lambda\mapsto \z^{-\lambda}\varpi^\lambda\cdot v_0.
\end{equation}
If $\lambda\in\Lambda^{(n)}$ then $\varpi^\lambda\in H$
and so by Lemma~\ref{lemma:vo} this map is constant on
the cosets of $\Lambda^{(n)}$. 
The image of our set of representatives $\Gamma$ of $\Lambda / \Lambda^{(n)}$ under the map~\eqref{ichibasis} is a basis of~$i(\chi)$.
Let $\mathcal{L}_\gamma=\mathcal{L}_\gamma^\z$ enumerated by $\gamma \in \Gamma$ be the
dual basis of $i(\chi)^\ast$, so
\begin{equation}
  \label{eq:functional_norm}
  \mathcal{L}^\z_\gamma(\varpi^{\lambda}\cdot v_0) =
  \begin{cases*}
    \z^\lambda & if $\lambda - \gamma\in\Lambda^{(n)}$,\\
    0 & otherwise.
  \end{cases*}
\end{equation}
Note that $\mathcal{L}^\mathbf{z}_\gamma$ does not depend on the choice of
coset representative $\gamma$.
Thus we may enumerate the basis by $\mathcal{L}^\mathbf{z}_\mu$ with any $\mu \in \Lambda / \Lambda^{(n)}$.

Let $\mathbf{W}^\chi : I(\chi) \to i(\chi)$ be the $i(\chi)$-valued Whittaker function
\begin{equation}
  \label{eq:W-chi}
  \mathbf{W}^\chi(f) = \int_{\mathbf{N}^-(F)} f(n)\psi(n)^{-1} \, dn \, .
\end{equation}
We then define the basis $\{\Omega^\z_\mu \}$ of $\W_\z$ as
\begin{equation}\label{mWhittakerdefinition} 
  \Omega^\z_\mu := \mathcal{L}_\mu^\z \circ \mathbf{W}^\chi, \qquad \mu \in \Lambda / \Lambda^{(n)}.
\end{equation}

The metaplectic Iwahori-Whittaker functions we consider are with respect to the contragredient representation $I(\mathbf{z}^{-1})$ and defined as
\begin{equation}
  \label{eq:phi}
  \phi_{\mu, w}(\z; g) := \delta^{1/2}(g) \Omega_\mu^{\z^{-1}}\bigl(\pi(g) \Phi_w^{\z^{-1}}\bigr) \qquad g \in \widetilde G, \quad \mu \in \Lambda/\Lambda^{(n)}.
\end{equation}
\begin{remark}
  There are different normalizations of the Whittaker function in the literature.
  We use the same normalization for the Whittaker function as
  in~\cite{BBBF}.
  However, the normalization for $\mathcal{L}_\mu$ is different
  in~\cite{BBBF}, which compensates for the extra factor in equation~(48) of
  the same paper.
  The normalization used in both~\cite{BBB} and~\cite{McNamaraCS} is different from ours.
\end{remark}

These Whittaker functions are invariant with respect to $J$ acting on the right and are $(\mathbf{N}^-(F),\psi)$-invariant on the left. They are also genuine.
Since $\Gn = \sqcup_{w\in W} \Bn w J$ and $\Bn = \mathbf{N}(F)\Tn$ we then have that $\phi_{\mu,w}(\mathbf{z}; g)$ is determined by its values on $g = \varpi^{-\lambda} w'$ where $\lambda \in \Lambda$ and $w' \in W$.

Our main goal in this section is the determination of these metaplectic Iwahori-Whittaker functions $\phi_{\mu, w}(\z; g)$ on all such elements $g = \varpi^{-\lambda} w'$ with $\lambda \in \Lambda$ and $w' \in W$.  In the next subsection, we provide this answer in terms of recursively defined ``metaplectic Demazure operators'' --- certain divided difference operators made with data from our metaplectic cover. In the remainder of this subsection, we prove two preliminary results on the vanishing of these functions at certain elements of the above form and evaluate $\phi_{\mu, w}$ at such $g$ in the special case where $w' = w$. This latter result may be viewed as the base case for the recursive relations to follow.

The proofs of these initial results are similar to the special case where the cover is trivial, that is, the linear algebraic group case. The proofs in this special case may be found in~\cite{BBBGIwahori}, but require the section on $\mathbf{N}^{-}(F)$ to be $\mathbf{T}(F)$-equivariant as in Proposition~\ref{prop:deligne}.

\begin{lemma}
  \label{lem:almost-dominant}
  Let $\lambda \in \Lambda$ and $w' \in W$.
  Then $\phi_{\mu,w}(\mathbf{z}; \varpi^{-\lambda} w') = 0$ unless $\lambda$
  is $w'$-almost dominant (as in Definition~\ref{def:almost_dominant}).
\end{lemma}

\begin{proof}
  The proof follows \cite[Lemma~3.5]{BBBGIwahori} and \cite[Lemma~5.1]{CasselmanShalika}.
  Denote $\phi_{\mu,w}(\mathbf{z}; g)$ by $\phi(g)$. Assume that $\lambda$ is
  not $w'$-almost dominant and let $\alpha_i$ be a simple root which does not
  satisfy~\eqref{eq:almost-dominant}.
  By the construction of the character $\psi$ there exists a $t \in \frac1\varpi \mathfrak{o}$ such $\psi \circ x_{-\alpha_i}(t) \neq 1$.
  Then, since $u := x_{-\alpha_i}(t) \in \mathbf{N}^{-}(F)$,
  \begin{equation}
    \label{eq:almost-dominant-j-invariant}
    \psi(u) \phi(\varpi^{-\lambda} w') = \phi(u \varpi^{-\lambda} w') = \phi(\varpi^{-\lambda} w' j)
  \end{equation}
  where $j = w'^{-1} \varpi^\lambda u \varpi^{-\lambda} w' = x_{-w'^{-1}\alpha_i}(\varpi^{-\langle \alpha_i, \lambda \rangle} t)$. No cocycles appear in this multiplication because of the $\mathbf{T}(F)$-equivariant section on $\mathbf{N}^-(F)$ and the splitting of the cover over $K$.
  We must now analyze two possible cases:
  \begin{itemize}[label=$\boldsymbol{\cdot}$, leftmargin=*]
    \item $w'^{-1}\alpha_i \in \Delta^+$ and by our assumption this means that $\langle \alpha_i, \lambda \rangle < 0$ which implies $\varpi^{-\langle \alpha_i, \lambda \rangle} t \in \mathfrak{o}$,
    \item $w'^{-1}\alpha_i \in \Delta^-$ which by our assumption implies that $\langle \alpha_i, \lambda \rangle < -1$ and $\varpi^{-\langle \alpha_i, \lambda \rangle} t \in \mathfrak{p}$.
  \end{itemize}
  In either case $j \in J$, and since $\phi$ is $J$-invariant, \eqref{eq:almost-dominant-j-invariant} gives that $\psi(u) \phi(\varpi^{-\lambda} w') = \phi(\varpi^{-\lambda} w')$ which must therefore vanish.
\end{proof}

\begin{proposition} \label{prop:rep-base-case} 
  Let $\mu \in \Lambda / \Lambda^{(n)}$, $w \in W$ and $\lambda \in \Lambda$ a $w$-almost dominant weight. Then
\begin{equation}
  \phi_{\mu,w}(\z; \varpi^{-\lambda} w) = 
  \begin{cases*}
    v^{\ell(w)} \mathbf{z}^{\lambda} & if $\mu + \lambda \in \Lambda^{(n)}$ \\
    0 & otherwise,
  \end{cases*}
\end{equation}
where $\ell(w)$ denotes the length of a reduced expression for $w$.
\end{proposition}

\begin{proof}
  This proof proceeds similarly to Proposition 3.6 in \cite{BBBGIwahori}, where the cover is trivial.
  By definition
  \begin{equation*}
    \begin{split}
      \phi_{\mu,w} (\mathbf{z}; \varpi^{- \lambda} w) &= 
      \delta^{1/2}(\varpi^{-\lambda}) \Omega^{\mathbf{z}^{-1}}_{\mu}\bigl(\pi(\varpi^{-\lambda}w)\Phi^{\mathbf{z}^{-1}}_w\bigr) \\
      &=
      \delta^{1/2}(\varpi^{-\lambda}) \mathcal{L}^{\mathbf{z}^{-1}}_{\mu} \left(
          \int_{\mathbf{N}^-(F)} \Phi^{\mathbf{z}^{-1}}_w(n \varpi^{-\lambda} w) \psi(n)^{-1} \, dn
        \right).
    \end{split}
  \end{equation*}
   We may apply the variable change $n \mapsto \varpi^{- \lambda} n
  \varpi^{\lambda}$ to the above integral without introducing cocycles, according to the $\mathbf{T}(F)$-equivariant section on $\mathbf{N}^-(F)$ from Proposition~\ref{prop:deligne}. This change of variables does produce a change in measure by $\delta
  (\varpi^{\lambda})$. Since $\Phi_w^{\mathbf{z}^{- 1}} (\varpi^{-
  \lambda} g) = \delta^{1 / 2} (\varpi^{- \lambda}) \mathbf{z}^{\lambda}
  \Phi_w^{\mathbf{z}^{- 1}} (g)$ we may write
  \begin{equation}
  \label{eq:base-case-integral}
    \delta^{1/2}(\varpi^{-\lambda}) \int_{\mathbf{N}^-(F)} \Phi^{\mathbf{z}^{-1}}_w(n \varpi^{-\lambda} w) \psi(n)^{-1} \, dn = 
    \mathbf{z}^{\lambda} \int_{\mathbf{N}^- (F)} \Phi^{\mathbf{z}^{- 1}}_w (n w)
     \psi (\varpi^{- \lambda} n \varpi^{\lambda})^{- 1} d n.
  \end{equation}
  The proof now proceeds as in Proposition~3.6 of~\cite{BBBGIwahori} by use of the Iwahori factorization for $J_w = wJw^{-1}$,
  which allows one to conclude that the integrand is supported on $N_w^-$ where
  \begin{equation}
    \label{nmindef} N_w^- = \prod_{\alpha \in \Delta^+} \left\{
    \begin{array}{ll}
      x_{- \alpha} (\mathfrak{o}) & \text{if $w^{- 1} \alpha \in
      \Delta^+$\;,}\\
      x_{- \alpha} (\mathfrak{p}) & \text{if $w^{- 1} \alpha \in \Delta^-$\;.}
    \end{array} \right.
  \end{equation}
  To complete the proof, we show that the value of the integrand is $\Phi^{\mathbf{z}^{-1}}_\circ(1) = v_0$ so \eqref{eq:base-case-integral} is just $\mathbf{z}^{\lambda}v_0$ times the volume of $N_w^-$. 
  From \eqref{def:Iwahoribasis} it follows that $\Phi_w^{\mathbf{z}^{- 1}} (n w) = \Phi^{\mathbf{z}^{-1}}_\circ(1)$ since the argument of the former is in $w J$. 
  It remains to show that $\varpi^{- \lambda} n \varpi^{\lambda}$ is in the kernel of $\psi$. 
  For this it is sufficient to show that if $\alpha = \alpha_i$ is a
  simple positive root then
  \[ \varpi^{- \lambda} x_{- \alpha_i} (t) \varpi^{\lambda} \in \mathbf{N}^-
     (\mathfrak{o}) \]
  where, using \eqref{nmindef}, we may assume that $t \in \mathfrak{o}$ if $w^{-
  1} (\alpha_i) \in \Delta^+$ and $t \in \mathfrak{p}$ otherwise. Now
  \[ \varpi^{- \lambda} x_{- \alpha_i} (t) \varpi^{\lambda} = x_{- \alpha_i}
     (\varpi^{\langle \lambda, \alpha_i \rangle} t), \]
  and because $\lambda$ is $w$-almost dominant, $\varpi^{\langle \lambda,
  \alpha_i \rangle} t$ is in $\mathfrak{o}$. Again no cocycles appear from our choice of splitting.
  The volume of $N_w^{-}$ is $v^{\ell(w)}$.
  From the definition (\ref{eq:functional_norm}) we then have
  \begin{equation*}
    \phi_{\mu,w} (\mathbf{z}; \varpi^{- \lambda} w) = v^{\ell(w)} \mathcal{L}^{\mathbf{z}^{-1}}_{\mu} ( \mathbf{z}^\lambda v_0 ) = v^{\ell(w)} \mathcal{L}^{\mathbf{z}^{-1}}_{\mu} (\varpi^{-\lambda} \cdot v_0 ) =
    \begin{cases*}
      v^{\ell(w)} \mathbf{z}^{\lambda} & if $\mu + \lambda \in \Lambda^{(n)}$ \\
      0 & otherwise.
    \end{cases*} \qedhere
  \end{equation*}
\end{proof}

\subsection{Intertwining operators and recurrence relations}
\label{sec:recursionrelations}

Define the intertwining operator
$\mathcal{A}_w^\z : I(\z) \to I(w \z)$ as
\begin{equation}
  \label{eq:intertwiner}
  \mathcal{A}_w^\z (f) (g) := \int_{\mathbf{N}(F) \cap w\mathbf{N}^-(F)w^{-1}} f(w^{-1} n g)dn
\end{equation}
when the integral is absolutely convergent, and by analytic continuation
otherwise. Because $\chi_\z$ is unramified, this definition does not depend on
the representative of the Weyl group element.

The intertwining operator can be restricted to the space of Iwahori fixed
vectors, and we can explicitly compute the action of this operator. 

\begin{lemma}
  \label{lem:intertwiner-Phi}
Let $s_i\in W$ be a simple reflection and $w\in W$. Then
\begin{equation}\label{intertwiningactionIwahori}
\mathcal{A}_{s_i}^\z (\Phi^\z_{w}) =      \begin{cases}
       (1-c_{\alpha_i}(s_i \z) )\Phi_w^{s_i \z} + \Phi^{s_i \z}_{s_i w} &\quad\text{if } \ell(s_i w) > \ell(w),\\
       (v - c_{\alpha_i}(s_i \z)) \Phi^{s_i \z}_w+ v \Phi^{s_i \z}_{s_i w} &\quad\text{if } \ell(s_i w) < \ell(w), \\
     \end{cases}
\end{equation}
where
\begin{equation}\label{calphai}
  c_{\alpha_i}(\z) = \frac{1-v \z^{n_i\alpha_i}}{1-\z^{n_i\alpha_i}}.
\end{equation}
Moreover
\begin{equation}
\label{asasconst}
\mathcal{A}_{s_i}^{s_i\z}\mathcal{A}_{s_i}^{\z}={c_{\alpha_i}(\z)c_{-\alpha_i}(\z)}.
\end{equation}
\end{lemma}

This is a generalization of
Casselman~\cite[Theorem~3.4]{CasselmanSpherical}. It is worth noting
that
\[1-c_{\alpha_i}(s_i\z)=c_{\alpha_i}(\z)-v,\]
so there are different ways of writing~(\ref{intertwiningactionIwahori}).

\begin{proof}[Proof of the Lemma]
The case $\ell(s_iw)>\ell(w)$ of \eqref{intertwiningactionIwahori}
follows from \cite[Lemma~6.3]{PatnaikPuskasIwahori} using
$1-c_{\alpha_i}(s_i\z)=1-c_{-\alpha_i}(\z)=\frac{1-v}{1-\z^{n_i\alpha_i}}$.
The second case of \eqref{intertwiningactionIwahori} follows from the
first case by interchanging the roles of $w$ and $s_iw$, after
some algebra.
For \eqref{asasconst},
see \cite[Proposition~4.4]{McNamaraCS}. 
\end{proof}

The intertwining integrals induce maps $\W_{\z^{-1}} \to \W_{w\z^{-1}}$ that can be expanded as
\begin{equation}
  \label{eq:Omega-intertwining}
  \Omega^{w\z}_{\mu} \circ \mathcal{A}^{\z}_w = \sum_{\nu \in \Lambda / \Lambda^{(n)}} \tau_{\mu,\nu}(w; \mathbf{z}) \Omega^{\z}_{\nu}.
\end{equation}
where $\tau_{\mu,\nu}$ is called the Kazhdan-Patterson scattering matrix that will be given explicitly later.
We will often apply it for a simple reflection $w = s_i$ and will then drop the Weyl group element from the notation $\tau_{\mu,\nu}(s_i; \z) = \tau_{\mu,\nu}(\z)$.
Applying this twice to the composition
$\Omega^\mathbf{z}_\mu\mathcal{A}_{s_i}^{s_i\z}\mathcal{A}_{s_i}^\z$ and using
(\ref{asasconst}) gives
\begin{equation}
  \label{asastau}
  \sum_{\nu \in \Lambda / \Lambda^{(n)}}\tau_{\mu,\nu}(s_i\z)\tau_{\nu,\lambda}(\z)=
  \left\{\begin{array}{cl}c_{\alpha_i}(\z)c_{-\alpha_i}(\z)&
\text{if $\lambda\equiv\mu$ mod $\Lambda^{(n)}$,}\\
0&\text{otherwise}.\end{array}\right.
\end{equation}

The Whittaker functions we study are with respect to the contragredient representation $I(\mathbf{z}^{-1})$, for which equation~\eqref{intertwiningactionIwahori} becomes
\begin{equation}\label{intertwiningactionIwahoriexplicit}
  \mathcal{A}_{s_i}^{\z^{-1}} (\Phi^{\z^{-1}}_{w}) 
  =
  \frac{1}{1-\z^{-n_i\alpha_i}}   
  \begin{cases*}
    (1-v) \Phi_w^{s_i \z^{-1}} + (1-\z^{-n_i\alpha_i})\Phi^{s_i \z^{-1}}_{s_i w} & if $\ell(s_i w) > \ell(w)$,\\
    (1-v) \z^{-n_i \alpha_i}  \Phi^{s_i \z^{-1}}_w+v(1-\z^{-n_i\alpha_i})\Phi^{s_i \z^{-1}}_{s_i w} & if $\ell(s_i w) < \ell(w)$. \\
  \end{cases*}
\end{equation}

Let $g(a), a \in \mathbb{Z}$ be the Gauss sum as defined in~\cite[Section 13]{McNamaraCS}, although with a different normalization.
If $\tilde{g}(a)$ are the Gauss sums in~\cite{McNamaraCS}, then $g(a)=v\tilde{g}(a)$.
For our purposes, it is important to note that the Gauss sums are periodic modulo $n$ and satisfy 
\begin{equation}
\label{eq:padicgaussproperties}
g(0)=-v, \quad \quad g(a)g(n-a)=v \quad \text{for $a \not\equiv 0 \bmod{n}$}
\end{equation}
which is Assumption~\ref{assumptionga}.

The following result was proved for $\GL_r$ by
Kazhdan and Patterson~\cite[Lemma~I.3.3]{KazhdanPatterson},
applied to the metaplectic Casselman-Shalika formula in
\cite{ChintaOffen}, and generalized by \cite{McNamaraCS} to a much larger family of
covers for reductive groups, including those we work with here.

\begin{proposition}[Kazhdan-Patterson, McNamara]
\label{propositionKP}
  For a simple reflection $s_i$ and $\mu, \nu \in \Lambda / \Lambda^{(n)}$, 
  \begin{equation}
    \label{eq:Omega-intertwining-si}
    \Omega^{s_i\z^{-1}}_{\mu} \circ \mathcal{A}^{\z^{-1}}_{s_i} = \sum_{\nu \in \Lambda / \Lambda^{(n)}} \tau_{\mu,\nu}(\mathbf{z}^{-1}) \Omega^{\z^{-1}}_{\nu}
  \end{equation}
  where the coefficients $\tau_{\mu,\nu}$ can be written as 
\[ \tau_{\mu,\nu}  = \tau^1_{\mu,\nu} + \tau^2_{\mu,\nu} \]
and where $\tau^1$ vanishes unless $\nu \sim \mu \bmod \Lambda^{(n)}$ and $\tau^2$ vanishes unless $\nu \sim s_i(\mu) + \alpha_i \bmod \Lambda^{(n)}$. Moreover,
\begin{equation}
  \label{tau1} 
  \tau^1_{\mu, \mu}(\mathbf{z}^{-1}) = \frac{1 - v}{1- \z^{-n_i \alpha_i
}} \, \z^{ - \left( \left\lfloor -\frac{B(\alpha_i, \mu)}{Q(\alpha_i)} \right\rfloor_{\!n_i} \right) \alpha_i}
\end{equation}
where $\floorn{x}$ denotes the least positive residue of $x \bmod{n}$, and
\begin{equation}
  \label{tau2} 
  \tau^2_{s_i\mu+\alpha_i, \mu}(\mathbf{z}^{-1}) = g\bigl(B(\alpha_i, \mu)-Q(\alpha_i)\bigr) \z^{\alpha_i}.
\end{equation}
\end{proposition}

\begin{proof}
We refer the reader to McNamara~\cite[Theorem~13.1]{McNamaraCS}, but make several comments about necessary modifications in the present set-up.
Our normalization for the Whittaker functional $\Omega_\mu^{\mathbf{z}^{-1}}$ is slightly different accounting for the different $\mathbf{z}$-exponents
in \eqref{tau1} and \eqref{tau2}, and we also use a different normalization for $\mathcal{A}_w^{\mathbf{z}^{-1}}$ explaining an overall factor of $c_{\alpha_i}(\mathbf{z}^{-1})$. Furthermore, we are applying this formula at $\z^{-1}$,
so McNamara's $x_\alpha$ is our $\z^{-\alpha_i}$. Thus
\begin{equation}
  \label{tau1-alt}
  \tau^1_{\mu, \mu}(\mathbf{z}^{-1}) = \frac{1 - v}{1- \z^{-n_i \alpha_i }} \, \z^{-n_i \ceil*{ \frac{B(\alpha_i, \mu)}{n_i Q(\alpha_i)} } \alpha_i - s_i(\mu) + \mu}  
\end{equation}
where $\ceil{x}$ denotes the smallest integer at least $x$.
From the bilinearity and $W$-invariance of $B$ we get that $s_i(\mu)-\mu=-\frac{B(\alpha_i,\mu)}{Q(\alpha_i)}\alpha_i$ which also tells us that $\frac{B(\alpha_i,\mu)}{Q(\alpha_i)}$ is an integer.
Then, using $n\ceil{\frac{x}{n}} - x = \floorn{-x}$ for $x \in \mathbb{Z}$ we obtain (\ref{tau1}).
\end{proof}

Note that~\eqref{tau2} may also be written as
\begin{equation}
  \label{tau2-alt} 
  \tau^2_{\mu, s_i\mu+\alpha_i}(\mathbf{z}^{-1}) = g\bigl(Q(\alpha_i)-B(\alpha_i, \mu)\bigr) \z^{\alpha_i}.
\end{equation}
As seen from~\eqref{eq:Omega-intertwining-si} the coefficients $\tau_{\mu,\nu}$ are independent of the coset representatives for $\mu$ and $\nu$ since $\Omega_\mu^{s_i\mathbf{z}^{-1}}$ and $\Omega_\nu^{\mathbf{z}^{-1}}$ are, which is an advantage of our normalization.
Furthermore, we see directly that the exponent in~\eqref{tau1} is invariant under shifts $\mu \to \mu + \lambda$ with $\lambda \in \Lambda^{(n)}$ since $B(\alpha_i, \lambda) \equiv 0 \bmod{n}$ which one can show is equivalent to $B(\alpha_i, \lambda)/Q(\alpha_i) \equiv 0 \bmod{n_i}$.

Combining the two formulas~\eqref{intertwiningactionIwahoriexplicit} and~\eqref{eq:Omega-intertwining-si} we can prove the following functional equation for metaplectic Iwahori Whittaker functions.
 
\begin{proposition}
  \label{prop:rep-recursion}
  The following functional equation holds:
  \begin{multline}
    \label{eq:rep-recursion}
    \tau^1_{\mu, \mu}(\mathbf{z}^{-1}) \phi_{\mu,w}(\mathbf{z}; g) + \tau^2_{\mu, s_i(\mu)+\alpha_i}(\mathbf{z}^{-1}) \phi_{s_i(\mu)+\alpha_i, w}(\mathbf{z};g)
    = \\
    \frac{1}{1-\z^{-n_i\alpha_i}}
    \begin{cases*}
      (1-v) \phi_{\mu,w}(s_i\mathbf{z};g) + (1-\mathbf{z}^{-n_i\alpha_i}) \phi_{\mu, s_iw}(s_i\mathbf{z}) & if $\ell(s_iw) > \ell(w)$, \\
      (1-v) \mathbf{z}^{-n_i\alpha_i} \phi_{\mu, w}(s_i \mathbf{z}) + v(1-\mathbf{z}^{-n_i\alpha_i}) \phi_{\mu,s_iw}(s_i \mathbf{z}) & if $\ell(s_iw) < \ell(w)$.
    \end{cases*}
  \end{multline}
\end{proposition}
\begin{proof}
The intertwining integral is a $\Gn$-homomorphism and therefore
\[ \Omega_\mu^{s_i \z^{-1}} \!\circ \A_{s_i}^{\z^{-1}} \bigl(\pi(g) \Phi^{\z^{-1}}_{w}\bigr) = \Omega_\mu^{s_i \z^{-1}} \! \bigl(\pi(g)  \A_{s_i}^{\z^{-1}} \Phi^{\z^{-1}}_{w} \bigr).\]
By the definition of $\phi$ in \eqref{eq:phi} the statement now follows by expanding the left- and right-hand sides using~\eqref{intertwiningactionIwahoriexplicit} and~\eqref{eq:Omega-intertwining}. 
\end{proof}

\subsection{Metaplectic Demazure-Whittaker operators}
\label{sec:DW-operators}
Recall that $\mathbb{C}[\Lambda]$ is the ring of Laurent polynomials in $\mathbf{z}^\lambda$ with $\lambda \in \Lambda$ over $\mathbb{C}$ and that $\mathbb{C}(\Lambda)$ is the corresponding field of fractions.
The results of the previous section allow us to compute all
Iwahori Whittaker functions recursively. Let us fix $g\in\widetilde{G}$
and $\mathbf{\z}\in\widehat{\mathbf{T}}(\mathbb{C})$. For fixed
$w\in W$, assemble the values $\boldsymbol{\phi}_w^\gamma \coloneqq \phi_{\gamma,w}(\mathbf{z};g)$
($\gamma\in\Gamma$) into a vector
\[\boldsymbol{\phi}_w=\boldsymbol{\phi}_w(\z;g)=(\boldsymbol{\phi}_w^\gamma)_{\gamma\in\Gamma}\in\mathbb{C}(\Lambda)^{|\Gamma|}.\]

\begin{theorem}
\label{thm:dem_recurse}
Define operators $\mathbf{T}_i$ on $\mathbb{C}(\Lambda)^{|\Gamma|}$
by
      \begin{equation}
        \label{eq:Ti}
        \mathbf{T}_i = D_{n_i}(\mathbf{z}^{-1}) + s_i\boldsymbol{\tau}(\mathbf{z}^{-1}) \qquad D_{n_i}(\mathbf{z}^{-1}) = \frac{1 - v}{\mathbf{z}^{n_i \alpha_i}-1}
      \end{equation}
where, for $\mathbf{f} \in \mathbb{C}(\Lambda)^{|\Gamma|}$, $\boldsymbol{\tau}(\mathbf{z}^{-1})\mathbf{f} \in \mathbb{C}(\Lambda)^{|\Gamma|}$ with components 
      \begin{equation}  
        \label{eq:tau-operartor}
        \bigl(\boldsymbol{\tau}(\mathbf{z}^{-1})\mathbf{f}\bigr)^\mu = \tau^1_{\mu,\mu}(\mathbf{z}^{-1}) \mathbf{f}^\mu + \tau^2_{\mu, s_i(\mu)+\alpha_i}(\mathbf{z}^{-1}) \mathbf{f}^{s_i(\mu)+\alpha_i}.
      \end{equation}
The operator $\mathbf{T}_i$ is invertible and
      \begin{equation}
        \label{eq:Ti-inverse}
        \mathbf{T}_i^{-1} = \frac{1}{v}\Bigl(\mathbf{z}^{n_i \alpha_i} D_{n_i}(\mathbf{z}^{-1}) + s_i\boldsymbol{\tau}(\mathbf{z}^{-1}) \Bigr).
      \end{equation}
Furthermore,
\begin{equation}
   \label{eq:Ti-recursion}
   \boldsymbol{\phi}_{s_iw} = 
   \begin{cases*}
      \mathbf{T}_i \boldsymbol{\phi}_w & if $\ell(s_iw) > \ell(w)$ \\
      \mathbf{T}_i^{-1} \boldsymbol{\phi}_w & if $\ell(s_iw) < \ell(w)$. \\
   \end{cases*}
\end{equation}
\end{theorem}

\begin{proof}
Both cases of (\ref{eq:Ti-recursion}) follow from
Proposition~\ref{prop:rep-recursion}. That the operator
in \eqref{eq:Ti-inverse} is actually the inverse of 
$\mathbf{T}_i^{-1}$ can be checked by direct computation using~\eqref{asastau}.
However we will instead make use of a fact that we will
prove below in Theorem~\ref{thm:hecke}, namely the Hecke relation
$(\mathbf{T}_i - v)(\mathbf{T}_i + 1) = 0$. This implies
that $\mathbf{T}_i$ is invertible with inverse
$\frac 1v(\mathbf{T}_i+1-v)$. Together with the fact that
$D_{n_i}(\z)+1-v=\z^{n_i\alpha}D_{n_i}(\z)$, this
proves (\ref{eq:Ti-inverse}).
\end{proof}

We will call the operators $\mathbf{T}_i$ \textit{vector metaplectic
Demazure-Whittaker operators}. Now let us consider how these operators
allow us to compute arbitrary Iwahori Whittaker functions on
$\widetilde{G}$. Since a complete set of coset representatives for
$\mathbf{N}^-(F)\backslash\widetilde{G}/J\mu_n$ consists of elements of the
form $g = \varpi^{-\lambda}w'$ with $\lambda\in\Lambda$ and $w'\in
W$ we may restrict to such $g$. Then $\boldsymbol{\phi}_w=0$ unless $\lambda$ is $w'$-almost
dominant. Assuming this, we have the following explicit algorithm.

\begin{corollary}
  \label{cor:evaluate}
  Let $w, w' \in W$ and $\lambda$ a $w'$-almost dominant weight,
  and let $g=\varpi^{-\lambda}w_2.$
  Let $\boldsymbol{\phi}_w = \boldsymbol{\phi}_w(\mathbf{z}; g) \in \mathbb{C}(\Lambda)^{|\Gamma|}$ as above whose components are given by $\boldsymbol{\phi}_w^\gamma = \phi_{\gamma,w}(\mathbf{z}; g)$.
  Consider any path in the Weyl group from $w'$ to $w$ given by simple reflections $s_{i_k}, \ldots, s_{i_1}$ as $w' \to s_{i_1}w' \to \cdots \to s_{i_k} \cdots s_{i_1} w' = w$ and let $e_{i_j} = 1$ if $s_{i_j}$ is an ascent in this path in the Bruhat order, and otherwise let $e_{i_j} = -1$.
  Then,
  \begin{equation}
    \label{eq:evaluate}
    \boldsymbol{\phi}_w(\mathbf{z}; \varpi^{-\lambda}w') = v^{\ell(w')} \mathbf{T}_{i_k}^{e_{i_k}} \cdots   \mathbf{T}_{i_1}^{e_{i_1}} \mathcal{L}(\mathbf{z}^\lambda) ,
  \end{equation}
  where $\mathcal{L}(\mathbf{z}^\lambda) \in \mathbb{C}(\Lambda)^{|\Gamma|}$ with components $\bigl(\mathcal{L}(\mathbf{z}^\lambda)\bigr)^\gamma = \mathcal{L}^{\mathbf{z}^{-1}}_\gamma(\mathbf{z}^\lambda v_0) = \mathcal{L}^{\mathbf{z}^{-1}}_\gamma(\varpi^{-\lambda} \cdot v_0)$.
\end{corollary}

\begin{proof}
  The value of $\boldsymbol{\phi}_{w'}(\mathbf{z}; \varpi^{-\lambda}w')$ is given explicitly by Proposition~\ref{prop:rep-base-case} as $v^{\ell(w')}\mathcal{L}(\mathbf{z}^\lambda)$.
  Then Theorem~\ref{thm:dem_recurse} shows that $\boldsymbol{\phi}_w(\mathbf{z}; \varpi^{-\lambda}w')$ can be obtained recursively from $\boldsymbol{\phi}_{w'}(\mathbf{z}; \varpi^{-\lambda}w')$ by applying operators $\mathbf{T}_i$ and $\mathbf{T}_i^{-1}$ in the stated order.
\end{proof}

Demazure operators, as the term is commonly understood, are ``divided
difference'' operators such as $f \mapsto (1-\z^{-\alpha_i})^{-1}\bigl(f-\z^{-\alpha_i}s_i(f)\bigr)$
that at first glance \textit{appear} to introduce denominators, but actually do not
since the numerator is automatically divisible by the denominator.
In our next proposition we will show that the space
$\bigoplus_{\gamma \in \Gamma} \mathbf{z}^{-\gamma} \mathbb{C}[\Lambda^{(n)}]$
is $\mathbf{T}_i$-invariant meaning that on this space of polynomials $\mathbf{T}_i$ does not introduce denominators.
In particular, since the vector
$\boldsymbol{\phi}_{w'}(\mathbf{z}; \varpi^{-\lambda}w') = v^{\ell(w')}\mathcal{L}(\mathbf{z}^\lambda)$ is in
$\bigoplus_{\gamma \in \Gamma} \mathbf{z}^{-\gamma} \mathbb{C}[\Lambda^{(n)}]$, this means that $\boldsymbol{\phi}_{w}(\mathbf{z}; \varpi^{-\lambda}w') = v^{\ell(w')} \mathbf{T}_{i_k}^{e_{i_k}} \cdots   \mathbf{T}_{i_1}^{e_{i_1}} \mathcal{L}(\mathbf{z}^\lambda)$ from~\eqref{eq:evaluate} also is, as expected from~\eqref{eq:functional_norm} and the definition of
$\phi_{\mu,w}(\mathbf{z};g)$ in~\eqref{eq:phi}.
 
\begin{proposition}
  \label{prop:T-descent}
  Let $R$ be $\mathbb{C}[\Lambda^{(n)}]$ or $\mathbb{C}(\Lambda^{(n)})$.
  The operators $\mathbf{T}_i$ and $\mathbf{T}_i^{-1}$ on $\mathbb{C}(\Lambda)^{|\Gamma|} = \bigoplus_{\gamma \in \Gamma} \mathbb{C}(\Lambda)$ descend to operators on~$\bigoplus_{\gamma \in \Gamma} \mathbf{z}^{-\gamma} R$.
\end{proposition}

Even though our Whittaker functions are polynomial we will still need to
consider $\mathbf{T}_i$ as an operator
on~$\bigoplus_{\gamma \in \Gamma} \mathbf{z}^{-\gamma} R$ with $R
= \mathbb{C}(\Lambda^{(n)})$ in Theorem~\ref{thm:T-average}. Note
that $\mathbf{T}_i$ is not $R$-linear because of the $s_i$ that
appears in the definition~(\ref{eq:Ti}). The integrality proved
in this result allows us to remove the assumption, imposed in
Remark~\ref{rem:zassumptions}, that $\z^{n_\alpha\alpha}\ne 1$
for $\alpha\in\Delta$.

\begin{proof}
  We will prove the statement for $\mathbf{T}_i$; the proof for $\mathbf{T}_i^{-1}$ follows similarly.
  Let $D(\mathbf{z}^{-1}) \coloneqq D_{n_i}(\mathbf{z}^{-1})$ and $\mathbf{f} \in \bigl(\bigoplus_{\gamma \in \Gamma} \mathbf{z}^{-\gamma} R\bigr)$.
  By linearity it is enough to consider $\mathbf{f}$ with a single non-zero component $\gamma$ with $\mathbf{f}^\gamma = \mathbf{z}^{-\gamma} h$ where $h \in \mathbb{C}(\Lambda^{(n)})$ for $R = \mathbb{C}(\Lambda^{(n)})$ and $h = \mathbf{z}^{\lambda^{(n)}}$ with $\lambda^{(n)} \in \Lambda^{(n)}$ for $R = \mathbb{C}[\Lambda^{(n)}]$.
  Then the only non-zero components of $(\mathbf{T}_i\mathbf{f})$ are
  \begin{equation}
    \label{eq:Ti-f-nonzero}
    \begin{split}
      (\mathbf{T}_i \mathbf{f})^\gamma &= D(\mathbf{z}^{-1})\mathbf{f}^\gamma + s_i \tau^1_{\gamma,\gamma}(\mathbf{z}^{-1}) \mathbf{f}^\gamma \\ 
      (\mathbf{T}_i \mathbf{f})^{s_i\gamma + \alpha_i} &= s_i \tau^2_{s_i\gamma + \alpha_i,\gamma}(\mathbf{z}^{-1}) \mathbf{f}^\gamma =
      g(B(\alpha_i, \gamma) - Q(\alpha_i)) \mathbf{z}^{-(s_i\gamma+\alpha_i)} s_ih.
    \end{split}
  \end{equation}
  Since $s_ih \in R$, it remains to show that $(\mathbf{T}_i \mathbf{f})^\gamma \in \mathbf{z}^{-\gamma}R$.
  Using~\eqref{tau1-alt}, we get that
  \begin{equation}
    (\mathbf{T}_i \mathbf{f})^\gamma = \frac{1-v}{\mathbf{z}^{n_i \alpha_i} - 1} \Bigl( 1 - \mathbf{z}^{k n_i \alpha_i + s_i(\gamma) - \gamma} s_i \Bigr) \mathbf{z}^{-\gamma} h = \mathbf{z}^{-\gamma}(v-1) \frac{1 - \mathbf{z}^{k n_i \alpha_i}s_i}{1-\mathbf{z}^{n_i\alpha_i}}h
  \end{equation}
  where $k = \ceil*{ \frac{B(\alpha_i, \gamma)}{n_i Q(\alpha_i)} } \in \mathbb{Z}$.
  For $R = \mathbb{C}(\Lambda^{(n)})$ we see that $(\mathbf{T}_i \mathbf{f})^\gamma$ is clearly in $\mathbf{z}^{-\gamma}R$.
  For $R = \mathbb{C}[\Lambda^{(n)}]$ we could assume that $h = \mathbf{z}^{\lambda^{(n)}}$.
  We have that $s_i \lambda^{(n)} - \lambda^{(n)} = -\bigl(B(\alpha_i, \lambda^{(n)})/Q(\alpha_i)\bigr) \alpha_i$.
  Since $B(\alpha_i, \lambda^{(n)}) \equiv 0 \bmod{n}$ is equivalent to $B(\alpha_i, \lambda^{(n)})/Q(\alpha_i) \equiv 0 \bmod{n_i}$ we get that $s_i \lambda^{(n)} - \lambda^{(n)} = k^{(n)} n_i \alpha_i$ with $k^{(n)} \in \mathbb{Z}$.
  Thus, in this case
  \begin{equation}
    \frac{1 - \mathbf{z}^{k n_i \alpha_i}s_i}{1 - \mathbf{z}^{n_i\alpha_i}}h = \mathbf{z}^{\lambda^{(n)}} \frac{1 - \mathbf{z}^{(k + k^{(n)}) n_i \alpha_i}}{1 - \mathbf{z}^{n_i\alpha_i}}
  \end{equation}
  which is a geometric series in $\mathbb{C}[\Lambda^{(n)}]$.
\end{proof}

The \textit{finite Iwahori Hecke algebra} $\mathcal{H}_v$
is the $\mathbb{C}$-algebra with generators $T_i$ subject to the
quadratic relations $T_i^2=(v-1)T_i+v$ and the braid relations.
The (extended) \textit{affine Iwahori Hecke algebra}
$\widetilde{\mathcal{H}}_v$ adds
an abelian subalgebra spanned by generators $\vartheta_\lambda$
with $\lambda\in\Lambda$ subject to
$\vartheta_\lambda\vartheta_\mu=\vartheta_{\lambda+\mu}$ together with
the \textit{Bernstein relations}
\begin{equation}
  \label{bernstein} \vartheta_{\lambda} T_i - T_i \vartheta_{s_i \lambda} = (v - 1) 
  \frac{\vartheta_{\lambda} - \vartheta_{s_i \lambda}}{1 - \vartheta_{-
  \alpha_i}}.
\end{equation}
The term \text{extended} is used here since, unless $\Lambda$ is the root
lattice, this Hecke algebra is slightly larger than the Hecke algebra of the
usual affine Weyl group, which is a Coxeter group.

As a variant, we may replace $\Lambda$ here by the lattice $\Lambda^{(n)}$.
Then the Bernstein relation must be modified to read
\begin{equation}
  \label{mbernstein} \vartheta_{\lambda} T_i - T_i \vartheta_{s_i \lambda} = (v - 1) 
  \frac{\vartheta_{\lambda} - \vartheta_{s_i \lambda}}{1 - \vartheta_{-
  n_i\alpha_i}}.
\end{equation}
Let $\widetilde{\mathcal{H}}^{(n)}_v$ be the corresponding
(extended) \textit{affine metaplectic Hecke algebra}.

Recall that $v^{-1}$ is the cardinality of the residue field.
For the algebraic group $\mathbf{G}(F)$,
Iwahori and Matsumoto~\cite{IwahoriMatsumoto} showed that
$\widetilde{\mathcal{H}}_{v^{-1}}$ acts on $I(\chi)^J$ by
convolution.
For metaplectic forms and representations,
Savin~\cite{SavinHecke} computed the corresponding genuine Hecke algebra which acts by convolution in a similar manner, and showed that it is isomorphic to the algebra~$\widetilde{\mathcal{H}}^{(n)}_{v^{-1}}$. (See Remark~\ref{rem:dichotomy}.)
This convolution action is
complementary to the action of
$\widetilde{\mathcal{H}}_{v}^{(n)}$ that we are
considering here. For example, both the intertwining
operator action of
$\widetilde{\mathcal{H}}_{v}^{(n)}$ and the convolution action 
$\widetilde{\mathcal{H}}_{v^{-1}}^{(n)}$ appear together
in an action on $\bigoplus_w I(w\z)^J$ in~\cite{BBBF},
and these actions commute with each other. 

The next result gives an action of this Hecke algebra on functions.
It may be extended to an action of the affine
Hecke algebra as in~\cite[Theorem~2.2]{BBBF}. 

\begin{theorem}
  \label{thm:hecke}
   We have an action of the affine metaplectic Hecke algebra $\widetilde{\mathcal{H}}^{(n)}_v$ on $\mathbb{C}(\Lambda)^{|\Gamma|}$ in which the generator $T_i$ acts by the operator $\mathbf{T}_i$ on
   $\mathbb{C}(\Lambda)^{|\Gamma|}$ and the generator $\vartheta_\lambda$ multiplies
   $f\in\mathbb{C}(\Lambda)^{|\Gamma|}$ by the constant $\z^{-\lambda}$.
   Thus the $\mathbf{T}_i$ satisfy the braid relations and the quadratic Hecke relation $(\mathbf{T}_i - v)(\mathbf{T}_i + 1) = 0$.
\end{theorem}

\begin{proof}
We will deduce this from \cite[Theorems~1.1 and~1.2]{BBBF}. These
depend on the choice of the finite Weyl group of a root system
acting on a weight lattice $\Lambda$. We will take the root system
to be $\Delta_n$ and the weight lattice to be $\Lambda^{(n)}$.
To avoid confusion, we will write roots in $\Delta_n$ as
$\widetilde{\alpha}$. Thus $\widetilde{\alpha}=n_\alpha\alpha$
which means more precisely that if
$\widetilde{\z}\in\widehat{\mathbf{T}}'(\mathbb{C})$ and if
$\z$ is chosen so that $\hat p(\z)=\widetilde{\z}$ in terms of the
canonical surjection
$\hat p:\widehat{\mathbf{T}}(\C)\longrightarrow\widehat{\mathbf{T}}'(\C)$,
then $\widetilde{\z}^{\widetilde{\alpha}}=\z^{n_\alpha\alpha}$.

The theorems also require, for each $\widetilde{\z}$ in a
$W$-invariant open subset of the torus (which for us will be
$\widehat{\mathbf{T}}'$) a vector space $M(\widetilde{\z})$.
We also need, for each simple reflection $s_i\in W$, a linear map
$A_i^\z:M(\widetilde{\z})\to M(s_i\widetilde{\z})$.
A general schema to produce Hecke algebra representations
is given in~\cite{BBBF}, subject to a couple of assumptions.

\medbreak\noindent
\emph{Assumption~1 of~\cite{BBBF}}. The composition $A^{s_i\z}_i\circ A^{\z}_i$ is the
\begin{equation}
\label{eq:expl_c}
\left(\frac{1-v\widetilde{\z}^{\widetilde\alpha_i}}{1-\widetilde{\z}^{\widetilde\alpha_i}}\right)
\left(\frac{1-v\widetilde{\z}^{-\widetilde\alpha_i}}{1-\widetilde{\z}^{-\widetilde\alpha_i}}\right).
\end{equation}

\medbreak\noindent
\emph{Assumption~2 of~\cite{BBBF}}. Suppose that $m_{ij}$ is the order of
$s_is_j$. Let $w_{ij}=\cdots s_is_js_i = \cdots s_js_is_j$ be
the longest element of the dihedral group generated by $s_i$
and $s_j$. (There are $m_{ij}$ factors on both sides.)
Then the $A^\z_i$ braid in the sense that if $m_{ij}$ is the
order of $s_is_j$ then
\[\cdots A_i^{s_js_i\z}A_j^{s_i\z}A_{i}^\z =
\cdots A_j^{s_is_j\z}A_i^{s_j\z}A_{j}^\z\]
with $m_{ij}$ factors on both sides as maps $M(\z)\to M(w_{ij}\z)$.

\medbreak
On these two assumptions, \cite[Theorems~1.1 and 1.2]{BBBF} give
a representation of $\widetilde{\mathcal{H}}_v$ on
$\bigoplus_w M(w\z^{-1})$. (We have replaced $\z$ in~\cite{BBBF} by $\z^{-1}$.)
The operator $T_i$ sends $f\in M(w\z)$ into
\[D_i(w\z^{-1})f\,+\, A^{(w\z)^{-1}}_if\quad\in\quad  M(w\z^{-1})\oplus M(s_iw\z^{-1}),\]
where $D_i(\z^{-1})=\frac{1-v}{\z^\alpha_i-1}$.
To apply this, we will use the Hecke algebra corresponding to the
metaplectic L-group, for which $\z$ is replaced by $\widetilde{\z}$ and $\alpha_i$
by $\widetilde\alpha_i$, so
\[D_i(\widetilde{\z}^{-1})=\frac{1-v}{\widetilde{\z}^{\widetilde{\alpha}_i}-1}
=\frac{1-v}{\z^{n_i\alpha_i}-1},\]
which is $D_{n_i}(\z^{-1})$ in the notation (\ref{eq:Ti}).
We take
$M(\widetilde{\z})=\mathbb{C}(\Lambda)^{|\Gamma|}$. The map
$A_i^{\widetilde{\z}}:\mathbb{C}(\Lambda)^{|\Gamma|}\to \mathbb{C}(\Lambda)^{|\Gamma|}$
is given by the
matrix $\tau(\z^{-1})$, where $\z\in\widehat{\mathbf{T}}(\mathbb{C})$
such that $p(\z)=\widetilde{\z}$. We must verify that two
Assumptions are satisfied.

Assumption~1 is satisfied by
(\ref{asastau}), on substituting
$\widetilde{\z}^{\widetilde{\alpha}_i}=\z^{n_i\alpha_i}$ into (\ref{eq:expl_c}); Assumption~2 follows from the
property of the intertwining operators, that $\mathcal{A}_{w_1}^{w_2\z}\circ
\mathcal{A}_{w_2}^\z=\mathcal{A}_{w_1w_2}^\z$ provided $\ell(w_1w_2)=\ell(w_1)+\ell(w_2)$. Theorem~1.1
of~\cite{BBBF} then gives an action on $\mathcal{M}(\z) = \bigoplus_w M(w\z)$.
Following the proof of \cite[Theorem~2.2]{BBBF}, we may extract the action
(\ref{eq:Ti}) on functions $\mathbf{f}$ in a single copy of $\mathbb{C}(\Lambda)^{|\Gamma|}$.
To this end, consider the map
\[\mathbb{C}(\Lambda)^{|\Gamma|}\longrightarrow \bigoplus_{w\in W}\mathbb{C}(\Lambda)^{|\Gamma|},\qquad
\mathbf{f}\mapsto ({^w\mathbf{f}})_{w\in W}.\]
This copy of $\mathbb{C}(\Lambda)^{|\Gamma|}$ is stable under the action of $\mathcal{H}_v$,
and the $T_i$ induce the operators $\mathbf{T}_i$ given by (\ref{eq:Ti}). The extension
to $\widetilde{\mathcal{H}}_v$ follows from \cite[Theorem~1.2]{BBBF}. 
\end{proof}

\subsection{Averaged metaplectic Demazure-Whittaker operators}
\label{sec:scalar-Demazure}
In Theorem~\ref{thm:dem_recurse} we define metaplectic Demazure-Whittaker operators acting on a vector of metaplectic Whittaker functions $\boldsymbol{\phi}_w \in \mathbb{C}(\Lambda)^{|\Gamma|}$ similar to those in~\cite{BBBF}.
Recall that the Whittaker model of $I(\z)$ has dimension equal to the cardinality of $\Gamma$ in the metaplectic setting. Therefore it is natural to consider a vector valued Whittaker function, as this allows us to study all Whittaker functions simultaneously. 
Other literature, such as \cite{ChintaGunnellsPuskas, PatnaikPuskasIwahori,
McNamaraCS, SSV} define metaplectic Demazure-Whittaker operators acting on a
particular \textit{averaged metaplectic Whittaker function}, which we will show is the sum of the components of $\boldsymbol{\phi}_w \in \mathbb{C}(\Lambda)^{|\Gamma|}$.
These averaged Demazure-Whittaker operators are defined in the above literature by (implicitly) making use of the Chinta-Gunnells action introduced in~\cite{ChintaGunnells}, while our operators are derived purely from the representation theory of the non-archimedean group. Indeed, in our case the operators are defined from the recursion relation in Proposition~\ref{prop:rep-recursion}, which was in turn obtained from the intertwining operators and the Kazhdan-Patterson scattering matrix $\tau_{\mu,\nu}$. 
We will in this section relate the two different types of metaplectic Demazure-Whittaker operators and show that the Chinta-Gunnells action arises naturally from our vector Demazure-Whittaker operators.

Another important distinction compared to previous literature is that we do not restrict to arguments $g$ on the torus, but in fact compute the Whittaker functions for all values $g \in \widetilde{G}$.

Recall that $\mathbb{C}(\Lambda^{(n)})$ denotes the field of rational functions in $\mathbf{z}^{\lambda^{(n)}}$ with $\lambda^{(n)} \in \Lambda^{(n)}$ over $\mathbb{C}$ and denote by $\mathbb{C}(\Lambda^{(n)})[\Lambda]$ the ring of polynomials in $\mathbf{z}^\lambda$ with $\lambda \in \Lambda$ over this field.
From our chosen representatives $\Gamma$ of $\Lambda / \Lambda^{(n)}$ we have an isomorphism
\begin{equation}
  \begin{split}
    \varphi : \Bigl(\bigoplus_{\gamma \in \Gamma} \mathbf{z}^{-\gamma} \mathbb{C}(\Lambda^{(n)})\Bigr) &\xlongrightarrow{\sim} \mathbb{C}(\Lambda^{(n)})[\Lambda] 
  \end{split}
\end{equation}
by taking the sum of components.

Then $\mathbf{T}_i$ induces an operator $\mathbb{T}_i$ on $\mathbb{C}(\Lambda^{(n)})[\Lambda]$ such that the following diagram commutes
\begin{equation}
  \label{eq:T-cd}
  \begin{tikzcd}
    \bigl(\bigoplus_{\gamma \in \Gamma} \mathbf{z}^{-\gamma} \mathbb{C}(\Lambda^{(n)})\bigr) \arrow[d, "\varphi" left, "\sim" {right, anchor=south, rotate=-90, inner sep=0.5mm}] \arrow[r, "\mathbf{T}_i"] & 
    \bigl(\bigoplus_{\gamma \in \Gamma} \mathbf{z}^{-\gamma} \mathbb{C}(\Lambda^{(n)})\bigr) \arrow[d, "\varphi" left, "\sim" {right, anchor=south, rotate=-90, inner sep=0.5mm}] \\
    \mathbb{C}(\Lambda^{(n)})[\Lambda] \arrow[r, "\mathbb{T}_i"] & \mathbb{C}(\Lambda^{(n)})[\Lambda]. 
  \end{tikzcd}
\end{equation}
Clearly, the $\mathbb{T}_i$ also satisfy the braid and the quadratic Hecke relations of $\widetilde{\mathcal{H}}_{v}^{(n)}$. 
Recall that Proposition~\ref{prop:T-descent} shows that $\mathbf{T}_i$ descends to an operator on $\bigoplus_{\gamma \in \Gamma} \mathbf{z}^{-\gamma} \mathbb{C}[\Lambda^{(n)}]$.
Furthermore, $\varphi$ restricts to an isomorphism
\begin{equation*}
    \Bigl(\bigoplus_{\gamma \in \Gamma} \mathbf{z}^{-\gamma} \mathbb{C}[\Lambda^{(n)}]\Bigr) \xlongrightarrow{\sim} \mathbb{C}[\Lambda] 
\end{equation*}
Thus, $\mathbb{T}_i$ descends to an operator on $\mathbb{C}[\Lambda]$. 

Let $\phi_{\Sigma, w}(\mathbf{z}; g) := \varphi(\boldsymbol{\phi}_w) = \sum_{\gamma \in \Gamma} \phi_{\gamma, w}(\mathbf{z}; g)$.
This particular metaplectic Iwahori Whittaker function has been studied extensively in the literature including for example \cite{McNamaraCS, PatnaikPuskasIwahori, SSV}.
By Theorem~\ref{thm:dem_recurse} and \eqref{eq:T-cd} it follows that, for a simple reflection $s_i$,
\begin{equation}
  \phi_{\Sigma, s_iw}(\mathbf{z}; g) =
  \begin{cases*}
    \mathbb{T}_i \phi_{\Sigma, w}(\mathbf{z}; g) & if $\ell(s_iw) > \ell(w)$ \\
    \mathbb{T}_i^{-1} \phi_{\Sigma, w}(\mathbf{z}; g) & if $\ell(s_iw) < \ell(w)$.
  \end{cases*}
\end{equation}
With the same setup as in Corollary~\ref{cor:evaluate}, with a path $w' \to s_{i_1}w' \to \cdots \to s_{i_k} \cdots s_{i_1} w' = w$ in the Weyl group, we obtain that
\begin{equation}
  \label{eq:scalar-T-evaluate}
  \phi_{\Sigma, w}(\mathbf{z}; \varpi^{-\lambda}w') = v^{\ell(w')} \mathbb{T}_{i_k}^{e_{i_k}} \cdots \mathbb{T}_{i_1}^{e_{i_1}} \mathbf{z}^\lambda \in \mathbb{C}[\Lambda].
\end{equation}

Chinta and Gunnells~\cite{ChintaGunnells} defined an action of the Weyl group
on (in our notation) $\mathbb{C}(\Lambda^{(n)})[\Lambda]$ that has appeared in
subsequent works such
as~\cite{ChintaGunnellsPuskas,PatnaikPuskasIwahori,PuskasWhittaker,ChintaOffen,McNamaraCS,SSV}.
Furthermore, Chinta, Gunnells, Pusk\'as and Patnaik in several
papers
\cite{ChintaGunnellsPuskas,PuskasWhittaker,PatnaikPuskasIwahori} study
metaplectic Iwahori Whittaker functions through operators that they
call \textit{Demazure-Lusztig operators}, though in this paper we use the term
Demazure-Whittaker operator. The Chinta-Gunnells Weyl group action, and
related Hecke actions are studied in Sahi, Stokman and
Venkateswaran~\cite{SSV}, where they are given a conceptual basis related
to the double affine Hecke algebra.

We will now relate these works to the Hecke algebra action $\mathbb{T}_i$.
Specifically, we will show in our next
Theorem~\ref{thm:T-average} that the operator $\mathbb{T}_i$ can be expressed
in terms of the Chinta-Gunnells action of \cite{ChintaGunnells} and that it
matches the metaplectic Demazure operators defined by
Chinta-Gunnells-Pusk\'as in \cite{ChintaGunnellsPuskas} and further studied
by Patnaik-Pusk\'as in \cite{PatnaikPuskasIwahori}.
Comparing~\eqref{eq:scalar-T-evaluate}
with~\cite[Corollary~5.4]{PatnaikPuskasIwahori}, this means that
$\phi_{\Sigma, w}(\mathbf{z}; \varpi^{-\lambda})$, with $w'=1$, equals the
particular metaplectic Iwahori Whittaker function studied in
Patnaik-Pusk\'as~\cite{PatnaikPuskasIwahori} up to normalization.  From the
isomorphism $\varphi$ it is possible to recover the vector
$\boldsymbol{\phi}_w$ from $\phi_{\Sigma,w}$.  Note however, that our
Corollary~\ref{cor:evaluate}, or indeed~\eqref{eq:scalar-T-evaluate}, allows
us to determine $\phi_{\Sigma, w}(\mathbf{z}; g)$ for all values of $g$, and
not only for $g = \varpi^{-\lambda}$.

For a simple reflection $s_i \in W$ the Chinta-Gunnells action on $\mathbb{C}(\Lambda^{(n)})[\Lambda]$ is linearly extended from the action $s_i \star \mathbf{z}^\lambda h$ where $h \in \mathbb{C}(\Lambda^{(n)})$ and $\lambda \in \Lambda$ which is defined as 
\begin{equation}
  s_i \star \z^\lambda h :=  \frac{(s_ih)\z^{s_i \lambda}}{1-v\z^{-n_i \alpha_i}} \Bigl( (1-v) \z^{\bigl(\floor*{\frac{B(\lambda, \alpha_i)}{Q(\alpha_i)}}_{n_i} \bigr)\alpha_i}  - g(-Q(\alpha_i) - B(\lambda, \alpha_i) ) \z^{(n_i -1)\alpha_i} (1-\z^{-n_{\alpha_i}\alpha_i}) \Bigr).
\end{equation}
To compare with equation (4.7) of~\cite{PatnaikPuskasIwahori}, note that our
Gauss sums are normalized differently: if $\mathfrak{g}_m$ is the Gauss sum
in~\cite{PatnaikPuskasIwahori}, then $v\mathfrak{g}_m = g(-m)$.

Recall from~\eqref{eq:Ti} and~\eqref{calphai} that
\begin{equation}
  D(\mathbf{z}^{-1}) \coloneqq D_{n_i}(\mathbf{z}^{-1}) = \frac{1-v}{\mathbf{z}^{n_i \alpha_i}-1} \qquad c_{\alpha_i}(\mathbf{z}^{-1}) = \frac{1-v \mathbf{z}^{-n_i\alpha_i}}{1 - \mathbf{z}^{-n_i\alpha_i}}.
\end{equation}

\begin{theorem}
  \label{thm:T-average}
  The operators $\mathbb{T}_i$ on $\mathbb{C}(\Lambda^{(n)})[\Lambda]$ defined from  $\mathbf{T}_i$ by~\eqref{eq:T-cd} are given by
  \begin{equation}\label{eq:CGPDemazureoperators}
    \mathbb{T}_i: f \mapsto D(\mathbf{z}^{-1})f - \mathbf{z}^{-n_i\alpha_i} c_{\alpha_i}(\mathbf{z}^{-1}) s_i \star f,
  \end{equation}
  and match the operators in~\cite{ChintaGunnellsPuskas,
  PatnaikPuskasIwahori}, there called Demazure-Lusztig operators. 
\end{theorem}

\begin{proof}
  By linearity it is enough to consider $f = \mathbf{z}^\lambda h$ with $\lambda \in \Lambda$ and $h \in \mathbb{C}(\Lambda^{(n)})$.
  Let $\gamma \in \Gamma$ such that $-\gamma \equiv \lambda \bmod \Lambda^{(n)}$ and let $\mathbf{f} := \varphi^{-1}(f) \in \bigl(\bigoplus_{\gamma' \in \Gamma} \mathbf{z}^{-\gamma'} \mathbb{C}(\Lambda^{(n)})\bigr)$.
  Then the only non-zero component of $\mathbf{f}$ is $\mathbf{f}^\gamma = \mathbf{z}^\lambda h$.
  Similar to the computation in~\eqref{eq:Ti-f-nonzero}, but here using~\eqref{tau1} for $\tau^1_{\gamma, \gamma}$ instead of~\eqref{tau1-alt} the only non-zero components of $(\mathbf{T}_i \mathbf{f})$ are
  \begin{equation}
    \begin{split}
      (\mathbf{T}_i \mathbf{f})^\gamma &= D(\mathbf{z}^{-1})\mathbf{f}^\gamma + s_i \tau^1_{\gamma,\gamma}(\mathbf{z}^{-1}) \mathbf{f}^\gamma = 
      \Bigl(\!D(\mathbf{z}^{-1}) + \frac{(1-v)}{1-\mathbf{z}^{n_i\alpha_i}}\mathbf{z}^{\bigl(\left\lfloor-\frac{B(\alpha_i,\gamma)}{Q(\alpha_i)}\right\rfloor_{\!n_i}\bigr)\alpha_i} s_i\Bigr) \mathbf{z}^\lambda h \\
      (\mathbf{T}_i \mathbf{f})^{s_i\gamma + \alpha_i} &= s_i \tau^2_{s_i\gamma + \alpha_i,\gamma}(\mathbf{z}^{-1}) \mathbf{f}^\gamma =
      g(B(\alpha_i, \gamma) - Q(\alpha_i)) \mathbf{z}^{-\alpha_i} s_i \mathbf{z}^\lambda h.
    \end{split}
  \end{equation}
  Thus, since $\lfloor{\scriptstyle-\frac{B(\alpha_i,\gamma)}{Q(\alpha_i)}}\rfloor_{\!n_i} = \lfloor{\scriptstyle\frac{B(\alpha_i,\lambda)}{Q(\alpha_i)}}\rfloor_{\!n_i}$ and $g(B(\alpha_i, \gamma) - Q(\alpha_i)) = g(-B(\alpha_i,\lambda) - Q(\alpha_i))$ we have that $\mathbb{T}_i f = \varphi(\mathbf{T}_i \mathbf{f})$ equals
  \begin{equation*}
      \Biggl(D(\mathbf{z}^{-1}) + \frac{1}{1-\mathbf{z}^{n_i \alpha_i}} \Bigl( (1-v) \mathbf{z}^{\bigl(\left\lfloor\frac{B(\alpha_i,\lambda)}{Q(\alpha_i)}\right\rfloor_{\!n_i}\bigr)\alpha_i} + g(-B(\alpha_i, \lambda) - Q(\alpha_i)) (1-\mathbf{z}^{n_i \alpha_i})\mathbf{z}^{-\alpha_i} \Bigr)s_i \Biggr)\mathbf{z}^\lambda h.
  \end{equation*}
  Since $\frac{1-v \mathbf{z}^{-n_i\alpha_i}}{1-\mathbf{z}^{n_i\alpha_i}}= -\mathbf{z}^{-n_i\alpha_i} c_{\alpha_i}(\mathbf{z}^{-1})$ and $(1-\mathbf{z}^{n_i \alpha_i})\mathbf{z}^{-\alpha_i} = -(1-\mathbf{z}^{-n_i\alpha_i})\mathbf{z}^{(n_i-1)\alpha_i}$ we conclude that
  \begin{equation}
    \mathbb{T}_i f = D(\mathbf{z}^{-1})f - \mathbf{z}^{-n_i\alpha_i} c_{\alpha_i}(\mathbf{z}^{-1}) s_i \star f.
  \end{equation}
  This expression agrees with~(4.10) of~\cite{PatnaikPuskasIwahori}.
\end{proof}

We have thus shown that the averaged Demazure-Whittaker operators $\mathbb{T}_i$ that arise from the component sum of the vector Demazure-Whittaker operators $\mathbf{T}_i$, as well as the particular metaplectic Iwahori Whittaker function $\phi_{\Sigma,w}$ agree with those studied in~\cite{McNamaraCS, PatnaikPuskasIwahori, SSV}.
Our equation~\eqref{eq:scalar-T-evaluate} is~\cite[Corollary~5.4]{PatnaikPuskasIwahori} and the fact that $\mathbb{T}_i$ descends to an operator on $\mathbb{C}[\Lambda]$ is~\cite[Claim~4.4]{PatnaikPuskasIwahori}.

The braid relations for $\mathbb{T}_i$ follow from those for $\mathbf{T}_i$ established in Theorem~\ref{thm:hecke} which is based on~\cite{BBBF}.
Independent proofs of this fact can be found in~\cite{ChintaGunnellsPuskas} and~\cite[Appendix~B]{PatnaikPuskasIwahori}, with a more conceptual proof in~\cite{SSV} based on a general construction of Weyl group and Hecke algebra representations.
 
\section{Main theorem and conclusions}
\label{sec:conclusions} 

The main theorem of this paper shows a relation between partition functions for the metaplectic Iwahori ice model described in Section~\ref{sec:lattice} and the metaplectic Iwahori Whittaker functions for $\mathbf{G} = \GL_r$ described for general $\mathbf{G}$ in Section~\ref{sec:whittaker}.
Before we can state the main theorem, Theorem~\ref{thm:main}, in Section~\ref{sec:main-thm-and-proof} we need to specify the particular metaplectic cover we consider as well as some notation for $\mathbf{G} = \GL_r$ in the upcoming section.

\subsection{Metaplectic Iwahori Whittaker functions for \texorpdfstring{$\mathbf{G} = \GL_r$}{G = GLr}}
\label{sec:GLr}

In this section we apply the theory from the previous sections to the case $\mathbf{G} = \GL_r$.
We identify $\Lambda = X_*(T) \iso \mathbb{Z}^r$ and use the standard elementary basis $e_i$, such that $\alpha_i = e_i - e_{i+1}$. 

We choose a cover $\Gn$ such that the bilinear form $B$, which determines the group multiplication on $\Tn$ by~\eqref{bcommrel}, becomes 
\[ B(e_i, e_j):= \delta_{ij}. \] 
which implies that $Q(\alpha_i) = 1$ and $n_i=n$. Such a cover can be obtained
by a particular cocycle as explained in~\cite[Section 5]{BBBF} and involves
the $2n$-power Hilbert symbol, which is one reason we assume that $F$ contains the $2n$-th roots of unity $\mu_{2n}$.

With this bilinear form $\Lambda^{(n)} \iso n\Lambda$ and thus $\Lambda / \Lambda^{(n)} \iso (\mathbb{Z}/n\mathbb{Z})^r$. 
We may choose the set $\Gamma$ of coset representatives to be the set of weights
$\mu \in \Lambda$ such that $\rho - \mu \in \{0,1, \ldots, n-1\}^{\times r}$
where $\rho = (r-1, r-2, \ldots, 0)$.  Letting the least nonnegative integer
residue $\floorn{\cdot}$ modulo $n$ act elementwise we may write this as
\begin{equation}
  \label{eq:mu-theta}
  \mu = \rho - \floorn{\theta}, \qquad \theta \in (\mathbb{Z}/n\mathbb{Z})^r. 
\end{equation}
Note that then $B(\alpha_i, \mu) = \floorn{\theta_{i+1}} - \floorn{\theta_i} + 1$ and that this is a different convention compared to the one used in (11) of~\cite{BBB}.

We will denote the Whittaker function $\phi_{\rho - \floorn{\theta}, w}(\z; g)$ simply by $\phi_{\theta, w}(\z; g)$.
A similar convention will be used for $\Omega_\theta^\mathbf{z}$ and $\tau_{\theta, \theta'}$.
In this setting, Proposition~\ref{prop:rep-base-case}, which computes $\phi_{\theta, w}(\mathbf{z}; \varpi^{-\lambda}w')$ when $w = w'$, becomes the following
\begin{corollary}[Proposition~\ref{prop:rep-base-case}]
  \label{cor:rep-base-case-GLr} 
  For $\mathbf{G} = \GL_r$, where we choose the representatives $\Gamma$ of $\Lambda / \Lambda^{(n)}$ as $\mu = \rho - \floorn{\theta}$ with $\theta \in (\mathbb{Z}/n\mathbb{Z})^r$ we get that
\begin{equation}
  \phi_{\theta,w}(\z; \varpi^{-\lambda} w) = 
  \begin{cases*}
    v^{\ell(w)} \mathbf{z}^{\lambda} & if $\theta \equiv \lambda + \rho \bmod{n}$ \\
    0 & otherwise.
  \end{cases*}
\end{equation}
\end{corollary}

We will now rewrite the expressions for $\tau$ more explicitly.
With $\mu = \rho - \floorn{\theta}$ we note for~\eqref{tau2} that $s_i(\rho - \floorn{\theta}) + \alpha_i = \rho - s_i\floorn{\theta}$.
Furthermore, we have that for $x \in \mathbb{Z}/n\mathbb{Z}$, $\floorn{-x} = \ceiln{1-x}-1$, where $\ceiln{x}$ is the least strictly positive integer modulo $n$. 
Hence, $\tau^1$ and $\tau^2$ coefficients can be reduced to the following.
\begin{corollary}[Proposition~\ref{propositionKP}]
  \label{cor:GL-tau}
  For $\mathbf{G} = \GL_r$ and $\theta \in (\mathbb{Z}/n\mathbb{Z})^r$ with components $\theta_i$ the $\tau^1$ and $\tau^2$ coefficients of~\eqref{tau1} and~\eqref{tau2} become
  \begin{equation}\label{eq:tau-theta}
  \begin{split}
    \tau^1_{\theta, \theta} &= \frac{1-v}{1-\z^{-n\alpha_i}} \z^{\alpha_i} \mathbf{z}^{-\ceiln{\theta_i - \theta_{i+1}}\alpha_i}  \\
    \tau^2_{\theta, s_i\theta} &= g(\theta_{i}-\theta_{i+1})\z^{\alpha_i}.
  \end{split}
  \end{equation}
\end{corollary}

\begin{corollary}[Proposition~\ref{prop:rep-recursion}]
  \label{cor:GL-rep-recursion}
  For $\mathbf{G} = \GL_r$ and $\theta \in (\mathbb{Z}/n\mathbb{Z})^r$ with components $\theta_i$ and $w \in W$ the following functional equation holds:
  \begin{multline}
    \label{eq:GL-rep-recursion}
    \mathbf{z}^{\alpha_i} \Bigl((1-v)\mathbf{z}^{-\ceiln{\theta_i - \theta_{i+1}} \alpha_i} \phi_{\theta,w}(\mathbf{z};g) + g(\theta_i - \theta_{i+1}) (1-\mathbf{z}^{-n\alpha_i}) \phi_{s_i\theta,w}(\mathbf{z};g) \Bigr) = \\
    =\begin{cases*}
      (1-v) \phi_{\theta,w}(s_i\mathbf{z};g) + (1-\mathbf{z}^{-n\alpha_i}) \phi_{\theta, s_iw}(s_i\mathbf{z}) & if $\ell(s_iw) > \ell(w)$, \\
      (1-v) \mathbf{z}^{-n\alpha_i} \phi_{\theta, w}(s_i \mathbf{z}) + v(1-\mathbf{z}^{-n\alpha_i}) \phi_{\theta,s_iw}(s_i \mathbf{z}) & if $\ell(s_iw) < \ell(w)$.
    \end{cases*}
  \end{multline}
\end{corollary}

Lastly, let $\mathbf{f} \in (\mathbb{C}(n\Lambda)[\Lambda])^{n^r}$ with components enumerated by $\theta$.
The operator $\mathbf{T}_i$ defined in~\eqref{eq:Ti} reduces to
\begin{equation}
  \label{eq:GLr-Ti}
  (\mathbf{T}_i \mathbf{f})^\theta = \frac{1}{1-\mathbf{z}^{n\alpha_i}} \Bigl( (1-v) \bigl(\mathbf{z}^{(\ceiln{\theta_i - \theta_{i+1}} - 1)\alpha_i} s_i - 1\bigr) \mathbf{f}^\theta + g(\theta_i - \theta_{i+1}) \mathbf{z}^{-\alpha_i}(1-\mathbf{z}^{n\alpha_i}) s_i\mathbf{f}^{s_i\theta}\Bigr).
\end{equation}

\subsection{The main theorem}
\label{sec:main-thm-and-proof}
We will now combine the results of Sections~\ref{sec:lattice} and~\ref{sec:GLr} for the main theorem of this paper.

\begin{maintheorem} 
  \label{thm:main}
  Let $w \in W$, $\theta \in (\mathbb{Z}/n\mathbb{Z})^r$ and let $\mu\in\mathbb{Z}_{\geqslant0}^r$. Let $w' \in W$ such that the order of the colors for the top boundary edges for the system $\mathfrak{S}_{\mu, \theta, w}$ in the monochrome picture is $w'P$ from left to right where $P = (c_r, c_{r-1}, \ldots, c_1)$. 
  Then, the condition on $\mu$ and $w'$ is equivalent to $\lambda := w' \mu - \rho$ being $w'$-almost dominant and, with the conventions of Section~\ref{sec:GLr},
  \begin{equation}
    Z\bigl(\mathfrak{S}_{\mu, \theta, w}\bigr)(\mathbf{z}) = \mathbf{z}^\rho \phi_{\theta,w}(\mathbf{z},\varpi^{-\lambda}w') .
  \end{equation}
\end{maintheorem}

For the definition of $w'$-almost dominant see Definition~\ref{def:almost_dominant}.
Note that the partition functions determine all values for the metaplectic Whittaker function at any $g \in \Gn$ and that the set of boundary conditions maps bijectively to the set of data that determine all values of the metaplectic Iwahori Whittaker functions as pictured in Figure~\ref{fig:dictionary}.
This combinatorial data consists of $\theta \in (\mathbb{Z}/n\mathbb{Z})^r$ that determines the Whittaker functional $\Omega_\theta$, $w \in W$ which determines the Iwahori fixed vector $\Phi_w$, and the pairs $(w', \lambda)$ where $\lambda \in \Lambda$ and $w' \in W$ such that $\lambda$ is $w'$-almost dominant which determines the argument $g = \varpi^{-\lambda}w'$.
Indeed, recall that $\phi_{\theta, w}(\mathbf{z};g)$ is a genuine function on $\Gn$, invariant under the Iwahori subgroup $J$ on the right, $(\mathbf{N}^-(F),\psi)$-invariant on the left and that the center $Z(\Gn)$ acts by scalars. Thus $\phi_{\theta, w}(\mathbf{z};g)$ is determined by its values on a set of double coset representatives for $\mathbf{N}^-(F)Z(\Gn)\backslash \Gn / J \mu_n$ which we may choose as $g = \varpi^{-\lambda}w'$ with $\lambda \in \Lambda$ and $w' \in W$.
That $\lambda$ is $w'$-almost dominant follows from Lemma~\ref{lem:almost-dominant}.
Using the center, we may assume that $\lambda + \rho$, and therefore also $\mu$ has only non-negative parts.

\begin{figure}[htbp]
    \centering
    \begin{equation*}
  \mathbf{z}^\rho\phi_{\theta, w}(\mathbf{z}; g) = \mathbf{z}^\rho \delta^{1/2}(g) \Omega^{\mathbf{z}^{-1}}_{\tikzmarknode{theta1}{\theta}}\bigl( \pi(\tikzmarknode{g1}{g}) \Phi_{\tikzmarknode{w1}{w}}^{\mathbf{z}^{-1}}\bigr) = Z \left(\hspace{.5em}
\begin{tikzpicture}[remember picture, baseline=0.5cm, scale=0.5, every node/.append style={scale=0.8}]
  
  \draw[selection] (-.5,-.5) rectangle (.5,3.2) coordinate[pos=0] (theta2) {};
  \draw[selection] (6,-.5) rectangle (7,3.2) coordinate[pos=0] (w2) {};
  \draw[selection] (.75,4.5) rectangle (5.75,2.75) coordinate[pos=0] (g2) {};

  \draw[very thick] (0,0) node[very thick, densely dotted, state] {} -- (0.75,0) node[label=right:$z_3$] {};
  \draw[very thick] (0,1) node[very thick, densely dotted, state] {} -- (0.75,1) node[label=right:$z_2$] {};
  \draw[very thick] (0,2) node[very thick, densely dotted, state, label={above:$\theta$}] {} -- (0.75,2) node[label=right:$z_1$] {};
  \foreach \x in {0,...,4}{
    \draw[very thick] (\x+1.25, -1.25) node[state] {$+$} -- (\x+1.25,-.5);
    \draw[very thick] (\x+1.25, 3.25) node[state] {} -- (\x+1.25, 2.5);
  }
  \node at (3.25, 4) {$(w', \lambda)$};
  \draw[very thick] (6.5, 0) node[very thick, green, state] {} -- (5.75, 0);
  \draw[very thick] (6.5, 1) node[very thick, blue, state] {} -- (5.75, 1);
  \draw[very thick] (6.5, 2) node[very thick, red, state, label={above:$w$}] {} -- (5.75, 2);
  \draw[very thick] (0.75,-.5) rectangle (5.75, 2.5);
\end{tikzpicture}
\hspace{.5em}\right)
= Z\bigl(\mathfrak{S}_{w'^{-1}(\lambda + \rho), \theta, w}\bigr)
\begin{tikzpicture}[overlay, remember picture]
  \draw[selection-line] ($(g2) + (1.25,0.1)$) -| ++(0,0.2) -| (g1);
  \draw[selection-line] ($(w1.south) - (0,0.2)$) -| ++(0,-1.5) -| ($(w2) + (0.25,0)$);
  \draw[selection-line] ($(theta1.south) - (0,0.2)$) -| ++(0,-1.8) -| ($(theta2) + (0.25,0)$);
\end{tikzpicture}
\vspace{1em}
\end{equation*}
    \caption{Dictionary between boundary data for lattice model and data determining the values of the metaplectic Iwahori Whittaker function illustrating the equality of Theorem~\ref{thm:main}.}
    \label{fig:dictionary}
\end{figure}

\begin{proof}
That the condition for $\mu$ and $w'$ is equivalent to $\lambda := w'\mu - \rho$ being $w'$-almost dominant follows from Remark~\ref{rem:almost-dominant}.
We prove the remaining statement by induction on the length of $ww'^{-1}$.
The base case $w=w'$ follows immediately from Lemma~\ref{lem:lattice-ground-state} and Corollary~\ref{cor:rep-base-case-GLr}.
Using Theorem~\ref{thm:dem_recurse} we may determine $\phi_{\theta,s_iw}$ from $\phi_{\theta',w}$, and thus it is enough to verify that $\overline{Z}_{\theta, w}(\mathbf{z}) := \mathbf{z}^{-\rho}Z\bigl(\mathfrak{S}_{\mu, \theta, w}\bigr)(\mathbf{z})$ satisfies the same recursion relations.
In fact, it is enough to compare with the recursion relations for $\phi$ in Corollary~\ref{cor:GL-rep-recursion} which determine the operator $\mathbf{T}_i$ in~\eqref{eq:GLr-Ti} for $\GL_r$.

Indeed, from the recursion relations for $Z\bigl(\mathfrak{S}_{\mu, \theta, w}\bigr)(\mathbf{z})$ in Proposition~\ref{prop:lattice-recursion} obtained from the Yang-Baxter equation, we get that 
\begin{multline}
  (1-v) \mathbf{z}^{-\ceiln{\theta_i-\theta_{i+1}}\alpha_i} \overline{Z}_{\theta, w}(\mathbf{z}) + g(\theta_i-\theta_{i+1}) (1 - \mathbf{z}^{-n\alpha_i}) \overline{Z}_{s_i\theta, w}(\mathbf{z}) 
  = \\
  = \mathbf{z}^{-\alpha_i} \begin{cases*}
    (1-v) \overline{Z}_{\theta, w}(s_i\mathbf{z}) + (1 - \mathbf{z}^{-n\alpha_i}) \overline{Z}_{\theta, s_iw}(s_i\mathbf{z}) & if $\ell(s_iw) > \ell(w)$ \\
    (1-v)\mathbf{z}^{-n\alpha_i} \overline{Z}_{\theta, w}(s_i\mathbf{z}) + v (1 - \mathbf{z}^{-n\alpha_i}) \overline{Z}_{\theta, s_i w}(s_i\mathbf{z}) & if $\ell(s_iw) < \ell(w)$ \\
  \end{cases*}
\end{multline}
which matches those for $\phi$ in Corollary~\ref{cor:GL-rep-recursion}.
The factor $\mathbf{z}^{-\alpha_i}$ comes from the fact that $Z\bigl(\mathfrak{S}_{\mu, \theta, w}\bigr)(s_i\mathbf{z}) = s_i\bigl( \mathbf{z}^\rho \overline{Z}_{\theta, w}(\mathbf{z}) \bigr) = \mathbf{z}^{-\alpha_i} \mathbf{z}^\rho \overline{Z}_{\theta, w}(s_i\mathbf{z})$ since $s_i \mathbf{z}^\rho = \mathbf{z}^{\rho - \alpha_i}$.
\end{proof}

\subsection{R-matrices and \texorpdfstring{$p$-adic}{p-adic} intertwiners}
\label{sec:intertwiners-R-matrices}
In this section we will present a dictionary relating the following topics:
\[\mbox{$p$-adic intertwiners} \longleftrightarrow \mbox{quantum group R-matrices} \longleftrightarrow \mbox{lattice models R-matrices}.\]
In Proposition~\ref{prop:Drinfeldtwisting} we connected the R-matrix for the lattice model introduced in this paper to the $U_q(\widehat{\mathfrak{gl}}(r|n))$ R-matrix. 
We shall now refine this connection and show how R-matrices coming from
quantum affine groups relate to non-archimedean local field intertwining integrals.  

Let $q=1/ \sqrt{v}$. For fixed $z \in \mathbb{C}^\times$, let $\mathbb{V}_{r,n}(z)$ be an $r+n$ dimensional complex vector space with basis given by elements $v_i^+(z)$ and $v_j^-(z)$, where $1 \leqslant i \leqslant r$ indexes the set of colors and $1 \leqslant j \leqslant n$ indexes the set of scolors.
The weights of the R-matrix in Figure~\ref{fig:proof-R-matrix} can be used to create a map 
\[\mathbb{R}_q (z_1,z_2): \mathbb{V}_{r,n}(z_1) \otimes \mathbb{V}_{r,n}(z_2) \to \mathbb{V}_{r,n}(z_1) \otimes \mathbb{V}_{r,n}(z_2)\] 
which satisfies the ungraded Yang-Baxter equation:
\[ \mathbb{R}_q (z_1,z_2)_{1,2} \mathbb{R}_q (z_1,z_3)_{1,3} \mathbb{R}_q (z_2,z_3)_{2,3} =  \mathbb{R}_q (z_2,z_3)_{2,3} \mathbb{R}_q (z_1,z_3)_{1,3} \mathbb{R}_q (z_1,z_2)_{1,2}. \]
The above should be thought of as an equality in $\operatorname{End}( \mathbb{V}_{r,n}(z_1) \otimes \mathbb{V}_{r,n}(z_2) \otimes  \mathbb{V}_{r,n}(z_3))$, and the subscripts $_{i,j}$ mean that we apply the map to the $(i,j)$ entry of the tensor product.  

In equations (2.3)-(2.8) of~\cite{Kojima}, Kojima computes the $U_q(\widehat{\mathfrak{gl}}(r|n))$ R-matrix $: V_{r,n}(z_1) \otimes V_{r,n}(z_2) \to V_{r,n}(z_1) \otimes V_{r,n}(z_2)$, where $V_{r,n}(z)$ is the evaluation module of $U_q(\widehat{\mathfrak{gl}}(r|n))$ of dimension $(r|n)$, i.e. its even part has dimension $r$ and its odd part has dimension $n$. 
This R-matrix satisfies the graded Yang-Baxter equation (see equation (2.11) in~\cite{Kojima}). 
We shall denote by 
\[R^{(r|n)}_q(z_1z^{-1}_2):V_{r,n}(z_1) \otimes V_{r,n}(z_2) \to V_{r,n}(z_1) \otimes V_{r,n}(z_2)\]
the (ungraded) Drinfeld twist of Kojima's R-matrix described in the proof of Proposition~\ref{prop:Drinfeldtwisting}. 
It follows by the same proposition that, after identifying $\mathbb{V}_{r,n}(z) \iso V_{r,n}(z)$, the two R-matrices are the same. 
A Drinfeld twist is a process which introduces extra complex factors in the R-matrix.
Here we use a particular Drinfeld twist to introduce Gauss sums in $R^{(r|n)}_q (z_1z_2^{-1})$ which otherwise do not appear naturally in the quantization of a universal enveloping algebra.

Define $\mathbb{V}_{r}(z)$ and $\mathbb{V}_{n}(z)$ to be the subspaces of $\mathbb{V}_{r,n}(z)$ with basis $v_i^+(z)$ and $v_j^-(z)$, respectively. 
We define $V_{r}(z)$ and $V_{n}(z)$ in a similar way, and let $V_r(\z) := V_r(z_1)  \otimes \cdots \otimes V_r(z_r)$, and $V_n(\z) := V_n(z_1)  \otimes \cdots \otimes V_n(z_r)$. 
Denote the restriction of $R^{(r|n)}_q(z_1z^{-1}_2)$ to $V_{r}(z_1) \otimes V_{r}(z_2)$ by $R^r_q(z_1z_2^{-1})$. 
It is a well known fact in quantum group theory that $R^r_q(z_1z_2^{-1})$ is then the $U_q(\widehat{\mathfrak{gl}}(r))$ R-matrix (up to a Drinfeld twist), where one can now think of $V_{r}(z)$ as the evaluation module for $U_q(\widehat{\mathfrak{gl}}(r))$.
Moreover, restricting $R^{(r|n)}_q(z_1z^{-1}_2)$ to $V_{n}(z_1) \otimes V_{n}(z_2)$, produces the $U_{q^{-1}}(\widehat{\mathfrak{gl}}(n))$ R-matrix (again up to a Drinfeld twist), which we will denote by $R^n_{q^{-1}}(z_1z_2^{-1})$. 

Recall the space of Whittaker functionals $\mathcal{W}_\z$ defined in~\eqref{eq:Wz} and the space of Iwahori fixed vectors in the unramified principal series $I(\z)^J$ from~\eqref{def:Iwahoribasis}. 
Intertwiners induce maps on these spaces.
It was shown in~\cite[Theorem 1.1]{BBB} and in~\cite[Theorem 10.5, Remark 10.6]{BBBGIwahori} that there is a vector space embedding $\xi_\z$ and a vector space isomorphism $\eta_\z$  
\begin{equation}
  \begin{matrix}
    \xi_\z: I(\z)^J & \xhookrightarrow{\quad} & V_r(\z^n), \\
                  &                    & \rotatebox[origin=c]{-90}{${}\iso{}$} \\
                  &                    & \mathbb{V}_r(\z^n)
  \end{matrix}
  \hspace{6em}
  \begin{matrix}
    \eta_\z: \W_{\z} & \xlongrightarrow{\sim} & V_n(\z^n) \\
                  &                     & \rotatebox[origin=c]{-90}{${}\iso{}$} \\
                  &                     & \mathbb{V}_n(\z^n)
  \end{matrix}
\end{equation}
such that the following diagrams commute:
\begin{equation}\label{eq:commutativediagrams}
\begin{tikzcd}
  \text{\footnotesize{non-archimedean}} & \text{\footnotesize{quantum group}} & \text{\footnotesize{lattice model}} \\[-1em]
I(\z)^J \arrow[r, "\xi_\z"] \arrow[d, "\mathcal{A}^{\z}_{s_i}"]
& V_r(\z^n) \arrow[r, "\sim"]  \arrow[d, "\tau R^r_{q}(\z^{n\alpha_i})_{i,i+1}"] 
& \mathbb{V}_r(\z^n) \arrow[d, "\tau \mathbb{R}^r_{q}(\z^{n\alpha_i})_{i,i+1}"] 
\\
I(s_i \z)^J \arrow[r, "\xi_{s_i \z}"]
& V_r(s_i \z^n) \arrow[r, "\sim"] 
& \mathbb{V}_r(s_i \z^n)
\\
\mathcal{W}_\z \arrow[r, "\eta_\z"] \arrow[d, "\mathcal{A}^{\z}_{s_i}"]
& V_n(\z^n) \arrow[r, "\sim"]  \arrow[d, "\tau R^n_{q^{-1}}(\z^{n\alpha_i})_{i,i+1}"] 
& \mathbb{V}_n(\z^n) \arrow[d, "\tau \mathbb{R}^n_{q^{-1}}(\z^{n\alpha_i})_{i,i+1}"] 
\\
\mathcal{W}_{s_i \z}  \arrow[r, "\eta_{s_i \z}"]
& V_n(s_i \z^n) \arrow[r, "\sim"]
& \mathbb{V}_n(s_i \z^n)
\end{tikzcd} 
\end{equation} 
Here $\tau$ is the flip map, that is $\tau : v \otimes w \mapsto w \otimes v$, not to be confused with the scattering matrix of Section~\ref{sec:recursionrelations}.
 
Recall from the pictorial description of our main result in Figure~\ref{fig:dictionary} the natural correspondence between the boundary data for the lattice model and the data determining the values of the Whittaker function.
We will now explain, by using this pictorial description, the natural correspondence between the two methods of deducing functional equations for the partition functions on the one side (as in Section~\ref{sec:lattice-recursion}) and for the Whittaker functions on the other (as in Section~\ref{sec:recursionrelations}).
On the lattice model side, we use the Yang-Baxter equation which allows us to \emph{move} an $R$-vertex (essentially the R-matrix) from right to left.
On the non-archimedean representation theory side one uses the fact that the intertwiner is a $\tilde{G}$-homomorphism, which allows us to \emph{move} it from interacting with the Iwahori fixed vector to interacting with the Whittaker functional. 
The two processes are pictured and compared in Figure~\ref{fig:intertwiners-R-matrices}.

\begin{figure}[htpb]
  \centering

\begin{equation*}
  \begin{gathered}
    \tikzmarknode[selection, draw, inner sep=2pt]{A1}{\Omega_\theta^{s_i\mathbf{z}^{-1}} \circ \mathcal{A}_{s_i}^{\mathbf{z}^{-1}}}\hspace{1pt} \bigl( \pi(g) \Phi_w^{\mathbf{z}^{-1}} \bigr) = \Omega_\theta^{s_i\mathbf{z}^{-1}}\bigl( \pi(g) \tikzmarknode[selection, draw, inner sep=2pt]{A2}{\mathcal{A}_{s_i}^{\mathbf{z}^{-1}}\Phi_w^{\mathbf{z}^{-1}}}\hspace{1pt} \bigr) \\[1em]
\begin{tikzpicture}[remember picture, baseline=0.5cm, scale=0.5, every node/.append style={scale=0.8}]
  \begin{scope}

    \draw[selection] (.5,2.5) rectangle (-1.5,.5) coordinate[pos=0] (R1) {}; 
    
    \draw[very thick] (0,0) node[very thick, densely dotted, state] {} -- (0.75,0);
    \draw[very thick] (0,1) node[very thick, densely dotted, state] (A) {} -- (0.75,1);
    \draw[very thick] (0,2) node[very thick, densely dotted, state] (B) {} -- (0.75,2);
    \foreach \x in {0,...,3}{
      \draw[very thick] (\x+1.25, -1.25) node[state] {$+$} -- (\x+1.25,-.5);
      \draw[very thick] (\x+1.25, 3.25) node[state] {} -- (\x+1.25, 2.5);
    }
    \draw[very thick] (5.5, 0) node[very thick, green, state] {} -- (4.75, 0);
    \draw[very thick] (5.5, 1) node[very thick, blue, state] {} -- (4.75, 1);
    \draw[very thick] (5.5, 2) node[very thick, red, state] {} -- (4.75, 2);
    \draw[very thick] (0.75,-.5) rectangle (4.75, 2.5);

    \draw[very thick] (-1,2) node[very thick, densely dotted, state] {} -- (A);
    \draw[very thick] (-1,1) node[very thick, densely dotted, state] {} -- (B);
    \node[scale=1/0.8] at (7.25+0.5,1) {$=\ldots=$};
  \end{scope}

  \begin{scope}[xshift=10cm]

    \draw[selection] (3,2.5) rectangle (5,.5) coordinate[pos=0] (R2) {};

    \draw[selection-line] (4,2.5) -- (4,4.5) node[halo-rect, inner sep=1pt, scale=1/0.8] {$R_q^{(r|n)}$};
    
    \draw[very thick] (0,0) node[very thick, densely dotted, state] {} -- (0.75,0);
    \draw[very thick] (0,1) node[very thick, densely dotted, state] {} -- (0.75,1);
    \draw[very thick] (0,2) node[very thick, densely dotted, state] {} -- (0.75,2);
    \foreach \x in {0,...,1}{
      \draw[very thick] (\x+1.25, -1.25) node[state] {$+$} -- (\x+1.25,-.5);
      \draw[very thick] (\x+1.25, 3.25) node[state] {} -- (\x+1.25, 2.5);
    }
    \draw[very thick] (5.25, 0) -- (2.75, 0);
    \draw[very thick] (3.5, 1) node[very thick, red, state] (G) {} -- (2.75, 1);
    \draw[very thick] (3.5, 2) node[very thick, densely dotted, state] (H) {} -- (2.75, 2);
    \draw[very thick] (0.75,-.5) rectangle (2.75, 2.5);

    \draw[very thick] (5.25,2) -- (4.5,2) node[very thick, red, state] {} -- (G);
    \draw[very thick] (5.25,1) -- (4.5,1) node[very thick, densely dotted, state] {} -- (H);
    
    \draw[very thick] (5.25,-.5) rectangle (7.25, 2.5);
    \foreach \x in {0,...,1}{
      \draw[very thick] (\x+1.25+4.5, -1.25) node[state] {$+$} -- (\x+1.25+4.5,-.5);
      \draw[very thick] (\x+1.25+4.5, 3.25) node[state] {} -- (\x+1.25+4.5, 2.5);
    }

    \draw[very thick] (8, 0) node[very thick, green, state] {} -- (7.25, 0);
    \draw[very thick] (8, 1) node[very thick, blue, state] {} -- (7.25, 1);
    \draw[very thick] (8, 2) node[very thick, red, state] {} -- (7.25, 2);
    \node[scale=1/0.8] at (9.5+0.5,1) {$=\ldots=$};
  \end{scope}
  \begin{scope}[xshift=22cm]

    \draw[selection] (5,2.5) rectangle (7,.5) coordinate[pos=0] (R2) {};
    
    \draw[very thick] (0,0) node[very thick, densely dotted, state] {} -- (0.75,0);
    \draw[very thick] (0,1) node[very thick, densely dotted, state] {} -- (0.75,1);
    \draw[very thick] (0,2) node[very thick, densely dotted, state] {} -- (0.75,2);
    \foreach \x in {0,...,3}{
      \draw[very thick] (\x+1.25, -1.25) node[state] {$+$} -- (\x+1.25,-.5);
      \draw[very thick] (\x+1.25, 3.25) node[state] {} -- (\x+1.25, 2.5);
    }
    \draw[very thick] (5.5, 0) node[very thick, green, state] {} -- (4.75, 0);
    \draw[very thick] (5.5, 1) node[very thick, red, state] (G) {} -- (4.75, 1);
    \draw[very thick] (5.5, 2) node[very thick, blue, state] (H) {} -- (4.75, 2);
    \draw[very thick] (0.75,-.5) rectangle (4.75, 2.5);

    \draw[very thick] (6.5,2) node[very thick, red, state] {} -- (G);
    \draw[very thick] (6.5,1) node[very thick, blue, state] {} -- (H);
  \end{scope}
\end{tikzpicture}
\end{gathered}
\begin{tikzpicture}[overlay, remember picture]
  \draw[selection-line] ($(R1) + (-0.5,0)$) .. controls ($(R1) + (0,1.5)$) and ($(A1) + (-1,-1)$) .. (A1) node[midway, halo-rect, inner sep=1pt] {$R_{q^{-1}}^{n}$};
  \draw[selection-line] ($(R2) + (0.5,0)$) .. controls ($(R2) + (0,1.5)$) and ($(A2) + (1,-1)$) .. (A2) node[midway, halo-rect, inner sep=1pt] {$R_{q}^{r}$};
\end{tikzpicture}
\end{equation*}
  \caption{Comparison between the two process of obtaining functional equations for the partition functions of the lattice model and the metaplectic Iwahori Whittaker functions.
  Different restrictions of the R-matrices are identified with different actions of the intertwiners where $R^r_{q}$ is an R-matrix for $U_{q}(\widehat{\mathfrak{gl}}(r))$, $R^n_{q^{-1}}$ an R-matrix for $U_{q^{-1}}(\widehat{\mathfrak{gl}}(n))$, and $R^{(r|n)}_q$ an R-matrix for $U_q(\widehat{\mathfrak{gl}}(r|n))$.}%
  \label{fig:intertwiners-R-matrices}
\end{figure}

Note that in the non-archimedean representation theory side we \emph{only} see
the $U_q(\widehat{\mathfrak{gl}}(r))$ and $U_{q^{-1}}(\widehat{\mathfrak{gl}}(n))$ R-matrices, which
appear as scattering matrices of intertwining operators on
Iwahori fixed vectors and Whittaker functionals, respectively.
On the lattice model side, one needs the full $U_q(\widehat{\mathfrak{gl}}(r|n))$ super R-matrix to solve the
Yang-Baxter equations for the model. Yet this super R-matrix does not
have any obvious role in the p-adic theory; only the $\mathfrak{gl}(r)$ and $\mathfrak{gl}(n)$ pieces do.
It would be interesting to understand if the full super quantum group also has
a direct role in the non-archimedean representation theory.

The $U_{q^{-1}}(\widehat{\mathfrak{gl}}(n))$ R-matrix also makes an appearance if we consider an alternative way of proving the braid relations for our vector metaplectic Demazure-Whittaker operators $\mathbf{T}_i$.
Recall that in the proof of Theorem~\ref{thm:hecke} we use that $\boldsymbol{\tau}$ of \eqref{eq:tau-operartor} satisfies similar braid relations which is Assumption 2 of~\cite{BBBF}.
As illustrated above, this Kazhdan-Patterson scattering matrix appears as the $U_{q^{-1}}(\widehat{\mathfrak{gl}}(n))$ R-matrix on the lattice model side for $\mathbf{G} = \GL_r$.
The braid relations are then translated to the Yang-Baxter equation illustrated in Figure~\ref{fig:braiding} involving three $R$-vertices (which we shall call the RRR Yang-Baxter equation).

\begin{figure}[hbtp]
  \centering
  \begin{equation*}
\begin{tikzpicture}[baseline={([yshift=-2pt]current bounding box.center)}, scaled, scale=0.6]
  \draw[thick] (-1,-1) node[scolored, spin, label=left:$z_i$] {} to[out=0, in=240] (0,0) to[out=60, in=180] (1,1) to[out=0, in=240] (2,2) to[out=60, in=180] (3,3) -- (5,3) node[scolored, spin, label=right:$z_i$] {};
  \draw[thick] (-1,1) node[scolored, spin, label=left:$z_j$] {} to[out=0, in=120] (0,0) node[halo, inner sep=1pt]{$R$} to[out=-60, in=180] (1,-1) -- (3,-1) to[out=0, in=240] (4,0) to[out=60, in=180] (5,1) node[scolored, spin, label=right:$z_j$] {};
  \draw[thick] (-1,3) node[scolored, spin, label=left:$z_k$] {} -- (1,3) to[out=0, in=120] (2,2) node[halo, inner sep=1pt]{$R$} to[out=-60, in=180] (3,1) to[out=0, in=120] (4,0) node[halo, inner sep=1pt]{$R$} to[out=-60, in=180] (5,-1) node[scolored, spin, label=right:$z_k$] {};
\end{tikzpicture}
\quad = \quad
\begin{tikzpicture}[baseline={([yshift=-2pt]current bounding box.center)}, scaled, scale=0.6, yscale=-1]
  \draw[thick] (-1,-1) node[scolored, spin, label=left:$z_k$] {} to[out=0, in=240] (0,0) to[out=60, in=180] (1,1) to[out=0, in=240] (2,2) to[out=60, in=180] (3,3) -- (5,3) node[scolored, spin, label=right:$z_k$] {};
  \draw[thick] (-1,1) node[scolored, spin, label=left:$z_j$] {} to[out=0, in=120] (0,0) node[halo, inner sep=1pt]{$R$} to[out=-60, in=180] (1,-1) -- (3,-1) to[out=0, in=240] (4,0) to[out=60, in=180] (5,1) node[scolored, spin, label=right:$z_j$] {};
  \draw[thick] (-1,3) node[scolored, spin, label=left:$z_i$] {} -- (1,3) to[out=0, in=120] (2,2) node[halo, inner sep=1pt]{$R$} to[out=-60, in=180] (3,1) to[out=0, in=120] (4,0) node[halo, inner sep=1pt]{$R$} to[out=-60, in=180] (5,-1) node[scolored, spin, label=right:$z_i$] {};
\end{tikzpicture}
\end{equation*}
\caption{The RRR Yang-Baxter equation with three $R$ vertices that can be used to prove the braid relations for the vector metaplectic Demazure-Whittaker operators $\mathbf{T}_i$ when $\mathbf{G} = \GL_r$.}%
  \label{fig:braiding}
\end{figure}

Furthermore, while the proof of Theorem~\ref{thm:hecke} used the fact that the function $g(a)$ appearing in $\boldsymbol{\tau}$ is actually a Gauss sum coming from the metaplectic representation theory, the RRR Yang-Baxter equation holds for any function $g(a)$ satisfying Assumption~\ref{assumptionga}, as the R-matrix comes from the Drinfeld twist of $U_{q^{-1}}(\widehat{\mathfrak{gl}}(n))$.
Thus, at least for $\GL_r$, the Demazure-Whittaker operators $\mathbf{T}_i$ may be defined for these more general parameters $g(a)$ as well, in accord with~\cite{SSV} where they work with the \emph{scalar} Demazure-Whittaker operators in similar generality.

\begin{remark}
In Figure~\ref{fig:corners} we explain how the results of this paper generalize prior work on this topic, highlighting their different underlying quantum groups.
The horizontal arrows in the figure reflect the fact that the models
in the left columns are special cases (with $n=1$) of the models
in the right column. If the boundary spins of the colored system
are restricted to a single color, we recover the Tokuyama or
metaplectic ice models. To look at it another way, the vertical arrows
reflect the fact that the spherical Whittaker function is the sum of Iwahori
Whittaker functions. 
\end{remark}

\begin{figure}[ht]
\newcommand{\textcell}[1]{\text{\makecell{#1}}}
\begin{tikzcd}[row sep = 0pt, nodes in empty cells, execute at end picture = {
\draw[shorten >= -1em, shorten <= -1em] ($ (\tikzcdmatrixname-1-1.north)!0.5!(\tikzcdmatrixname-1-2.north) $) -- ($ (\tikzcdmatrixname-4-1.south)!0.5!(\tikzcdmatrixname-4-2.south) $);
}]
  & \textcell{$\GL_r$ \\ no scolors} & \textcell{metaplectic $n$-cover of $\GL_r$ \\ $n$ scolors (or charges)} \\[0.25em] \hline & & \\ 
  \textcell{spherical vector\\ no colors} & \textcell{Tokuyama ice \cite{hkice} \\ $U_{q}(\widehat{\mathfrak{gl}}(1|1))$} &
  \arrow[l, "\substack{\text{identify}\\\text{scolors}}", start anchor={[xshift=-1.1ex]}] \textcell{Metaplectic ice \cite{mice,BBB} \\ $U_{q}(\widehat{\mathfrak{gl}}(1|n))$} \\[2.5em]
  \textcell{Iwahori fixed vector\\ $r$ colors} & \arrow[u, "\text{identify colors}" left]\textcell{Iwahori ice \cite{BBBGIwahori} \\ $U_{q}(\widehat{\mathfrak{gl}}(r|1))$} &
  \arrow[u, "\text{identify colors}" right] \arrow[l, "\substack{\text{identify}\\\text{scolors}}", end anchor={[xshift=1.275ex]}] \textcell{Metaplectic Iwahori ice \\ $U_{q}(\widehat{\mathfrak{gl}}(r|n))$} \\
  & & \\
\end{tikzcd}
\caption{\label{fig:corners} Summary of the different lattice models for
  non-archimedean Whittaker functions. }
\end{figure}

\begin{remark}
  Sahi, Stokman and Venkateswaran~\cite{SSV} construct the metaplectic Demazure-Whittaker operators $\mathbb{T}_i$ in~\eqref{eq:CGPDemazureoperators} (and therefore the Chinta-Gunnells action) without making use of the representation theory of the $p$-adic group and for general parameters $g(a)$.  
They consider quotients of a representation of affine Hecke algebras induced from the finite Hecke algebra. 
The latter is built by using (finite) R-matrices (see~\cite[Lemma 3.1]{SSV}) which give rise to affine R-matrices via Baxterization.

The operators $\mathbb{T}_i$ that are studied in this paper 
and~\cite{ChintaGunnellsPuskas,PatnaikPuskasIwahori} using non-archimedean
methods can be seen to agree with the metaplectic Demazure-Whittaker
operators of~\cite{SSV} via direct comparison in terms of the Chinta-Gunnells action as in Theorem~\ref{thm:T-average}.
However, it would be desirable to give a conceptual explanation why the construction of~\cite{SSV} (which does not use the $p$-adic group) produce the same operators. Given the appearances of affine R-matrices in both~\cite{SSV} and the current paper, it would be worthwile to understand if their method of constructing \emph{quotients} of induced Hecke representations has an analogue in the non-archimedean world.  
 \end{remark}

\bibliographystyle{habbrv} 
\bibliography{metahori}

\end{document}